\documentclass[12pt,reqno]{amsart}

\usepackage[utf8]{inputenc}
\usepackage{lmodern}
\usepackage{framed} 
\usepackage{color}
\usepackage{xcolor}
\usepackage{bbm}
\usepackage{mathrsfs}
\usepackage{enumitem}
\usepackage{amsmath}
\usepackage{amssymb}
\usepackage{amsfonts}
\usepackage{latexsym}
\usepackage{amsthm}
\usepackage[english]{babel}
\usepackage[T1]{fontenc}
\usepackage{caption}
\usepackage{pgf,tikz}
\usetikzlibrary{arrows}
\usetikzlibrary{shapes}
\usetikzlibrary{patterns}
\usetikzlibrary{shapes.geometric}
\usepackage{geometry}
\usepackage{graphicx}
\usepackage{fancyhdr}
\usepackage{faktor}
\usepackage{mdframed}
\usepackage{ytableau}

\usetikzlibrary{matrix,automata,topaths,decorations.pathreplacing}

\definecolor{Theme}{gray}{0}

\makeatletter

\renewenvironment{proof}[1][\proofname]{\par
   \pushQED{\begin{center}\textcolor{Theme}{\ensuremath{\blacksquare}}\end{center}}%
  \normalfont \topsep6\p@\@plus6\p@\relax
  \trivlist
  \item[\hskip\labelsep
        \bfseries
    #1\@addpunct{.}]\ignorespaces
}{%
  \popQED\endtrivlist\@endpefalse
}

\newtheoremstyle{Thm}  
  {10pt}   
  {20pt}   
  {\itshape}  
  {0pt}       
  {\bfseries} 
  {}         
  {\newline}  
  {\textcolor{Theme}{\thmname{ #1}\thmnumber{ #2}}\textbf{\thmnote{ (#3)}}} 
  
\newtheoremstyle{Def}  
  {10pt}   
  {20pt}   
  {}  
  {0pt}       
  {\bfseries} 
  {}         
  {\newline}  
  {\textcolor{Theme}{\thmname{ #1} \thmnumber{#2}}\textbf{\thmnote{ (#3)}}} 

\theoremstyle{Thm} \newtheorem{ThmIntro}{Theorem}
\theoremstyle{Thm} \newtheorem{Lemma}{Lemma}[section]
\theoremstyle{Thm} \newtheorem{Prop}{Proposition}[section]
\theoremstyle{Thm} 
\theoremstyle{Thm} \newtheorem{CorL}{Corollary}[section]
\theoremstyle{Thm} \newtheorem{CorT}{Corollary}
\theoremstyle{Def} \newtheorem{Def}{Definition}[section]
\theoremstyle{Def} 
\theoremstyle{Def} 
\theoremstyle{Def} \newtheorem*{Rem}{Remark}
\theoremstyle{Def} 

\let\<\langle
\let\>\rangle
\newcommand{\R}{\mathbbm{R}}

\newcommand{\C}{\mathbbm{C}}
\newcommand{\N}{\mathbbm{N}}

\newcommand{\Z}{\mathbbm{Z}}

\newcommand{\D}{\mathbbm{D}}
\newcommand{\1}{\mathbbm{1}}

\newcommand{\wt}{\widetilde}
\newcommand{\Hom}{\mathop{\rm{Hom}}}

\newcommand{\SO}{\mathop{\rm{SO}}}
\newcommand{\U}{\mathrm{U}}
\newcommand{\Super}{\mathrm{S}}
\newcommand{\Ss}{\mathrm{S}}
\newcommand{\SU}{\mathop{\rm{SU}}}
\newcommand{\Sp}{\mathop{\rm{Sp}}}
\newcommand{\Spin}{\mathop{\rm{Spin}}}

\newcommand{\F}{\mathrm{F}}
\newcommand{\SOe}{\mathop{\rm{SO_e}}}

\newcommand{\Res}{\mathop{\rm{Res}}}

\newcommand{\Tr}{\mathop{\rm{Tr}}}
\newcommand{\End}{\mathop{\rm{End}}}
\newcommand{\Id}{\mathop{\rm{Id}}}
\newcommand{\Ind}{\mathop{\rm{Ind}}}
\newcommand{\WF}{\mathop{\rm{WF}}}
\newcommand{\triv}{\mathrm{triv}}

\newcommand{\supp}{\mathop{\rm{supp}}}
\newcommand{\p}{\mathfrak{p}}
\newcommand{\g}{\mathfrak{g}}
\newcommand{\h}{\mathfrak{h}}
\newcommand{\m}{\mathfrak{m}}
\newcommand{\n}{\mathfrak{n}}
\newcommand{\Ak}{\mathfrak{k}}
\newcommand{\Aa}{\mathfrak{a}}

\newcommand{\At}{\mathfrak{t}}
\newcommand{\sL}{\mathfrak{sl}}
\newcommand{\Algt}{\mathfrak{t}}

\newcommand{\su}{\mathfrak{su}}
\newcommand{\soe}{\mathfrak{so}}
\newcommand{\Asp}{\mathfrak{sp}}

\newcommand{\E}{\mathscr{E}}
\newcommand{\Ad}{\mathsf{Ad}}

\newcommand{\Hi}{\mathscr{H}}

\newcommand{\IWk}{{\bf\textrm{k}}}
\newcommand{\IWH}{{\bf\textrm{H}}}
\newcommand{\IWn}{{\bf\textrm{n}}}

\newcommand{\fonction}[5]{\begin{array}{lrcl}
#1: & #2 & \longrightarrow & #3 \\
    & #4 & \longmapsto & #5 \end{array}}
\newcommand{\fracobl}[2]{\raisebox{0.8ex}{$#1$} \bigg/ \raisebox{-0.7ex}{$#2$}}

\renewcommand{\textbf}[1]{\begingroup\bfseries\mathversion{bold}#1\endgroup}

\newcommand*\circled[1]{\tikz[baseline=(char.base)]{
            \node[shape=circle,draw,inner sep=2pt] (char) {#1};}}

\makeatother

\title[Resonances on homogeneous vector bundles]{Resonances of the Laplace operator on homogeneous vector bundles on symmetric spaces of real rank-one}
\author[Resonances on homogeneous vector bundles]{Simon ROBY}
\address{Institut Elie Cartan de Lorraine, IECL (UMR CNRS 7502), 57070 Metz, FRANCE}
\email{simon.roby@univ-lorraine.fr}
\date{}
\subjclass[2010]{Primary: 43A85; secondary: 22E30,58J50}

\begin{document}

\maketitle

\tableofcontents

\begin{abstract}
We study the resonances of the Laplacian acting on the compactly supported sections of a homogeneous vector bundle over a Riemannian symmetric space of the non-compact type. The symmetric space is assumed to have rank-one but the irreducible representation $\tau$ of $K$ defining the vector bundle is arbitrary. We determine the resonances.
Under the additional assumption that $\tau$ occurs in the spherical principal series, we determine the resonance representations. They are all irreducible. We find their Langlands parameters, their wave front sets and determine which of them are unitarizable.
\end{abstract}

\section*{Introduction}


Resonances are spectral objects attached to differential operators acting on non-compact domains and appear as poles of the meromorphic continuation of the resolvent of these operators. 
Their study evolved from an investigation of the Schrödinger operators on the Euclidean spaces like $\R^n$, to a study of the Laplacian on curved spaces, like hyperbolic or asymptotically hyperbolic manifolds, symmetric or locally symmetric spaces.
In a typical setting, one works on a complete Riemannian manifold $X$ with a finite geometry, for which the positive Laplacian $\Delta$ is an essentially self-adjoint operator on the Hilbert space $L^2(X)$ of square integrable functions on $X$. We suppose that $\Delta$ has a continuous spectrum $[\rho_X,+\infty[$, with $\rho_X \geq 0$. The spectrum of $\Delta$ might have some discrete parts, but these parts do not play any significant role, so we neglect them. Also, for simplicity, we assume to have shifted the Laplacian so that its spectrum has bottom at $0$, and have changed variables $z \mapsto z^2$ for the resolvent so that, as above, the resolvent is analytic away from the real axis. The resolvent $R(z) = (\Delta - \rho_X - z^2)^{-1}$ of the shifted Laplacian $\Delta - \rho_X$ is then a holomorphic function of $z$ on the upper (and on the lower) complex half plane. For each such $z$, $R(z)$ is a bounded linear operator from $L^2(X)$ to itself. As such, it cannot be extended across the real axis. However, let us restrict the resolvent to the dense subspace $C^\infty_c(X)$ of compactly supported smooth functions on $X$. Then the map $z \mapsto R(z)$ might admit a meromorphic extension across the real axis to a larger domain in $\C$ or to a cover of such a domain. The poles, if they exist, are the resonances, also called quantum resonances or scattering poles, of $\Delta$.

The basic questions concern the existence of the meromorphic extension of the resolvent, the distribution and counting properties of the resonances, the rank and interpretation of the so-called residue operators associated with the resonances. 
Resonances are linked to interesting geometric, dynamical, and analytic objects.  As a consequence, they are intensively studied, in many different settings, using a variety of techniques and different viewpoints. Standard references for the introduction of resonances are \cite{Agmon1, Agmon2}. A recent overview, also containing an extensive list of references, is \cite{Zwo17}. Riemannian symmetric spaces of the non-compact type are important geometrical settings to study resonances of the Laplacian. Besides being intrinsically interesting objects, they play the role of model spaces to understand phenomena on more complicated or less regular geometries. 

Let us introduce some notations.
A Riemannian symmetric space of the non-compact type is a homogeneous space of the form $X = G/K$ where $G$ is a connected non-compact real semisimple Lie group with finite center and $K$ is a maximal compact subgroup of $G$.
The basic examples are the $n$-dimensional real hyperbolic spaces $H^n$. In this case, the Lie group $G$ is the Lorentz group $\SOe(n,1)$ and $K = \SO(n)$. A Riemannian symmetric space of the non-compact type $X$ has maximal flat subspaces, all of the same dimension, called the (real) rank of $X$. For instance, the rank of $H^{n}$ is 1. Since $X$ is a symmetric space of the Lie group $G$, all natural operators acting on $X$, like the Laplacian and its resolvent, are $G$-invariant. They can therefore be studied using the representation theory of $G$.
The analytic study of the resolvent of the Laplacian acting on functions on $H^{n}$, in particular its meromorphic continuation across its spectrum, is classical and well-understood. It plays a central role when studying the resolvent on more general complete Riemannian manifolds for which the hyperbolic spaces are models.
Still in the case of functions on a more general Riemannian symmetric space $X = G/K$, the study of resonances can in principle be done using an adapted harmonic analysis, the so called Helgason-Fourier analysis, which provides a diagonalization of the Laplacian and hence an explicit formula for its resolvent as a singular integral operator over the spectrum. This formula allowed Hilgert and Pasquale \cite{HilgPasq} to determine and study the resonances for an arbitrary $X$ of rank one. The general higher-rank case is still open. Relevant works in this context are \cite{MaVa05} and \cite{Stro05}.
Complete answers to the basic problems concerning the existence and location of the resonances of the Laplacian, as well as the representation-theoretic interpretation of the so-called residue operators at the resonances, are available only for (most of the) Riemannian symmetric spaces of rank 2. These results appeared in joint articles by Hilgert, Pasquale and Przebinda; see \cite{HiPaPz,HiPaPz2,HiPaPz3}.

All the articles mentioned above consider the Laplacian acting on scalar functions on $X$. A more general question is to consider the Laplacian acting on sections of a homogeneous vector bundle on $X$. Such a bundle is determined by a finite-dimensional representation $\tau$ of $K$. Let us denote this bundle by $E_\tau$. The sections of $E_\tau$ can be seen as vector-valued functions on $G$, with values in the space of the representation $\tau$, such that 
\begin{equation*}
    f(xk) = \tau(k^{-1}) f(x) ~~~\text{ for all } x\in G \text{ and }  k\in K\,.
\end{equation*} 
The space of such functions which are smooth and compactly supported is denoted by $C_c^\infty(G,\tau)$. This means that we are replacing complex-valued functions on $X$ with vector-valued functions which have specific transformation properties on the orbits of the compact subgroup $K$. Examples of sections of homogeneous vector bundles on $X$ are the differential forms, the vector fields, and more generally the tensor fields on $X$: all these objects naturally arise in physical models, and it is therefore natural to look for resonances of the Laplacian in these settings.

The resolvent of the Laplacian of forms on a rank-one Riemannian symmetric space of the non-compact type has been studied by several authors; see \cite{Camp1, Camp2, Camp3, Camp4, PedTh, PedCar, BunkeOlbrich}. In partciular, \cite{PedCar} gives (for the differential forms on rank 1) the list of resonances and the Riemann surface on which the resolvent admits meromorphic extension. Nevertheless, to our knowledge, there is only one article studying the resonances and the residue operators of the Laplacian acting on the sections of a homogeneous vector bundle over $X$, namely \cite{Will}, where $X$ is a complex hyperbolic space and the fibers have dimension one. 



The goal of the present paper is to study the resonances of the Laplacian acting on the sections of a homogeneous vector bundle over a Riemannian symmetric space of the non-compact type $X = G/K$. The symmetric space is assumed to have rank one but the representation of $K$ is arbitrary. Since every finite-dimensional representation of $K$ decomposes into irreducibles, we restrict our attention to irreducible representations. 
As in the case of functions, the basic problems are to determine the existence, the localisation of the resonances and study the residue operators associated with them. 

This paper is organized as follows. In section \ref{generalnotations}, we introduce the notation and recall some basic facts about the structure of Riemannian symmetric spaces of the non-compact type and real rank one. There are four cases, listed in the following table: 
\begin{center}
 \begin{tabular}{|c c c|} 
 \hline
 $G$ & $K$ & $X=G/K$  \\
 \hline
 $\Spin(n,1)$ &$\Spin(n)$ & real hyperbolic space   \\  
 \hline
$\SU(n,1)$ &$\Ss(\U(n) \times \U(1))$ & complex hyperbolic space  \\ 
 \hline
$\Sp(n,1)$ &$\Sp(n)$ & quaternionic hyperbolic space  \\
 \hline
 $\F_4$ &$\Spin(9)$ & octonion hyperbolic space  \\
 \hline
\end{tabular}
\end{center}
We set $\Aa$ to be a maximal flat subspace in $\p = T_{eK}(X)$ the tangeant space of X at the base point $eK$, and $M$ the centralizer of $\Aa$ in $K$. In section \ref{section H-F trsfrm / sph f}, we recall some facts on the generalisation of the Helgason-Fourier transform to homogeneous vector bundles. They are principally due to Camporesi \cite{Camp1}. In particular, the Plancherel Theorem for $L^2$ sections of the homogeneous vector bundles is given there. Denote the decomposition of $\tau$ over $M$ as follows: 
\begin{equation}
    \tau = \bigoplus_{\sigma \in\hat{M}(\tau)} d_\sigma \sigma
\end{equation}
where $\hat{M}(\tau)$ is the set of irreducible unitary representations of $M$ which occurs in the restriction of $\tau$ to $M$, $d_\sigma$ is the degree of $\sigma$. We need some properties of the generalised spherical functions $\varphi_\tau^{\sigma,\lambda}$ associated with the irreducible representations $\sigma \in \hat{M}(\tau)$, which is detailed in section \ref{section H-F trsfrm / sph f}. The explicit formula for the Plancherel density $p_\sigma$ corresponding to these $\sigma \in \hat{M}$ is given in Proposition \ref{Propo plancherel density} (see also Appendix \ref{Plancherel densities section}).
Corollary \ref{boundedconvol} proves the convergence of the singular integral operator providing an explicit formula for the resolvent $R$ of the Laplacian using the inversion formula of vector-valued Helgason-Fourier transform. In section \ref{resonancesparagraph}, we compute the resonances, which is the first main goal of this paper. The holomorphic function $z \mapsto R(z)$ is meromorphically extended from the complex upper half-plane to the whole space, using the residue theorem. The extended resolvent is a meromorphic function with simple poles on the imaginary axis: these poles are the resonances. This leads to our first theorem: 
\begin{ThmIntro}\label{Thmintro1}
Let $G$ be a connected non-compact semisimple Lie group with finite center and with Iwasawa decomposition  $G =KAN$, where $K$ be a fixed maximal compact subgroup of $G$. Suppose $\dim A = 1$. Let $M$ denote the centralizer of $A$ in $K$.
Let $(\tau, \Hi_\tau)$ be an irreducible unitary representation of $K$, and let $E_\tau$ be the homogeneous vector bundle over $G$ associated with $\tau$.
For each $\sigma \in \hat{M}(\tau)$, let $\N_\sigma$ be the set of $k \in \Z$ such that 
\begin{equation*}
    \lambda_k := -i(B_{\max}+k)
\end{equation*}
is a pole of the Plancherel density (see \eqref{Plancherelformula eq} for the formula) and $B_{\max}+k \geq 0$. Here $B_{\max}$ is a nonnegative constant depending only on $G$ and $\sigma$. We refer to \eqref{eq definition of B_j}, \eqref{eq definition of B_j F4} and \eqref{eq Bmax} for the precise definition. 

In this setting, the meromorphic continuation of the resolvent $R$ of the Laplace operator acting on the smooth compactly supported sections of $E_\tau$
can be written as the sum
\begin{equation}
\label{Rtau}
    R=\sum_{\sigma\in \hat{M}(\tau)} d_\sigma R_\sigma\,,
\end{equation}
 where $R_\sigma$ is given for all $f\in C_c^\infty(G,\tau)$ and for all $N \in \N$ by the following formula: 
\begin{multline}\label{meromorphiccontinuationIntro}
\left(R_\sigma(\zeta_{\sigma}) f \right)(x) = \frac{1}{|\alpha|}\int_{\R-i(N+1/4)}\frac{1}{\zeta_\sigma-\lambda|\alpha|} \Big(\varphi_\tau^{\sigma,\lambda\alpha} \ast f  \Big)(x) ~ \frac{p_\sigma(\lambda\alpha)}{\lambda} ~ d\lambda \\
+~~~
\frac{2i\pi}{|\alpha|}\sum_{\substack{k\in \N_\sigma \\ \lambda_k > -i(N+1/4)}} \left(\zeta_\sigma-\lambda_k|\alpha|\right) \Big(\varphi_\tau^{\sigma,\lambda_k\alpha} \ast f  \Big)(x) ~ \Res_{\lambda=\lambda_k}\frac{p_\sigma(\lambda\alpha)}{\lambda}
\end{multline}
In $\eqref{meromorphiccontinuationIntro}$, $\alpha$ is the longest restricted root, $\varphi_\tau^{\lambda,\sigma}$ is a spherical function of type $\sigma$ and 
\begin{equation}\label{zetasigma}
    \zeta_\sigma := \sqrt{-z-\langle\rho,\rho\rangle + \langle\mu_\sigma +\rho_M,\mu_\sigma+\rho_M\rangle)}
\end{equation}
with $z\in \C$ such that $\Im(\zeta_\sigma) > -(N+1/4)$. Here $\sqrt{\cdot}$ denotes the single-valued branch of the square root function determined on $\C\setminus[0,+\infty[$ by the condition $\sqrt{-1} = -i$.

The resonances of the Laplace operator acting on the sections of $E_\tau$ appear in families parametrized by the elements of $\hat{M}(\tau)$. The family corresponding to such a representation $\sigma$ consists of the complex numbers 
\begin{equation}
    z^\sigma_k = (B_{\max}+k)^2|\alpha|^2 -\rho_\alpha^2|\alpha|^2 + \langle\mu_\sigma +\rho_M,\mu_\sigma+\rho_M\rangle
\end{equation}
where $\mu_\sigma$ is the highest weight of the representation $\sigma$, the numbers $k$ are in  $\N_\sigma$ and $\rho_M$ is the half sum of roots for $M$.
\end{ThmIntro}
We refer to section \ref{generalnotations} for the definitions of the various objects appearing in this Theorem and to section \ref{resonancesparagraph} for its proof. 

The second problem we address in this article is the representation theoretic interpretation of the resonances. More precisely, consider the residual part of the meromorphic continuation of the resolvent in \eqref{meromorphiccontinuationIntro}. For each pole 
$\lambda_k$ of the Plancherel density for $\sigma\in\hat{M}(\tau)$, one can introduce an operator, called the residue operator at $\lambda_k\alpha$, defined as follows:
\begin{equation}\label{residueoperatordefinitionintro}
    \fonction{R^\sigma_k}{C_c^\infty(G,\tau)}{C^\infty(G,\tau)}{f}{\varphi_\tau^{\sigma,\lambda_k\alpha} \ast f}
\end{equation} 
Since the convolution product is on the left, it seems that it does not commute with the left translation of $f$. But it does, as we shall see in \eqref{eq def convolution product}. As $G$ acts on the image of $R^\sigma_k$ by the left translations, we get a representation of $G$, called the residue representation at $\lambda_k \alpha$.

In section \ref{section residue repr} we restrict our attention to the representations $\tau$ which contains the trivial representation of $M$. In this case the structure of the principal series is well known \cite{HoweTan,JohnWallPrincseries,JohnF4}. The complexity of the general case (see \cite{Collin1}) is formidable and might lead to much less pleasing results, thus we avoid it. We consider the family of resonances corresponding to $\sigma=\triv$. To simplify the notation, we write $R_k$ instead of $R^\triv_k$. Let $\E_k$ be the residue representation at $\lambda_k\alpha$. We show that the $\E_k$'s are irreducible and equivalent to a subquotient of a spherical principal series representation of $G$. 
We determine which of them are unitary and which are finite-dimensional. Also, we identify their Langlands parameters and compute their wave front sets. 
The Langlands parameters are of the form $(MA,\delta,\lambda)$ and denotes the induced representation $\Ind^G_{MAN}(\delta\otimes e^{i\lambda}\otimes \triv)$ for a nilradical $N$. 
A lowest $K$-type of the induced representation with highest weight $\mu_{\min{}}$ identifies a unique irreducible subquotient of that induced representation (See \cite{VoganProcAcad}). 
The wave front set of a representation has been introduced by Howe (see \cite{HoweWFS}). When $G$ is semisimple, it is a closed set consisting of nilpotents orbits. For $\E_k$ it turns out to be the closure of a single nilpotent orbit. 
In the following theorem, $\alpha$ is the longest restricted root as in Theorem \ref{Thmintro1} and for each restricted root $\beta$, the corresponding root space in $\g$ is denoted by $\g_\beta$. The minimal $K$-type of $\E_k$ is given in the proof the theorem in each case: they can be found in tables \ref{Table Ktypes of E_k - Rcase}, \ref{Table Ktypes of E_k - Ccase}, \ref{Table Ktypes of E_k - Qcase} and \ref{Table Ktypes of E_k - Ocase} respectively for the real, complex, quaternionic and octonionic hyperbolic spaces.

\begin{ThmIntro}\label{Thmintro2}
Suppose that the representation $\tau$ contains the trivial representation of $M$. The residue representations $\E_k$ are then irreducible. 
\begin{enumerate}
    \item If $G = \Spin(2n,1)$, then $\tau$ has highest weight of the form $(N, 0, \ldots, 0)$, where $N$ is a nonnegative integer.
        \begin{itemize}
            \item If $N\geq k+1$, then $\E_k$ has Langlands parameters \mbox{$\Big(MA,\Hi^{k+1}(\R^{2n-1}), \left(n-\frac{3}{2}\right)\alpha\Big)$} with $(k+1,0,\ldots,0)$ as a lowest $K$-type's highest weight. Here $\Hi^{k+1}(\R^{2n-1})$ are harmonics of degree $k+1$ on $\R^{2n-1}$. This representation is unitary. Its wave front set is the nilpotent orbit generated by $\g_\alpha$. 
            \item If $N< k+1$, then $\E_k$ has Langlands parameter \mbox{$\Big(MA,\triv, i\lambda_k\Big)$} with the trivial representation as a lowest $K$-type. It is finite-dimensional. Also, it is non-unitary if $k\not=0$.
        \end{itemize}
    \item If $G = \SU(n,1)$, then $\tau$ has highest weight of the form $(M_1,0, \ldots, 0,-M_2,-L)$, where $M_1$ and $M_2$ are positive integers such that $M_1\geq M_2\geq 0$, $L\in \Z$ and $M_1+M_2+L$ is even.
        \begin{itemize}
            \item If $M_1+M_2 \geq 2k+2$ and $|L| \leq -2k-2+M_1+M_2$, then $\E_k$ is unitary.  
            \begin{itemize}
                \item if $n>2$, then $\E_k$ has minimal $K$-type of highest weight $\big((k+1),0, \ldots,0,-(k+1),0)$. Its Langlands parameter is $\Big(MA,\delta, \left(\frac{n}{2}-1\right)\alpha\Big)$ where the highest weight of $\delta$ is $\big((k+1), 0,\ldots,0,-(k+1),0\big)$. Its wave front set is the nilpotent orbit generated by $\g_{\alpha/2}$. 
                \item if $n=2$, this representation is the discrete series with Blattner parameter \newline\mbox{$\left((k+1),-(k+1),0,\ldots,0\right)$}. Its wave front set is the nilpotent orbit generated by $\g_{\alpha/2}$.
            \end{itemize}
    \item If $L\geq |-2k-1+M_1+M_2| +1$, then $\E_k$ is the representation with Langlands parameter $\Big(MA,\delta, \left(\frac{k}{2}+\frac{n}{2}-\frac{1}{2}\right)\alpha\Big)$, where the highest weight of $\delta$ is \newline\mbox{$\big(0, 0,\ldots,0,-(k+1),(k+1)/2\big)$}. This representation is non-unitary. Its wave front set is the nilpotent orbit generated by the element $n_2$ of $\g_\alpha$ (see Lemma \ref{definition of n1 and n2} for the definition). 
    \item If $L\leq |-2k-1+M_1+M_2|-1$, then $\E_k$ is the representation with Langlands parameter $\Big(MA,\delta, \left(\frac{k}{2}+\frac{n}{2}-\frac{1}{2}\right)\alpha\Big)$, where the highest weight of $\delta$ is \newline\mbox{$\big((k+1), 0,\ldots,0,(k+1)/2\big)$}. This representation is non-unitary. Its wave front set is the nilpotent orbit generated by the element $n_1$ of $\g_\alpha$ (see Lemma \ref{definition of n1 and n2} for the definition). 
    \item If $M_1+M_2 \in [0,2k+2[$ and $L< |2k+2-M_1+M_2 |$, then $\E_k$ is the representation with Langlands parameter $\Big(MA,\triv, i\lambda_k\alpha\Big)$. It is finite-dimensional and non unitary (if $k\not=0$).
\end{itemize}
\item If $G = \Sp(n,1)$, then $\tau$ has a highest weight of the form $(t_1,t_2, 0,\ldots, 0,t_{n+1})$, where $t_1$, $t_2$ and $t_{n+1}$ are positive integers such that $t_1\geq t_2$, $t_{n+1}\leq t_1+t_2$ and $t_1+t_2+t_{n+1}$ is even.
\begin{itemize}
    \item If $t_{n+1}\leq t_1+t_2-2k-4$ , then $\E_k$ is the representation with Langlands parameter $\Big(MA,\delta, \pm\left(n-\frac{3}{2}\right)\alpha\Big)$ with $\tau$ as a lowest $K$-type, where the highest weight of $\delta$ is $(k+2,k+2, 0,\ldots,0)$. This representation is non-unitary. Its wave front set is the nilpotent orbit generated by $\g_{\alpha/2}$.
    \item If $t_{n+1}\geq |t_1+t_2-2k-2|$ the residue representation is unitary. Its wave front set is the nilpotent orbit generated by $\g_{\alpha}$.
    \begin{itemize}
        \item If $k\leq 2n-4$, then $\E_k$ is the representation with Langlands parameter \newline\mbox{$\Big(MA,\delta,\pm \left(n-\frac{k}{2}\right) \alpha\Big)$} with lowest $K$-type $(k+1,0,\ldots,0,k+1)$, where the highest weight of $\delta$ is $(k+1, 0,\ldots,0, \frac{k+1}{2})$. 
        \item If $k\geq 2n-3$, then $\E_k$ is the discrete series representation with Blattner parameter $\mu_k = (k+1,0,\ldots,0,k+1)$. 
    \end{itemize}
    \item If $t_1+t_2< 2k+2-t_1-t_2$, then $\E_k$ is the representation with Langlands parameter $\Big(MA,\triv, i\lambda_k\alpha\Big)$. It is finite-dimensional and non unitary (if $k\not=0$).
\end{itemize}
\item If $G = F_4$, then $\tau$ has a highest weight of the form $(a/2,b/2,b/2,b/2)$, where $a$ and $b$ are positive integers such that $a\geq b$ and $a-b$ is even.
\begin{itemize}
    \item If $b\leq a-2k-8$ , then $\E_k$ is the representation with Langlands parameter \newline$\Big(MA,\delta, \frac{1}{2}\left(k+1\right)\alpha\Big)$, where the highest of $\delta$ is $\frac{k+4}{4}(3,1,1,1)$. This representation is non-unitary. Its wave front set is the nilpotent orbit generated by $\g_{\alpha/2}$.
    \item If $b> a-2k-8$ and $b\geq 2k+2-a$, then $\E_k$ is the representation with Langlands parameter $\Big(MA,\delta,\frac{1}{2}\left(k+10\right)\alpha\Big)$, where the highest of $\delta$ is $\frac{k+1}{4}(3,1,1,1)$. This representation is unitary. Its wave front set is the nilpotent orbit generated by $\g_{\alpha}$.
    \item If $b< 2k+2$, then $\E_k$ is the representation with Langlands parameter \newline\mbox{$\Big(MA,\triv, i\lambda_k\alpha\Big)$}. It is finite-dimensional and non unitary (if $k\not=0$).
\end{itemize}
\end{enumerate}
\end{ThmIntro}

A nice consequence of our case-by-case results is the following. 

\begin{CorT}
For a fixed $k \in \N$, there is one-to-one correspondence between the irreducible subquotients of $\Hi_{\lambda_k\alpha}$  and the (real) nilpotent orbits of $\g$ under the adjoint action. This correspondence maps each subquotient into the orbit whose closure is the wave front set of that subquotient. 
As we showed, $\E_k$ is equivalent to one of these subquotients. Its wave front set is then the closure of \emph{one} nilpotent orbit in $\g$. 
\end{CorT}

\hspace{2cm}

\textit{Acknowledgement:}
I want to thank first of all my two advisors, Angela Pasquale and Tomasz Przebinda, for being there when I needed them. I would also like to thank the Fulbright Program (IIE grantee ID Number : PS00289128), which supported my stay at the University of Oklahoma. This stay turned out to be very important for reaching the results of this paper and was interesting at a personal level.

\section{General notations}

\label{generalnotations}

We shall use the standard notations $\N,~ \Z,~\R, ~\C$ and $\C^\times$ for the nonnegative integers, the integers, the real numbers, the complex numbers and the nonzero complex numbers. For a complex number $z \in \C$, we denote by $\Re(z)$ and $\Im(z)$ its real and imaginary parts. The positive constants in the Haar measures do not matter in our computations and equalities. Integrals have to be considered up to positive multiples. 

\underline{Context}: Let $G$ be a connected non-compact real semisimple Lie group with finite center and let $B(\cdot,\cdot)$ be the Killing form on the Lie algebra $\g$ of $G$. We denote by $\theta$ a Cartan involution on $\g$.
We denote by $\Ak$ the set of fixed points of $\theta$ and by $\p$ the eigenspace of $\theta$ for the eigenvalue $-1$. In other words: 
$$\Ak = \{X\in\g ~|~ \theta X = X\} ~~~\text{ and }~~~ \p = \{X\in\g ~|~ \theta X = -X\}~.$$
Then $\Ak$ is a Lie subalgebra of $\g$. The corresponding connected Lie subgroup of $G$ is maximal compact. We denote it by $K$. 
The Cartan decomposition of the Lie algebra $\g$ is given by: $\g = \Ak \oplus \p$. 

Let $\Aa$ be a maximal abelian subspace of $\p$ and $A = \exp \Aa$ its associated subgroup of $G$. The exponential map $\exp:\g \to G$ restricts to a diffeomorphism between $\Aa$ and $A$. The inverse map is the logarithm "$\log$".

\underline{Roots and restricted roots systems}: Let $\Aa^\ast$ be the vector space of linear forms on $\Aa$ and $\Aa_\C^\ast$ its complexification. The set $\Sigma$ of restricted roots of the pair $(\g,\Aa)$ consists of all linear forms $\alpha \in \Aa^\ast$ for which the vector space $$\g_\alpha:= \{X \in \g ~|~ [H,X] = \alpha(H)X \text{ , for every }H \in \Aa\}$$ contains nonzero elements. The dimension of $\g_\alpha$ is called the multiplicity of the root $\alpha$ and is denoted by $m_\alpha$. 

Let $\Sigma_+$ be a fixed set of positive restricted roots and let $\rho:= \displaystyle\frac{1}{2}\sum_{\alpha \in \Sigma^+} m_\alpha \alpha $ be the half sum of the positive roots counted with their multiplicities.
Set $\n = \displaystyle\bigoplus_{\alpha \in \Sigma_+}\g_\alpha$ and $N$ the connected Lie subgroup of $G$ having $\n$ for Lie algebra.
According to the Iwasawa decomposition $G = KAN$, every element $x$ in $G$ can be uniquely written as \begin{equation}\label{eq Iwasawa}
   x = \IWk(x) e^{\IWH(x)} \IWn(x) 
\end{equation}
where $\IWk(x) \in K$, $\IWH(x) \in \Aa$ and $\IWn(x) = \IWn \in N$. In the following, we set
\begin{equation}
    a^\lambda := \exp(\lambda(\log a))~ \text{ for } a\in A \text{ and }\lambda \in \Aa^*_\C~.
\end{equation}
Let $M$ be the centralizer of $\Aa$ in $K$, $\m$ its Lie algebra and let $\At$ be a Cartan subalgebra of $\m$. Then the Lie algebra $\h = \At\oplus\Aa$ is a Cartan subalgebra of $\g$. We denote by $\h_\C$ its complexification. 
The set $\Pi$ of roots of the pair $(\g_\C,\h_\C)$ consists of all linear forms $\varepsilon \in \h_\C^\ast$ for which the vector space
$$\g_\varepsilon:= \{X \in \g_\C ~|~ [H,X] = \varepsilon(H)X \text{ , for every }H \in \h_\C\}$$
contains nonzero elements.
\newline We choose a set $\Pi_+$ of positive roots in $\Pi$ which is compatible with $\Sigma_+$, i.e. such that a root $\varepsilon\in \Pi$ is positive when $\varepsilon|_\Aa \in \Sigma_+$. 
Let also $\Pi_\Ak$ ($\Pi_{\Ak~+}$) be the set of (positiv) roots of the pair $(\Ak_\C, \h_\C)$. 

\underline{The rank-one case}: In this paper, we are restricting ourself to real rank-one groups $G$. In other words, we suppose that $\Aa$ is one-dimensional. 

Rank-one symmetric spaces of the non-compact type are classified into three infinite families -- namely, the real, complex and quaternionic hyperbolic spaces -- and one exceptional example, the octonionic hyperbolic plane. 

Since $G$ is of real rank one, the set $\Sigma$ is either equal to $\{\pm \alpha \}$ or $\{\pm \alpha,\pm \alpha/2 \}$. Among the groups listed in the table in the introduction, only $G = \Spin(n,1)$ has restricted root system $\{\pm \alpha \}$. As a system of positive roots $\Sigma_+$ we choose $\{\alpha \}$ and $\{\alpha,\alpha/2\}$.
Then \mbox{$\rho= \frac{1}{2}\big(m_\alpha + \frac{m_{\alpha/2}}{2} \big) \alpha$}, where we set $m_{\alpha/2} = 0$, if $\Sigma = \{\pm \alpha \}$.

The Killing form $B$ is positive definite on $\p$, so $\<X,Y\> :=B(X,Y)$ defines a Euclidean structure on $\p$ and on $\Aa \subset\p$. For all $\lambda\in\Aa^*$, let $H_\lambda$ denote the unique element in $\Aa$ such that $\<H_\lambda,H\> =\lambda(H)$ for all $H \in \Aa$. We extend the inner product to $\Aa^\ast$ by setting $\<\lambda,\mu\> :=\<H_\lambda,H_\mu\>$ for all $\lambda, ~\mu\in\Aa^*$. Further, we denote the $\C$-bilinear extension of $\<\cdot,\cdot\>$ on $\Aa$ to $\Aa_\C^*$ by the same symbol. 
We identify $\Aa_\C^*$ to $\C$ by means of the isomorphism: 

\begin{equation}
    \label{lambda}\begin{array}{ccc}
    \Aa_\C^\ast & \longrightarrow & \C \\
    \lambda & \longmapsto & \lambda_\alpha:= \frac{\<\lambda,\alpha\>}{\<\alpha,\alpha\>}
\end{array}
\end{equation}
which identifies $\rho$ with $\rho_\alpha:=\frac{1}{2}\left(m_\alpha+\frac{m_{\alpha/2}}{2}\right)$.

\underline{Homogeneous vector bundles}: We fix a  finite-dimensional unitary representation $(\tau, \Hi_\tau)$ of $K$. 
Let $E_\tau:= X \times_\tau \Hi_\tau$ denote the homogeneous vector bundle over $X$. For the definition and properties of $E_\tau$, we refer the reader to \cite[\S 5.2 p. 114]{Wall1}. 
We write $\Gamma^\infty(E_\tau)$ for the space of all smooth sections of $E_\tau$. As proved in \cite[\S 5.4 p. 119]{Wall1}, there is an isomorphism between $\Gamma^\infty (E_\tau)$ and 
\begin{equation*}
    C^\infty(G, \tau): = \{f:G \rightarrow \Hi_\tau \text{ smooth}~| ~f(xk) = \tau(k^{-1}) f(x)~\text{ for all } x \in G \text{ and } k \in K \}
\end{equation*}
Set  
{\small\begin{equation*}
    C^\infty(G,K,\tau,\tau):=\{ F:G \rightarrow \End(\Hi_\tau) C^\infty ~|~F(k_1xk_2) = \tau(k_2^{-1}) F(x) \tau(k_1^{-1})~\text{ for all }x\in G\text{ and }k_1,k_2 \in K\}
\end{equation*}}
The elements $F\in C^\infty(G,K,\tau,\tau)$ are sometimes called the radial systems of sections of $E_\tau$.
The link with the sections comes by the fact that for every $v\in \Hi_\tau$, the function $F(\cdot)v$ is a smooth section of $E_\tau$.  

If $\tau$ is not irreducible then $E_\tau = \displaystyle\bigoplus_i E_{\tau_i}$, where $\tau_i$ are the irreducible components of $\tau$. Studying the sections of $E_\tau$ amounts to studying the sections of each bundle $E_{\tau_i}$. We can therefore suppose without loss of generality that $\tau$ is irreducible. 
\newline We notice that if $\tau$ is the trivial representation "$\triv$" of $K$ on the one-dimensional vector space $\C$, Then the sections of $E_\triv$ are the functions on $G/K$. They can be seen as right-$K$-invariant functions on $G$.
Moreover, in this case, the radial systems of sections agree with the $K$-bi-invariant functions on $G$. We will refer to this situation as the scalar case.

\underline{Principal series representations}: Let $\hat{M}$ be the set of all equivalence classes of irreducible unitary representations of $M$.
For $(\sigma,\Hi_\sigma) \in \hat{M}$ and $\lambda \in \Aa^\ast_\C$, we denote by \newline \mbox{$\pi^\sigma_\lambda=\Ind^G_{MAN}(\sigma \otimes e^{i\lambda} \otimes \triv)$} the induced representation from $MAN$ to $G$ by the representation $\sigma \otimes e^{i\lambda} \otimes \triv$.
We will use the same notation for its derived representation of $\g$ too. The representation space $\Hi^\sigma_\lambda$ 
of $\pi^\sigma_\lambda$ is the Hilbert space completion of 
\begin{equation}\label{defprincseries}
    \{f: G \rightarrow \Hi_\sigma ~|~ f(xman) = a^{-i\lambda-\rho}\sigma(m^{-1})f(x) ~\text{ for all }x \in G,~ m\in M,~ a\in A \text{ and }n\in N  \}
\end{equation}
with respect of the $L^2$ inner product
$$\<f,g\>_{\sigma} = \int_K\<f(k),g(k)\>_{\Hi_\sigma}~dk\,,$$
where $\<\cdot,\cdot\>_{\Hi_\sigma}$ is an inner product on $\Hi_\sigma$ making $\sigma$ unitary.
The action of $\pi^\sigma_\lambda$ on $\Hi^\sigma_\lambda$ is given by 
$$\pi^\sigma_\lambda(g)f(x) := f(g^{-1}x)$$
for all $g,x \in G$ and $f\in \Hi^\sigma_\lambda$. The set $\{\pi^\sigma_\lambda ~|~\lambda\in \Aa_\C^\ast, \sigma \in \hat{M}\}$ is called the minimal principal series of $G$. 

The compact picture of the principal series representations is obtained by restriction of the elements of $\Hi^\sigma_\lambda$ to $K$. Its representation space, which we denote by $\Hi^\sigma$, is the Hilbert completion of: 
\begin{equation*}
    \{f: K \rightarrow \Hi_\sigma ~|~ f(km) = \sigma(m^{-1})f(k) ~\text{ for all }k \in K,~ m\in M \}
\end{equation*}
with respect to $L^2$ inner product. It is independent of $\lambda$. The action is given by:
\begin{equation*}
    \pi^\sigma_\lambda(g)f(k) := e^{-(i\lambda+\rho)\IWH(g^{-1}k)} f(\IWk(g^{-1}k))
\end{equation*} 
for all $g \in G$, $k \in K$ and $f\in \Hi^\sigma$.
The representation $\pi^\sigma_\lambda$ is unitary for $\lambda \in i\Aa^*$. In the following, when working with principal series, we actually work with their Harish-Chandra modules.
The restriction of $\pi^\sigma_\lambda$ to $K$ is the representation $\Ind_M^K\sigma$ of $K$ induced from $\sigma$. In particular, because of Frobenius reciprocity theorem, for any $\tau \in \hat{K}$: \begin{equation*}
    m(\tau,\pi^\sigma_\lambda| _K)=m(\sigma,\tau| _M)\,.
\end{equation*}
Here the symbol $m(\beta,\alpha)$ denotes the multiplicity of the irreducible representation $\beta$ in the representation $\alpha$.
We say that $\tau$ is a $K$-type of $\pi^\sigma_\lambda$ if it occurs in $\pi^\sigma_\lambda| _K$. We say that $\tau$ is a minimal $K$-type of an admissible representation $\pi$ of $G$ if and only if its highest weight $\mu$ minimizes the Vogan norm
$$\lVert \mu\rVert_V = \langle \mu +2\rho_K, \mu +2\rho_K \rangle$$
in the set of $K$-types of $\pi$. Here $2\rho_K$ is the sum of positive roots of the pair $(\Ak_\C, \h_\C | _{\Ak_\C})$. \cite[Theorem 1]{VoganProcAcad} ensures that each minimal $K$-type $\tau_{\min{}}$ has multiplicity  one in $\pi^\sigma_\lambda$. Therefore there exists a unique irreducible subquotient $J(\sigma,\lambda,\mu)$ of $\pi^\sigma_\lambda$ containing $\tau_{\min{}}$.

\underline{Homogeneous differential operators}: A homogeneous differential operator $D$ on $E_\tau$ is a linear differential operator from $\Gamma^\infty(E_\tau)$ to itself which is invariant under the $G$-action by left translations, that is
\begin{equation}
    L(g)D = DL(g)\quad\text{ for all }g \in G~.
\end{equation}
The set of homogeneous differential operators on $E_\tau$ is an algebra with respect to composition. We denote it by $\D(E_\tau)$. 
It acts on $C^\infty(G, \tau)$ because of the isomorphism with the space smooth sections $\Gamma^\infty(E_\tau)$. Unlike in the scalar case, i.e. when $\tau$ is the trivial representation, this algebra need not be commutative. Conditions equivalent to the commutativity of $\D(E_\tau)$ are stated in \cite[Proposition 2.2]{Camp4} and \cite[Proposition 3.1]{Ricci}. In the rank one case, this algebra is always commutative when $G$ is $\Spin(n,1)$ or $\SU(n,1)$. See for instance \cite[Theorem 2.3]{Camp4}. The structure of $\D(E_\tau)$ can be found in \cite[Section 2.2]{Olb}.

Let $U(\g_\C)$ be the universal enveloping algebra of the complexification $\g_\C$ of $\g$. Each element of $U(\g_\C)$ induces a left-invariant differential operator on $G$ by:
\begin{equation}
    \big(X_1 \cdots X_k\cdot f\big)(g): = \frac{\partial}{\partial t_1}\frac{\partial}{\partial t_2} \cdots \frac{\partial}{\partial t_k} f (g \exp t_1X_1 \exp t_2X_2 \cdots \exp t_kX_k ) \Big| _{t_1=\ldots= t_k =0}
\end{equation}
for all $X = X_1 \cdots X_n \in U(\g_\C)$, $f \in C^\infty(G)$ and $g\in G$.\newline
Let $U(\g_\C)^K$ denote the subalgebra of the elements in $U(\g_\C)$ which are invariant under the adjoint action $\Ad$ of $K$. The elements of $U(\g_\C)^K$ act on on $C^\infty(G,\tau)$ as homogeneous differential operators. As $K$ is compact, Theorem 1.3 in \cite{Minemura} ensures that all elements of $\D(E_\tau)$ can be written as an element of $U(\g_\C)^K$. But there is no isomorphism in general. \newline
We can extend this action to the set of radial systems of section $C^\infty(G,K,\tau, \tau)$ by setting:
\begin{equation*}
    \big(D\cdot \phi \big) v := D\cdot (\phi \cdot v)
\end{equation*}
for all $D \in U(\g_\C)$, $\phi \in C^\infty(G,K, \tau, \tau)$ and $v\in \Hi_\tau$.\newline
\underline{The Laplace operator}: Let $\{X_1,\ldots, X_{\dim \g}\}$ be any basis of $\g$. We denote by $g^{ij}$ the $ij$-th coefficient of the inverse of the matrix $\big( B(X_i,X_j)\big)_{1\leq i,j\leq \dim \g}$, where $B$ is the Killing form. The Casimir operator is defined by
\begin{equation*}
\Omega := \sum_{1\leq i,j\leq \dim \g} g^{ij}X_jX_i ~.
\end{equation*}
If $\big(X_k\big)_{k=1,\ldots,\dim \Ak}$ and $\big(X_k\big)_{k=\dim \Ak + 1,\ldots,\dim \g}$ are respectively orthonormal basis of $\Ak$ and $\p$ with respect to $B_\theta$, then: 
\begin{equation*}
    \Omega = -\sum_{i=1}^{\dim \Ak} X_i^2 + \sum_{i=\dim \Ak+ 1}^{\dim \g} X_i^2~.
\end{equation*}
In fact, $\Omega$ is in the center of $U(\g_\C)$. The invariant differential operator corresponding $-\Omega$ is the positive Laplacian $\Delta$.

We can extend any representationof $\g$ to $\g_\C$ by linearity and to a representation of the associative algebra $U(\g_\C)$. These representations will always be denoted by the same symbol. 
Since $\Omega$ is in the center of $U(\g_\C)$, the linear operator $\pi^\sigma_\lambda(\Omega)$ is an interwining operator of the representation $\pi^\sigma_\lambda$ for all $\lambda \in \Aa^\ast_\C$ and $\sigma \in \hat{M}$. Lemma 4.1.8 in \cite{VoganRRRLG} ensures that $\pi^\sigma_\lambda(\Omega)$ acts by a scalar. 
To compute this scalar, one can use \cite[Proposition 8.22 and Lemma 12.28]{Kna1}, and get that:
    \begin{equation}\label{Multope}
    \pi^\sigma_\lambda(\Omega) = -\langle\lambda,\lambda\rangle - \langle\rho,\rho\rangle + \langle\mu_\sigma +\rho_\m,\mu_\sigma+\rho_\m\rangle \Id ~.
    \end{equation}
Here $\mu_\sigma$ is the highest weight of $\sigma$ and $\rho_\m$ is the half sum of the positive roots $\varepsilon \in \Pi^+$ such that $\varepsilon| _\Aa = 0$.

\section[Vector-valued Helgason-Fourier transform]{The vector-valued Helgason-Fourier transform and spherical functions of type $\tau$}

\label{section H-F trsfrm / sph f}

In this section we review some basic facts on Camporesi's extension of the Helgason-Fourier transform to homogeneous vector bundles. We refer the reader to \cite{Camp1} for more information. 

We keep the notations of the introduction. In particular, since we suppose that $G$ is of real rank one, in the Plancherel formula only (minimal) principal series representations, and if $G\not= \Spin(2n+1,1)$, discrete series representations occur. 

Let $(\tau, \Hi_\tau)$ be an irreducible unitary representation of $K$. Let 
\begin{equation*}
    \hat{M}(\tau) = \{\sigma \in \hat{M} ~|~ m(\sigma,\tau|_M) \geq 1\}
\end{equation*}
denote the set of unitary irreducible representations of $M$ which occur in the restriction of $\tau$ to $M$. We will denote by $d_\gamma$ the dimension of a representation $\gamma$. 
For $\sigma\in \hat{M}(\tau)$, let $P_\sigma$ be the projection of $\Hi_\tau$ onto the subspace of vectors of $\Hi_\tau$ which transform under $M$ according to $\sigma$. Explicitly, 
\begin{equation}
\label{Psigma}
P_\sigma=d_\sigma \int_M \tau(m^{-1})\chi_\sigma(m)\,dm\,, 
\end{equation}
where $\chi_\sigma$ denotes the character of $\sigma$.
\newline We denote by $p_\sigma(\lambda)$ the Plancherel density associated to the principal series representation $\pi^\sigma_ \lambda$. 
Recall the Iwasawa decomposition \eqref{eq Iwasawa} of $x\in G$. According to \cite[Theorem 1.1]{Camp1}, the vector-valued Helgason-Fourier transform of $f\in C_c^\infty(G,\tau)$ is the function from $\Aa^\ast_\C \times K$ to $\Hi_\tau$ defined by
    \begin{equation}
     \wt f(\lambda,k) = \int_G F^{i\overline{\lambda}-\rho}(x^{-1}k)^\ast f(x) ~dx ~. \end{equation}
Here, for $\mu\in \Aa_\C$ and $x\in G$, \begin{equation}\label{Ffunction}
    F^\mu(x) = e^{\mu(H(x))}\tau(\IWk(x))
\end{equation}
and $ ^\ast$ denotes the Hilbert space adjoint.

In the rank-one case, the inversion formula is
    \begin{multline}\label{plancherelformula}
    f(x) = \frac{1}{d_\tau} \sum_{\sigma \in \hat{M}(\tau)} \int_{\Aa^\ast}\int_K F^{i\lambda-\rho}(x^{-1}k)~P_\sigma~ \wt f(\lambda,k)~p_\sigma(\lambda)~~d\lambda~ dk \\+ \sum_{\gamma \in D_G}C_\gamma \int_K F^{-i\mu-\rho}(x^{-1}k)~P_{\gamma'}~ \wt f(i\mu,k)~ dk
    \end{multline}
Here $\mu \in \Aa^\ast$ and $\gamma' \in \hat{M}$ are chosen so that $\gamma$ is infinitesimally equivalent to a subrepresentation of  \mbox{$\Ind_{MAN}^G(\gamma'\otimes e^\mu\otimes \triv)$}. Moreover, $C_\gamma$ is a suitable constant depending on $\gamma$. 
The set $D_G$ is the set of discrete series of $G$. It consists of all irreducible unitary representations of $G$ such that all its matrix coefficients are in $L^2(G)$. 
\newline The second sum term is called the discrete part of the Plancherel formula. In the following, we will disregard this term. In fact, only the first term, i.e. the continuous part of the Plancherel formula, can contribute to the resonances, by means of the poles of the Plancherel density. 

As a consequence not involving discrete series, Parseval's formula for the continuous spectrum reads as follows: for $f,h \in C_c^\infty (G,\tau)$ 
    \begin{equation}\label{scalprod}
    \langle f,h \rangle_c = \sum_{\sigma \in \hat{M}(\tau)} \frac{1}{d_\sigma} \int_{\Aa^\ast} \int_K  \langle P_\sigma \wt f(\lambda,k) , P_\sigma \wt h(\lambda,k) \rangle~ p _\sigma(\lambda)~d\lambda~d k~,
    \end{equation}
where the index $c$ underlines that we are only considering the contribution from the continuous spectrum. See \cite[p. 286]{Camp1}. On the right-end side of \eqref{scalprod}, $\langle\cdot,\cdot\rangle$ denotes the inner product on $\Hi_\tau$ making $\tau$ unitary. The corresponding norm will be denoted by $\|u\| = \sqrt{\langle u,u\rangle}$. 

We will also need a vector-valued analogue of Harish-Chandra's spherical functions on a non-compact reductive Lie group. These vector-valued functions were introduced by Godement \cite{God1} and Harish-Chandra \cite{HCCol3}. They depend on the fixed representation $\tau$ of $K$ and on a representation of the principal series indexed by $\sigma \in \hat{M}(\tau)$ and $\lambda \in \Aa^\ast_\C$.

Keep the above notation for the principal series. 
Let $P_\tau$ denote the projection of $\Hi^\sigma_\lambda$ onto its subspace of vectors which transform under $K$ according to $\tau$,
that is,
    \begin{equation}
    \label{Ptau}
P_\tau=d_\tau \int_K \pi^\sigma_{\lambda}(k)\chi_\tau(k^{-1})\, dk\,.
    \end{equation}

\begin{Def}
The spherical function $\varphi_\tau^{\sigma,\lambda}$ is defined as the $\End(\Hi_\tau)$-valued function on $G$ given by
    \begin{equation}
\varphi_\tau^{\sigma,\lambda}(x):=\varphi_\tau^{\pi^\sigma_{\lambda}}(x):=d_\tau \int_K \tau(k) \psi_\tau^{\sigma,\lambda}(xk^{-1})\, dk\,,  
    \end{equation}
where 
    \begin{equation}
    \label{psi-tausigma}
\psi_\tau^{\sigma,\lambda}(x)=\Tr\big(P_\tau \pi^\sigma_{\lambda}(x) P_\tau\big)\,.
    \end{equation}
\end{Def}

Let $\Hom_K(\Hi^\sigma_\lambda,\Hi_\tau)$ be the space of $K$-intertwining operators between $\pi^\sigma_\lambda| _K$ and $\tau$. We equip this space with the scalar product $\langle P,Q\rangle = \frac{1}{d_\tau}\Tr(P Q^\ast)$, where $\ast$ denotes the adjoint. We fix an orthonormal basis $\{P_\xi\}_{\xi = 1,\ldots, m(\sigma,\tau|_M)}$ of this space. Then
    \begin{equation}
        \varphi_\tau^{\sigma,\lambda}(g)= \sum\limits^{m(\sigma,\tau|_M)}_{\xi = 1} P_\xi ~\pi^\sigma_\lambda(g)~ P_\xi^\ast\,.
    \end{equation}
See pp. 268--269 and 273 in \cite{Camp1}.

\begin{Lemma}
The spherical functions $\varphi_\tau^{\sigma,\lambda}$ are even functions of $\lambda\in \Aa^\ast$ for all $\sigma$.
\end{Lemma}

\begin{proof}
Due to Lemma 3.1 in \cite{Camp4}, the spherical functions $\varphi_\tau^{\sigma,\lambda}$ are in one-to-one correspondence with their traces. Now $\Tr(\varphi_\tau^{\sigma,\lambda}) = m(\sigma, \tau|_M) ~\chi_\tau \ast \Theta^\sigma_\lambda$, where $\chi_\tau$ and $\Theta^\sigma_\lambda$ are the respective characters of $\tau$ and $\pi^\sigma_\lambda$. Lemma 4, page 162, in \cite{HCHARRG3} gives us that $\Theta^\sigma_\lambda = \Theta^{-\sigma}_{-\lambda}$ because $-1$ is in the Weyl group. As $-1$ acts trivially on $M$ (so on $\sigma$), the lemma follows. 
\end{proof}

The spherical function $\varphi_\tau^{\sigma,\lambda}$ can be described as an Eisenstein integral (see \cite[Lemma 3.2]{Camp1}), 
\begin{equation}
    \varphi_\tau^{\sigma,\lambda}(x) = \frac{d_\tau}{d_\sigma} \int_K \tau(\IWk(xk))~P_\sigma ~\tau(k^{-1}) ~e^{(i\lambda-\rho)(\IWH(xk))}~ dk~.
\end{equation}
Notice that $\varphi^{\sigma,\lambda}_\tau$ satisfies $\varphi^{\sigma,\lambda}_\tau(k_1 x k_2) = \tau(k_1) \varphi^{\sigma,\lambda}_\tau(x) \tau(k_2)$ for every $x\in G$ and $k_1, k_2 \in K$. 
The convolution with a function $f\in C^\infty_c(G,\tau)$ is defined by: \begin{equation}\label{eq def convolution product}
   ( \varphi_\tau^{\sigma,\lambda} \ast f)(x) := \frac{d_\tau}{d_\sigma} \int_G \varphi_\tau^{\sigma,\lambda}(x^{-1}g) f(g) ~dg~.
\end{equation} 
According to \cite[Proposition 3.3]{Camp1}, it can be expressed in terms of the vector-valued Helgason-Fourier transform of $f$:
\begin{equation}\label{sphconv}
   ( \varphi_\tau^{\sigma,\lambda}\ast f)(x) = \frac{d_\tau}{d_\sigma} \int_K F^{i\lambda -\rho}(x^{-1}k) ~P_\sigma ~ \wt f(\lambda,k)~ dk~.
\end{equation}

\begin{Lemma}\label{sphiseigenf}
The spherical functions $\varphi_\tau^{\sigma,\lambda}$ are joint eigenfunctions of the homogeneous differential operators on $E_\tau$. Moreover, for all $z \in Z(\g_\C)$, the center of $U(\g_\C)$, we have 
    \begin{equation}\label{eigenfunctions}
        z \cdot \varphi_\tau^{\sigma,\lambda} = \gamma(z)(i\lambda -\mu_\sigma -\rho_\m) ~\varphi_\tau^{\sigma,\lambda}~.
    \end{equation}
Here $\gamma$ is the Harish-Chandra homomorphism described in \cite[Chapter VIII, paragraph 5]{Kna1} and $\mu_\sigma$ is the highest weight of $\sigma$.
\end{Lemma}
\begin{proof}
In fact, the function $\Psi_\lambda$ defined in \cite{Yang} by 
\begin{equation*}
    \Psi_\lambda(nak) := \tau(k^{-1})a^{\lambda+\rho}, \text{ with } k\in K, ~a\in A,~ n \in N,
\end{equation*}
is nothing but the function $x \mapsto F^{-\lambda-\rho}(x^{-1})$ defined in \eqref{Ffunction}. Hence \cite[Proposition 1.3, Corollary 1.4 and Theorem 1.6]{Yang} allows us to prove \eqref{eigenfunctions}. 
\end{proof}

\begin{Rem}
For the Casimir operator, the eigenvalue is given by  \eqref{Multope} :  \begin{equation}\label{eigenvalue Casimir}
\gamma(\Omega)(i\lambda -\mu_\sigma -\rho_\m) =-\langle\lambda,\lambda\rangle - \langle\rho,\rho\rangle + \langle\mu_\sigma +\rho_\m,\mu_\sigma+\rho_\m\rangle
\end{equation}
\end{Rem}

To compute the resonances of the Laplacian, we need to describe the vector-valued Helgason-Fourier transform of functions on $C_c^\infty(G,\tau)$. 

By the Cartan decomposition, we can uniquely write an element $x\in G$ as $x = k \exp X$, with $k\in K$ and $X \in \p$.
For $X \in \p$, we write $|X| = B(X,X)^{1/2}$. Define the open ball centered at $0$ and of radius $R>0$ in $\Aa \subset \p$ by $$B_\Aa^R = \{X \in \Aa ~\big|~|X| < R\}~.$$ 
Moreover, we denote the geodesic distance between $o=eK$ in $xK$ by $d(o,xK)$. Let $$B_R = \{x \in G ~\big|~d(o,xK) \leq R\}~.$$
We fix an orthonormal basis $\{e_1, \ldots, e_{d_\tau}\}$ of $\Hi_\tau$, then we define  \begin{equation*}
        ||\wt f(\lambda,k)||^2 := \sum_{i=1}^{d_\tau} \langle \wt f(\lambda,k), e_i\rangle^2 ~.
    \end{equation*}
The direct implication of the Paley-Wiener theorem for $C_c^\infty(G,\tau)$ is given by the following Lemma.

\begin{Lemma}\label{PW}
Let $f$ be in $C_c^\infty(G,\tau)$ and $R>0$. If $\supp f\subset B_R$, then $\wt f(\lambda,k)$ satisfying is an entire function of $\lambda \in \Aa^*_\C$ for all $N\in\N$:
\begin{equation}\label{eq PW} \tag{$\ast$}
    \sup\limits_{\lambda\in \Aa_\C^\ast, k\in K} e^{-R ~|\Im (\lambda) |}(1+|\lambda|)^N ||\wt f(\lambda,k)||<\infty
\end{equation}
\end{Lemma}

\begin{proof}
Using the integration formula with respect to the Iwasawa decomposition $G = ANK$ (for example in \cite[Ch.I \S 5 Corollary 5.3]{HelgGGA}), one can prove that for all $f \in C_c^\infty(G,\tau)$ \begin{equation}\label{FourierAbeltransform}
    \wt f (\lambda,k) = \F_\Aa\left(L_{k^{-1}}\hat{f}\right)
\end{equation}
Here $\lambda \in \Aa^*$ and $k\in K$, \begin{equation*}
    \hat{f}(g):= e^{\rho(H(g))} \int_N f(gn)dn
\end{equation*} is called the Radon transform of $f$ and 
\begin{equation}
    \F_\Aa(\phi)(\lambda) = \int_\Aa \phi(X)e^{i\lambda (X)} ~dX
\end{equation}
the Fourier transform on $\Aa$ for $\phi \in C_c^\infty(\Aa)$ and $\lambda \in \Aa^*$.
Recall that the Euclidean Paley-Wiener theorem ensures that $\F_\Aa(\phi)$ is an entire function of the exponential type and rapidly decreasing, i.e. if $\supp \phi \subset B^\Aa_R$ then
$$\forall N \in \N, ~\exists C_N\geq 0 , \text{ so that } |\F_\Aa(\phi)(\lambda)|\leq C_N(1+|\lambda|)^{-N}e^{R~|\Im (\lambda)|}$$ 
Define $\hat{f_i}(\cdot) := \langle \hat{f}(\cdot), e_i\rangle$. Then $\hat{f_i}$ is an smooth function.

The idea of the proof is as follows:
\begin{equation*}
    \supp f \subset B_\R ~~\stackrel{\circled{2}}{\Longrightarrow}~~ \supp L_{k^{-1}}\hat{f_i} \subset B^\Aa_\R ~~\stackrel{\circled{1}}{\Longleftrightarrow}~~ \eqref{eq PW}
\end{equation*}
\begin{enumerate}
    \item [$\circled{1}$] Since $\sup\limits_{i = 1, \ldots,d_\tau}|\langle~ \cdot~, e_i\rangle|$ is a norm on $\Hi_\tau$ and since all norms on $\Hi_\tau$ are equivalent, \eqref{eq PW} is equivalent to
    \begin{equation*}
        \sup\limits_{\lambda\in \Aa_\C^\ast, ~k\in K, ~i = 1, \ldots,d_\tau} e^{-r |\Im \lambda |}(1+|\lambda|)^N | \langle \wt f(\lambda,k), e_i \rangle|<\infty\,.
    \end{equation*}
    Moreover, by \eqref{FourierAbeltransform}, this is also equivalent to 
    \begin{equation*}
        \sup\limits_{\lambda\in \Aa_\C^\ast, ~k\in K, ~i = 1, \ldots,d_\tau} e^{-r |\Im \lambda |}(1+|\lambda|)^N |\F_A\big(L_{k^{-1}} \hat{f_i}\big)(\lambda)|<\infty\,.
    \end{equation*}
    In turn, by the Paley-Wiener theorem for the Fourier transform on $\Aa$, this is equivalent to $\supp L_{k^{-1}}\hat{f_i} \subset B_R^\Aa$ as a function on $\Aa$ for every $i$ and for every $k\in K$.
    \item [$\circled{2}$] Suppose $\supp f \subset B_R$ and let $X \not\in B^\Aa_R$. For all $k\in K$ and $n\in N$, due to \cite[Chapter IV, (13)]{HelgGGA}: 
    \begin{equation*}
        d(o,ke^XnK) >\left|\IWH(ke^Xn)\right| = |X| \geq R
    \end{equation*}
    So $ke^Xn \not\in B_R \supset \supp f$ and then $\hat{f_i}(ke^X) = 0$ for every $i$ which implies that $X \not\in \supp L_{k^{-1}}\hat{f_i}$. 
\end{enumerate}
\end{proof}

From Lemma \ref{PW} using \eqref{sphconv} and the fact that $K$ is compact, we obtain the following corollary.

\begin{CorL}\label{boundedconvol}
For every function $f\in C_c(G,\tau)$, the convolution product $\varphi^{\sigma,\lambda}_\tau\ast f$ is an even entire function of $\lambda \in \Aa_\C^*$ with the property that there exists a constant $r>0$ such that for all $N\in \N$ the following inequality holds: 
\begin{equation}
    \sup\limits_{\lambda\in \Aa_\C^\ast} e^{-r |\Im \lambda |}(1+|\lambda|)^N ||\varphi^{\sigma,\lambda}_\tau\ast f||<\infty~.
\end{equation}
\end{CorL}

\section{Computation of the resonances}

\label{resonancesparagraph}

In this section, we prove Theorem \ref{Thmintro1}. We recall that the resonances of the positive Laplace operator $\Delta$ are defined as the poles of the meromorphic continuation of its resolvent $(\Delta -z)^{-1}$ considered as an operator defined on $C_c^\infty(G,\tau)$. We know thanks to Lemma \ref{sphiseigenf} that the spherical functions $\varphi_\tau^{\sigma,\lambda}$ are eigenfunctions of  $\Delta$ for eigenvalue $M(\sigma,\lambda):= \langle\lambda,\lambda\rangle + \langle\rho,\rho\rangle - \langle\mu_\sigma +\rho_\m,\mu_\sigma+\rho_\m\rangle$.
The Plancherel theorem \eqref{plancherelformula} gives us a decomposition of $L^2(G,\tau)$ into a continuous and, possibly, a discrete part:
\begin{equation*}
    L^2(G,\tau) = L^2_\text{cont}(G,\tau) \oplus L^2_\text{dicr}(G,\tau)
\end{equation*}
By a slight abuse of notation, we identify $R(z)$ with its restriction to the set $L^2_\text{cont}(G,\tau)\cap C_c^\infty(G,\tau)$. The reader will notice that the discrete part we are omitting brings no additional resonances. 

\begin{Lemma}
    Let $z\in \C\setminus [\langle\rho,\rho\rangle, +\infty[$. The function $R(z)$ can be written 
    \begin{equation}
        R(z) = \frac{1}{d_\tau } \sum_{\sigma \in \hat{M}(\tau)}  R_\sigma(\zeta_\sigma)
    \end{equation}
    for all $f \in C_c^\infty(G,\tau)$ where 
    \begin{equation}\label{Rsigma}
    R_\sigma(\zeta_\sigma) f(x)=\frac{1}{|\alpha|}\int_\R\frac{1}{\zeta_\sigma-\lambda|\alpha|} \Big(\varphi_\tau^{\sigma,\lambda\alpha} \ast f  \Big)(x) ~ \frac{p_\sigma(\lambda\alpha)}{\lambda} ~ d\lambda
\end{equation}
and $\zeta_\sigma$ is defined in \eqref{zetasigma}. 
\end{Lemma}

\begin{proof}
Let $f \in C_c^\infty(G,\tau)$. Due to the inversion of the vector-valued Helgason-Fourier transform \eqref{plancherelformula} and the formula \eqref{sphconv} for the convolution product between $\varphi_\tau^{\sigma,\lambda}$ and $f \in L^2_\text{cont}(G,\tau)\cap C_c^\infty(G,\tau)$, we have
\begin{equation*}
    f(x) = \frac{1}{d_\tau} \sum_{\sigma \in \hat{M}(\tau)} \int_{\Aa^\ast}\Big(\varphi_\tau^{\sigma,\lambda} \ast f  \Big)(x)~p_\sigma(\lambda)~~d\lambda
\end{equation*}
Because of Lemma \ref{PW}, we know that $~\varphi_\tau^{\sigma,\lambda} \ast f~$ is a rapidly decreasing smooth function. So by Lemma 2.2,
\begin{align*}
    R(z)f &= \frac{1}{d_\tau} \sum_{\sigma \in \hat{M}(\tau)} \int_{\Aa^\ast}\Big((\Delta -z)^{-1}\varphi_\tau^{\sigma,\lambda} \ast f  \Big)(x)~p_\sigma(\lambda)~~d\lambda\\
         &= \frac{1}{d_\tau } \sum_{\sigma \in \hat{M}(\tau)}  \int_{\Aa^*}  (M(\sigma,\lambda) - z)^{-1} \Big(\varphi_\tau^{\sigma,\lambda} \ast f  \Big)(x) ~ p_\sigma(\lambda) ~ d\lambda~.
\end{align*}
Computing $R(z)$ is then equivalent  to compute for each $\sigma\in \hat{M}(\tau)$ \begin{equation*}R_\sigma(z) := \int_{\Aa^*}  (M(\sigma,\lambda) - z)^{-1} \Big(\varphi_\tau^{\sigma,\lambda} \ast f  \Big)(x) ~ p_\sigma(\lambda) ~ d\lambda \end{equation*}
 As the values of $\mu_\sigma$, $\rho$ and $\rho_M$ are constant we can introduce the variable $\zeta_\sigma$ defined in \eqref{zetasigma}. Finally, changing the variable with the isomorphism \eqref{lambda} between $\Aa^*_\C$ and $\C$:
\begin{equation*}
    R_\sigma(\zeta_\sigma) f(x):= \int_{\R}  (\zeta_\sigma^2 -\lambda^2|\alpha|^2)^{-1} \Big(\varphi_\tau^{\sigma,\lambda\alpha} \ast f  \Big)(x) ~ p_\sigma(\lambda\alpha) ~ d\lambda
\end{equation*}
Notice that we are using the same symbol $\lambda$ for the the variable in $\Aa^*$ and the corresponding value $\lambda_\alpha \in \R$.
Since the function $\varphi^{\sigma,\lambda}_\tau$ is even in $\lambda$ the same computation as in \cite[Lemma 2.4]{HilgPasq} yields the formula \eqref{Rsigma}. 
\end{proof}

Morera's theorem ensures that the function $R_\sigma(~\cdot~) f(x)$ is holomorphic in $\zeta_\sigma \in \C\setminus \R$. We want to extend meromorphically $R_\sigma(~\cdot~) f(x)$ from one half-plane to the other. 
To fix notation, we will consider $R_\sigma(~\cdot~) f(x)$ as a function defined on the upper half plane  $\Im (\zeta)>0$. We will find the meromorphic continuation of this function to $\C$ by shifting the contour of integration in the direction of the negative imaginary axis and by applying the residue theorem. The poles of the Plancherel density give then poles of the meromorphic continuation. That is why, if the Plancherel density has no poles, so does the meromorphic continuation, and the Laplacian has no resonances (as for $G=\Spin(2n+1,1)$). 

The formula of Plancherel density is given in Appendix \ref{Plancherel densities section} (equations \eqref{eq Plancherel density Rcase}, \eqref{eq Plancherel density Ccase}, \eqref{eq Plancherel density Qcase} and \eqref{eq Plancherel density Ocase}) for each of the rank-one groups $G$. The following proposition unifies the case-by-case formulas found in the literature. The proof is a direct computation. 

\begin{Prop}\label{Propo plancherel density}
Let $G$ be of real rank-one. Set \begin{equation*}
    m^\alpha := \frac{1}{2}(m_{\alpha/2} + m_\alpha -1)~.
\end{equation*}
Then the Plancherel density is given by the following formula:
\begin{equation}\label{Plancherelformula eq}
  p_\sigma(\lambda\alpha) =
         (-1)^s ~ \lambda ~ \tanh\left(\pi\lambda + \frac{3\pi si}{2}\right)  \prod_{j=1}^{m^\alpha} \left(\lambda^2 +\left(B_j +\rho-j \right)^2 \right)
\end{equation}
where 
\begin{equation}\label{eq definition of s}
    s= \left\{\begin{array}{cl}
    2 b_1                          & \text{ if } G = \Spin(2n,1)  \\
     2b_0 + n - 1                  & \text{ if } G = \SU(n,1)\\
     2b_0                          & \text{ if } G = \Sp(n,1)\\
     2 b_1                          & \text{ if } G = \F_4
\end{array}\right.\end{equation} and \begin{equation}\label{eq definition of B_j}
    B_j = \left\{\begin{array}{cl}
      b_j                     & \text{ if } G = \Spin(2n,1)  \\
     b_{j+1} - b_0                & \text{ if } G = \SU(n,1)\\
     \left(\begin{array}{cl}
         b_{1+j} - b_0 -1& \text{ for } j=1,\ldots,n \\
         b_0 - b_{2n+1-j} -1 & \text{ for } j=n+1,\ldots,2n-1
     \end{array}\right)                       & \text{ if } G = \Sp(n,1)
\end{array}\right.
\end{equation}

If $G = \F_4$ the exceptional case, then
\begin{equation}\label{eq definition of B_j F4}
    (B_1, \ldots, B_7) = (b_1+b_2+b_3, ~b_1+b_2-b_3,~ b_1/2, ~b_2/2,~ b_3/2, ~-b_1+b_2+b_3,~ -b_1+b_2-b_3)
\end{equation}
\end{Prop}

Here the $b_j$'s are the coefficients of the highest weight of $\sigma$. We refer the reader to Appendix \ref{Plancherel densities section} to see the computations of this formula.

\begin{Rem}
To make the computations in each of the four cases, one needs the following table:

\renewcommand{\arraystretch}{1.5}

\begin{center}
 \begin{tabular}{|c|c|c|c|c|c|c|} 
 \hline
 $G$ & $K$ & $\Sigma^+$ & $m_{\alpha/2}$ & $m_\alpha$ & $\rho_\alpha$ & $m^\alpha$ \\
 \hline
 $\Spin(2n,1)$ &$\Spin(2n)$ & $\{\alpha\}$ & $0$ & $2n-1$ & $n-\frac{1}{2}$ & $n-1$ \\  
 \hline
$\SU(n,1)$ &$\Ss(\U(n) \times \U(1))$ & $\{\alpha/2,\alpha\}$ & $2n-2$ & $1$ & $\frac{n}{2}$ & $n-1$\\ 
 \hline
$\Sp(n,1)$ &$\Sp(n)$ & $\{\alpha/2,\alpha\}$ & $4n-4$ & $3$ & $n+\frac{1}{2}$ & $2n-1$\\
 \hline
 $\F_4$ &$\Spin(9)$ & $\{\alpha/2,\alpha\}$ & $8$ & $7$ & $\frac{11}{2}$ & $7$\\
 \hline
\end{tabular}
\end{center}
\end{Rem}

\vspace{3mm}
\begin{minipage}{0.6\textwidth}
\underline{Suppose $s$ even}: The formula above goes in $\tanh(\pi\lambda)$, which has first order poles at $\lambda \in i\left(\Z+\frac12\right)$ however the singular points of this function are the imaginary numbers of the form $i\left(\Z + \frac{1}{2}\right)$. Since $s$ is even, the $B_j + \rho - j$ are in $\Z + \frac{1}{2}$. 
Hence the zeros of the polynomial part of $\frac{p_\sigma(\lambda\alpha)}{\lambda}$ are the element
\begin{equation}\label{eq zero polyn Plch density}
    \left\{\pm i(B_j +\rho -j) ~|~ j=1,\ldots, m^\alpha \right\}~.
\end{equation}
Thus the Plancherel density has simples poles in the complement of this set in $i(\Z +\frac{1}{2})$.

\vspace{0.5cm}

\underline{Suppose $s$ odd}: The formula above goes in $\coth(\pi\lambda)$, with simple poles at $\lambda\in i\Z$. Since $s$ is odd, the $B_j+ \rho - j$ are in $\Z$. 
Thus the poles of the Plancherel density are simple and located in the complementary of the set \eqref{eq zero polyn Plch density} in $i\Z $.
\end{minipage}
\begin{minipage}{0.04\textwidth}
\hspace{1mm}
\end{minipage}
\begin{minipage}{0.3\textwidth}
\begin{center}
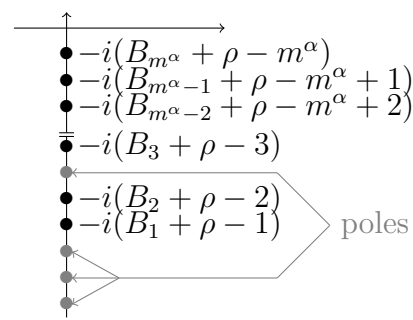

\begin{tikzpicture}[scale=0.7]
\draw [->](-1,0) to (3,0);
\draw [<-|](0,0.3) to (0,-2);
\draw [-|](0,-5.5) to (0,-2.05);
\draw (0,-3.75) node[right] {$-i(B_1+\rho-1)$} node{$\bullet$};
\draw (0,-3.25) node[right] {$-i(B_2+\rho-2)$} node{$\bullet$};
\draw (0,-2.75)  node{\textcolor{gray}{$\bullet$}};
\draw (0,-2.25) node[right] {$-i(B_3+\rho-3)$} node{$\bullet$};
\draw (0,-1.5) node[right] {$-i(B_{m^\alpha-2}+\rho-m^\alpha+2)$} node{$\bullet$};
\draw (0,-1) node[right] {$-i(B_{m^\alpha-1}+\rho-m^\alpha +1)$} node{$\bullet$};
\draw (0,-0.5) node[right] {$-i(B_{m^\alpha}+\rho-m^\alpha)$} node{$\bullet$};
\draw (0,-4.25)  node{\textcolor{gray}{$\bullet$}};
\draw (0,-4.75)  node{\textcolor{gray}{$\bullet$}};
\draw (0,-5.25)  node{\textcolor{gray}{$\bullet$}};
\draw [gray,<-](0.1,-2.75) to (4,-2.75);
\draw [gray,<-](0.1,-4.75) to (1,-4.75);
\draw [gray,<-](0.1,-4.25) to (1,-4.75);
\draw [gray,<-](0.1,-5.25) to (1,-4.75);
\draw [gray,-](4,-4.75) to (5,-3.75);
\draw [gray,-](4,-2.75) to (5,-3.75);
\draw [gray,-](4,-4.75) to (1,-4.75);
\draw (5,-3.75) node[right] {\textcolor{gray}{poles}};
\end{tikzpicture}
\captionof{figure}{Example of poles with $B_2 = B_3+1$ for $\SO(2n,1)$}
\end{center}
\end{minipage}

\vspace{3mm}


Let   \begin{equation}\label{eq Bmax}
     B_{\max} = \max(|B_1+\rho-1|,|B_{m^\alpha}+\rho-m^\alpha|)~.
  \end{equation} 
  All the values of the form $-i(B_{\max} +k)$, with $k\in \N$, are poles of the Plancherel density. 
Let us view $R_\sigma$ as the $D'(X)$-valued holomorphic function on $\Im(\zeta_\sigma) >0$ defined in \eqref{Rsigma}. We want determine the meromorphic continuation of this function through the real axis.
For this we are shifting the contour of integration in the direction of the negative imaginary axis as in Figure \ref{shift} below. Due to Corollary \ref{boundedconvol}, the convolution product is rapidly decreasing in $\lambda$. Then we just need to upper bound the expression $\left|(\zeta_\sigma-\lambda|\alpha|)^{-1} ~ \frac{p_\sigma(\lambda\alpha)}{\lambda}\right|$ by a polynomial in $|\lambda|$, to make the two integrals along the vertical segments between $-R$ and $-R-i(\rho_1+N+1/2)$ and between $R$ and $R-i(\rho_1+N+1/2)$ tend to $0$ when $R$ goes to infinity.
For $|\Re(\lambda)|$ near to infinity and $\Im (\zeta_\sigma)>0$, we have  $\left|(\zeta_\sigma-\lambda|\alpha|)^{-1} \right|< \Im(\zeta_\sigma)^{-1}$ and $\left|\frac{p_\sigma(\lambda\alpha)}{\lambda}\right| < (1+|\lambda|)^{\deg(p_\alpha)+1}$, where $p_\alpha$ is the polynomial part of $p_\sigma$. So the shift is allowed for all $N \in \N$. 

\definecolor{light-gray}{gray}{0.75}
\begin{center}
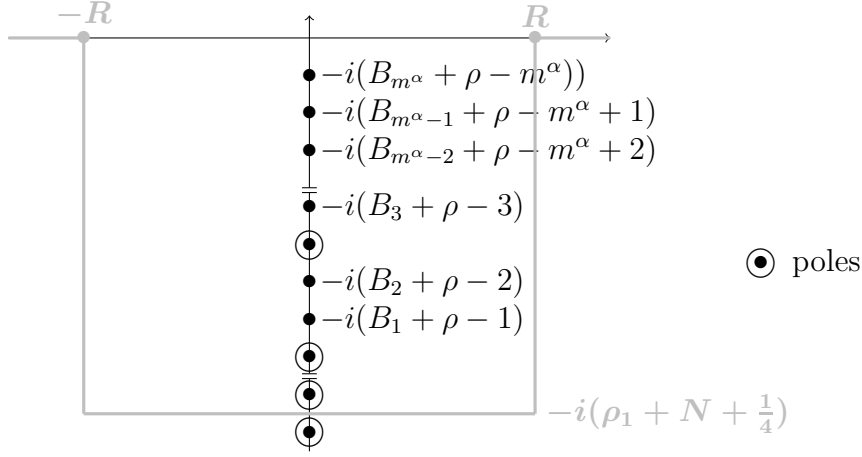

\begin{tikzpicture}[scale=1]
\draw [->](-4,0) to (4,0);
\draw [<-|](0,0.3) to (0,-2);
\draw [|-|](0,-4.475) to (0,-2.05);
\draw [|-](0,-4.525) to (0,-5.5);
\draw (0,-3.75) node[right] {$-i(B_1+\rho-1)$} node{$\bullet$};
\draw (0,-3.25) node[right] {$-i(B_2+\rho-2)$} node{$\bullet$};
\draw (0,-2.75)  node{$\bullet$} node{$\bigcirc$};
\draw (0,-2.25) node[right] {$-i(B_3+\rho-3)$} node{$\bullet$};
\draw (0,-1.5) node[right] {$-i(B_{m^\alpha-2}+\rho-m^\alpha+2)$} node{$\bullet$};
\draw (0,-1) node[right] {$-i(B_{m^\alpha-1}+\rho-m^\alpha+1)$} node{$\bullet$};
\draw (0,-0.5) node[right] {$-i(B_{m^\alpha}+\rho-m^\alpha)$)} node{$\bullet$};
\draw (0,-4.25)  node{$\bullet$} node{$\bigcirc$};
\draw (0,-4.75) node{$\bullet$} node{$\bigcirc$};
\draw (0,-5.25)  node{$\bullet$} node{$\bigcirc$};
\draw (3,0) node[above] {\textcolor{light-gray}{\textbf{$R$}}} node{\textcolor{light-gray}{$\bullet$}};
\draw (-3,0) node[above] {\textcolor{light-gray}{\textbf{$-R$}}} node{\textcolor{light-gray}{$\bullet$}};
\draw [light-gray,very thick,-](-4,0) to (-3,0);
\draw [light-gray,very thick,-](-3,0) to (-3,-5);
\draw [light-gray,very thick,-](-3,-5) to (3,-5);
\draw [light-gray,very thick,-](3,-5) to (3,0);
\draw [light-gray,very thick,-](4,0) to (3,0);
\draw (3,-5) node[right] {\textcolor{light-gray}{\textbf{$-i(\rho_1+N+\frac{1}{4})$}}} node{};
\draw (6,-3)  node{$\bullet$} node{$\bigcirc$} node[right]{~~poles};
\end{tikzpicture}
\captionof{figure}{Shift of contour and residue theorem for $\SO(2n,1)$}\label{shift}
\end{center}

Let $\N_\sigma$ be the set of $k \in \Z$ such that 
\begin{equation*}
    \lambda_k := -i(B_{\max}+k)
\end{equation*}
is a pole of the Plancherel density \eqref{Plancherelformula eq} and $B_{\max}+k \geq 0$. The residue theorem ensures us that for all $N\in \N$ and $\zeta_\sigma$ with $\Im(\zeta) > 0$ : 
\begin{multline}\label{meromorphiccontinuation}
\left(R_\sigma(\zeta_\sigma)f\right)(x) = 
\frac{1}{|\alpha|}\int_{\R-i(N+1/4)}\frac{1}{\zeta_\sigma-\lambda|\alpha|} \Big(\varphi_\tau^{\sigma,\lambda\alpha} \ast f  \Big)(x) ~ \frac{p_\sigma(\lambda\alpha)}{\lambda} ~ d\lambda \\
+~~~
\frac{2i\pi}{|\alpha|}\sum_{\substack{k\in \N_\sigma \\ \lambda_k > -i(N+1/4)}} \frac{1}{\zeta_\sigma-\lambda_k|\alpha|} \Big(\varphi_\tau^{\sigma,\lambda_k\alpha} \ast f  \Big)(x) ~ \Res_{\lambda=\lambda_k}\frac{p_\sigma(\lambda\alpha)}{\lambda}
\end{multline}
The right-hand side of this formula yields a meromorphic continuation of the resolvant of the Laplace operator on $\Im(\zeta) > -(N+1/4)$. The singular values of the Plancherel density induce then poles of the meromorphic continuation of the resolvent of the Laplace operator. If we resume the notation in the expression \eqref{zetasigma}: $z = -\zeta_\sigma^2 -\langle\rho,\rho\rangle + \langle\mu_\sigma +\rho_M,\mu_\sigma+\rho_M\rangle$. This proves Theorem \ref{Thmintro1}.


\section{Residue representations}
\label{section residue repr}

We consider $C^\infty(G,\tau)$ as a $G$-module by left-translations. One can see that for each $\sigma \in \hat{M}(\tau)$ and for each $k \in N_\sigma$, the residues at $\lambda_k$ in the meromorphic continuation \eqref{meromorphiccontinuation} span a $G$-invariant subspace of $C^\infty(G,\tau)$ if $f \in C_c^\infty(G,\tau)$. This is exactly the image of the $G$-intertwining map :
\begin{equation}\label{residueoperatordefinition}
    \fonction{R^\sigma_k}{C_c^\infty(G,\tau)}{C^\infty(G,\tau)}{f}{\varphi^{\sigma,\lambda_k}_\tau \ast f}~.
\end{equation}
We denote this space by
\begin{equation}\label{residuerepresentationdefinition}
    \E_k^\sigma := \{\varphi^{\sigma,\lambda_k}_\tau \ast f ~|~ f \in  C_c^\infty(G,\tau)\}~.
\end{equation}
We want to identify these representations in terms of Langlands parameters, decide which of them are unitarizable and compute their wave front sets. The idea is to decompose $R^\sigma_k$ as follows: 
\newline 
Recall the notation $(\pi^\sigma_{\lambda_k}, \Hi^\sigma_{\lambda_k\alpha})$ for the principal series representation corresponding to $\sigma$ and $\lambda_k\alpha$. For all $l = 1,\ldots,m(\sigma,\tau|_M)$, we denote by $P_l$ the projection of $\Hi_{\lambda_k\alpha}^\sigma$ on its $l$-th $K$-isotypic component, which we identify with $(\tau, \Hi_\tau)$. We get the following decomposition of the residue operator $R^\sigma_k$ as a composition of two $G$-intertwining map:
\begin{equation}\label{eq decompo of Rk}
    \begin{array}{cll}
        R^\sigma_k : C_c^\infty(G,\tau) & \rightarrow \Hi^\sigma_{\lambda_k\alpha} &\rightarrow C^\infty(G,\tau)\\
         ~~f                  & \mapsto ~T_l(f)                     &\mapsto \displaystyle\sum_l P_l ~\pi^\sigma_{\lambda_k\alpha}(~\cdot~^{-1})\big(T_l(f)\big)
    \end{array}
\end{equation}
where $T_l$ is the map from $C^\infty_c(G,\tau)$ to $\Hi_{\lambda_k\alpha}^\sigma$ defined by 
\begin{equation}\label{mapT}
    T_l(f) = \int_G \pi_{\lambda_k}^\sigma(g)\big(P_l^\ast f(g)\big) ~dg~. 
\end{equation}
Here $\ast$ denotes the Hermitian adjoint. Hence, $P_l^*$ maps $\Hi_\tau$ into the principal series and is $K$-equivariant. 
\begin{Lemma}\label{imageofTl} $T_l$ is an intertwining operator between the left regular representation on $C^\infty_c(G,\tau)$ and the principal series representation $(\pi_{\lambda_k}^\sigma, \Hi_{\lambda_k\alpha}^\sigma)$. Moreover, for each $l$ the range of the map $T_l$ is the closed subspace of $\Hi^\sigma_{\lambda_k\alpha}$ spanned by the left translates of $P_l^*\Hi_\tau$. We will denote this space by $\<\pi_{\lambda_k}^\sigma(G) P_l^*\Hi_\tau\>$.
\end{Lemma}
\begin{proof}
First of all, by definition, $T_l\big( C^\infty_c(G,\tau)\big)$ is contained in $\<\pi_{\lambda_k}^\sigma(G) P_l^*\Hi_\tau\>$. 
Fix an element $g_0$ of $G$ and a vector $v_0$ in $\Hi_\tau$. 
Let $\delta_{g_0}\in C_c(G)^*$ be the (scalar) Dirac delta at $g_0$ and set $\delta_{g_0, v_0} = \delta_{g_0}v_0$.
Consider the operator on $C_c^\infty (G,\tau)$ defined by as the distribution $f \mapsto \int_G\int_K\tau(k)\delta_{g_0, v_0}(gk)dk ~f(g)dg$. Since $P_l^*$ is a $K$-intertwining operator between $\tau$ and $\pi_{\lambda_k}$, calculations show that \begin{equation*}
    T_l\left(\int_K\tau(k)\delta_{g_0, v_0}(~\cdot~k)~dk\right) = \pi_{\lambda_k}^\sigma(g_0)P_l^\ast v_0 ~.
\end{equation*}
Since the Dirac delta can be approximated by smooth compactly supported functions, each element of $\pi_{\lambda_k}^\sigma(G)\big( P_l^*\Hi_\tau\big)$ can be written as limit of elements of $T_l\big( C^\infty_c(G,\tau)\big)$. This proves the lemma.
\end{proof}

\begin{Rem}
The map from $\Hi_{\lambda_k\alpha}^\sigma$ to $C^\infty_c(G,\tau)$ defined by $\phi \mapsto  P_l\pi_{\lambda_k}^\sigma(~\cdot~^{-1}) \phi$ is a $G$-intertwining operator between $(\pi_{\lambda_k}^\sigma, \Hi_{\lambda_k\alpha}^\sigma)$ and the left regular representation on $C^\infty_c(G,\tau)$. It is known as the Poisson transform (see \cite{Olb, Yang}).
\end{Rem}

The idea how to identify $\E_k^\sigma$ is to compute the range of the map $T_l$, for each $l$, using Lemma \ref{imageofTl}. Knowing the composition series of $\Hi_{\lambda_k\alpha}^\sigma$, we can identify this range in this principal series and then project back with the Poisson transform, the second part of the map in \eqref{eq decompo of Rk}.
The main issue is that,even in rank-one the structure of the principal series representations is very complicated in general. See \cite{Collin1}.
In this paper we will consider only the case when $\sigma$ is the trivial representation of $M$. In this case the composition series is more transparent and the result have a pleasant uniform form. For example, $\tau$ occurs in $\pi^\sigma_{\lambda_k}$ with multiplicity one. So, from now on, we fix an irreducible representation $\tau \in \hat{K}$ which contains the trivial representation of $M$. We study only the residue representations which arise from $\sigma = \triv_M$ and we will denote it $\E_k:=\E_k^{\triv_M}$. Similarly, we set $R_k := R_k^\triv$ et write $T$ and $P$ instead of $T_l$ and $P_l$ respectively. The map in \eqref{eq decompo of Rk} becomes:
\begin{equation}\label{eq decompo of Rk triv case}
    \begin{array}{cll}
        R_k : C_c^\infty(G,\tau) & \rightarrow \Hi^\triv_{\lambda_k\alpha} &\rightarrow C^\infty(G,\tau)\\
         ~~f                  & \mapsto ~~~~~~~~T(f)                     &\mapsto  P ~\pi^\triv_{\lambda_k\alpha}(~\cdot~^{-1})\big(T(f)\big)
    \end{array}
\end{equation}
where $P$ is the projection onto the $\tau$-isotypic component. 
The structure of the spherical principal series representations $(\pi_\lambda,\Hi_\lambda) := (\pi_\lambda^\triv,\Hi^\triv_\lambda)$ of our groups $G$ has been studied by different authors (see e.g. \cite{HoweTan,JohnWallPrincseries, JohnF4, Molchanov}) for every $G$ we are studying. Our main reference will be the paper of Howe and Tan \cite{HoweTan} which provides an explicit description of the subquotients of this principal series representation.
We will treat the residue representations case by case. For each case, we compute the Langlands parameters of $\E_k$, $k\in \N_\sigma$. Then to have more information about the size of the infinite-dimensional ones, we compute the wave front set of the representations. For semisimple Lie groups, this object is a closed union of nilpotent orbits of $\g$  (see \cite[Proposition 2.4]{HoweWFS}). Its computation needs some additional results about Gelfand-Kirillov dimension \cite[Theorem 1.2]{VoganGelfKirill} and on nilpotent orbits in semisimple Lie algebras \cite{CollinMcGov}.

\subsection{Case of $\SO(2n,1)$} \label{section E_k study - Rcase}

Let $G = \SO(2n,1)$. The $K$-types of the spherical principal series representations are parametrized by $m\in\N$ and known as the space of spherical harmonics on $\R^{2n}$ of homogeneous degree $m$, denoted by $\Hi^m(\R^{2n})$. 
Their highest weight of the form $m\epsilon_1$ with respect to the fundamental weights described in section \ref{Section Plancherel density Rcase}. 
A representation $\tau$ containing the trivial representations of $M$ is of this form. From now on, let $\tau$ act on the harmonic polynomials on $\R^{2n}$ of fixed homogeneous degree $N$: \begin{equation*}
    \Hi_\tau \simeq \Hi^N(\R^{2n})~.
\end{equation*}

As $\sigma$ is trivial the poles of the Plancherel density $p_{\triv_M}$ are the $\lambda_k = -i(\rho_\alpha+k) = -i(n+k-1/2)$ with $k\in \N$. The composition series of $\Hi_{\lambda_k\alpha}$ described in \cite{HoweTan} is the following:
\begin{equation}\label{eq composition series real}
    \Hi_{\lambda_k \alpha}\simeq \overbrace{\sum_{m=0}^{k} P_m^* \big(\Hi^m(\R^{2n})\big)}^{\circled{1}} \oplus \overbrace{\sum_{m>k} P_m^*\big(\Hi^m(\R^{2n})\big)}^{\circled{2}}
\end{equation}
where $P_m$ is the projection of $\Hi_{\lambda_k \alpha}$ onto the $K$-isotypic component isomorphic to $\Hi^m(\R^{2n})$.
The action of $G$ cannot send a $K$-type from the second summand to the first one.


\begin{proof}[Langlands parameters of $\E_k$]
In Figure \ref{realKtypes} below, each bullet corresponds to one $K$-type of the representation $\Hi_{\lambda_k\alpha}$, the abscissa of the bullet being the coefficient appearing in the highest weight of the $K$-type. The figure describes the two cases $N\geq k$ and $N< k$. The barrier at $k+1$ and the arrows mean that the action of $\g$ cannot send a $K$-type which is on the right of $k+1$ to a $K$-type which on the left of $k+1$.

\begin{center}
\begin{tikzpicture}[scale=1]
\draw [|->](-5,0) to (5,0);
\draw (-5,0.25) node [above]{$0$};
\draw (-5,0)node[left] {$\triv_K$} node{$\bullet$};
\draw (-4.5,0) node[above] {} node{$\bullet$};
\draw (-4,0) node[above] {} node{$\bullet$};
\draw (-3.5,0) node[above] {} node{$\bullet$};
\draw (-3,0) node[above] {} node{$\bullet$};
\draw (-2.5,0) node[above] {} node{$\bullet$};
\draw (-2,0) node[above] {} node{$\bullet$};
\draw (-1.5,0) node[above] {} node{$\bullet$};
\draw (-1,0)  node{$\bullet$};
\draw (-1,-1.5) node[below] {$k+1$};
\draw (-0.5,0) node[above] {} node{$\bullet$};
\draw (0,0) node[above] {} node{$\bullet$};
\draw (5,0) node[above] {} node{$\bullet$};
\draw (4.5,0) node[above] {} node{$\bullet$};
\draw (4,0) node{$\bullet$} node{$\bigcirc$};
\draw (4,0.5) node[above] {\textbf{$N$}};
\draw (3.5,0) node[above] {} node{$\bullet$};
\draw (3,0) node[above] {} node{$\bullet$};
\draw (2.5,0) node[above] {} node{$\bullet$};
\draw (2,0) node[above] {} node{$\bullet$};
\draw (1.5,0) node[above] {} node{$\bullet$};
\draw (1,0) node[above] {} node{$\bullet$};
\draw (0.5,0) node[above] {} node{$\bullet$};
\draw [-](-1,-1.5) to (-1,1.5);
\draw [->](-1,1) to (0,1);
\draw [->](-1,-1) to (0,-1);

\draw [very thick,light-gray,-](-1.1,-0.5) to (-1.1,0.5);
\draw [very thick,light-gray,-](-1.1,-0.5) to (4,-0.5);
\draw [very thick,light-gray,-,dashed](5.5,-0.5) to (4,-0.5);
\draw [very thick,light-gray,-](-1.1,0.5) to (4,0.5);
\draw [very thick,light-gray,-,dashed](5.5,0.5) to (4,0.5);
\draw (3,-1.5) node[above] {\color{light-gray}\textbf{Image of $T$}} node{};

\draw [|->](-5,-4) to (5,-4);
\draw (-5,-4) node[left] {$\triv_K$} node{$\bullet$};
\draw (-4.5,-4) node[above] {} node{$\bullet$};
\draw (-4,-4) node[above] {} node{$\bullet$};
\draw (-3.5,-4) node[above] {} node{$\bullet$};
\draw (-3,-4) node{$\bullet$}node{$\bigcirc$};

\draw (-2.5,-4) node[above] {} node{$\bullet$};
\draw (-2,-4) node[above] {} node{$\bullet$};
\draw (-1.5,-4) node[above] {} node{$\bullet$};
\draw (-1,-4) node{$\bullet$};
\draw (-1,-5.5) node[below] {$k+1$};
\draw (-0.5,-4) node[above] {} node{$\bullet$};
\draw (0,-4) node[above] {} node{$\bullet$};
\draw (5,-4) node[above] {} node{$\bullet$};
\draw (4.5,-4) node[above] {} node{$\bullet$};
\draw (4,-4) node[above] {} node{$\bullet$};
\draw (3.5,-4) node[above] {} node{$\bullet$};
\draw (3,-4) node[above] {} node{$\bullet$};
\draw (2.5,-4) node[above] {} node{$\bullet$};
\draw (2,-4) node[above] {} node{$\bullet$};
\draw (1.5,-4) node[above] {} node{$\bullet$};
\draw (1,-4) node[above] {} node{$\bullet$};
\draw (0.5,-4) node[above] {} node{$\bullet$};
\draw [-](-1,-5.5) to (-1,-2.5);
\draw [->](-1,-3) to (0,-3);
\draw [->](-1,-5) to (0,-5);

\draw [very thick,light-gray,-](-5,-4.5) to (-5,-3.5);
\draw [very thick,light-gray,-](-5,-4.5) to (4,-4.5);
\draw [very thick,light-gray,-,dashed](5.5,-4.5) to (4,-4.5);
\draw [very thick,light-gray,-](-5,-3.5) to (4,-3.5);
\draw [very thick,light-gray,-,dashed](5.5,-3.5) to (4,-3.5);
\draw (3,-3.5) node[above] {\color{light-gray}\textbf{Image of $T$}};
\draw (-3,-3.5) node[above] {\textbf{$N$}};
\draw (-5,-3.5) node [above]{$0$};

\draw [-](4,-4.25) to (-1.1,-4.25) to (-1.1,-3.75) to (4,-3.75);
\draw [-,dashed](5.5,-4.25) to (4,-4.25);
\draw [-,dashed](5.5,-3.75) to (4,-3.75);
\fill [pattern=north east lines,pattern color=light-gray] (5.5,-4.25) to (-1.1,-4.25) to (-1.1,-3.75) to (5.5,-3.75) to (5.5,-4.25);
\draw (-6,-2) node[above] { Kernel of the};
\draw (-6,-2) node[below] { Poisson transform};
\fill [pattern=north east lines,pattern color=light-gray] (-8,-2.25) rectangle (-9,-1.75);
\draw (-8,-2.25) rectangle (-9,-1.75);

\end{tikzpicture}

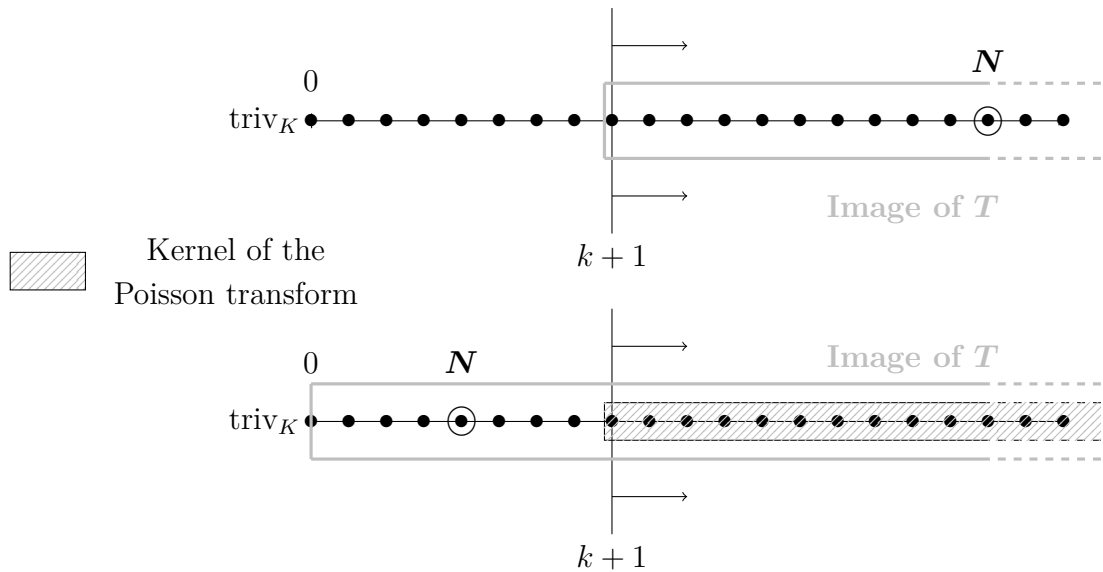
\captionof{figure}{$K$-types in $\Hi_{\lambda_k\alpha}$ and composition series for $k=7$}\label{realKtypes}
\end{center} 

\underline{Suppose that $N\geq k+1$}: The image of $T$ is the space spanned by the action of $\g$ on $P_m^*\big(\Hi^N(\R^{2n})\big)$. Since that action cannot cross the barrier at $k+1$, the image of $T$ is the infinite-dimensional irreducible representation 
\begin{equation*}
    \sum_{m>k} j_{-2n - k+1}\big(\Hi^m(\R^{2n})\big)
\end{equation*}
The second map in \eqref{eq decompo of Rk} (the Poisson transform) is an intertwining operator. So its kernel is either $0$ or the entire representation. Choosing a nonzero $h \in\Hi^N(\R^{2n})$ nonzero, the function $P_{\tau_N} ~\pi_{\lambda_k\alpha}(~\cdot~^{-1})j_{-2n - k+1}\big(h\big)$ has value $h$ at $e_G$. So the kernel of the Poisson transform is not the entire space, thus it is $0$. 
Consequently $\E_k$ is $\circled{2}$ in \eqref{eq composition series real}.

\underline{Suppose that $N\leq k$}: Here the image of $T$ is the entire representation $\Hi_{\lambda_k\alpha}$ because the action of $\g$ can cross the barrier at $k+1$.
The kernel of the second intertwining map in \eqref{eq decompo of Rk triv case} is $\circled{2}$ in \eqref{eq composition series real} by the same argument as before. It follows that $\E_k$ is the subquotient isomorphic to $\circled{1}$ in \eqref{eq composition series real}. 

The unitarity is given in \cite[Diagram 3.15]{HoweTan}. To identify the representations we determined, we compute their Langlands parameters. 
They are collected in Table \ref{Table Ktypes of E_k - Rcase}.

\begin{center}\begin{tabular}{|c|c|c|c|c|}
\hline
    Case        & Minimal $K$-type      & $\delta$                    & values of $\nu$\\
    \hline
    $N>k+1$     & $\Hi^{k+1}(\R^{2n})$  & $\Hi^{k+1}(\R^{2n-1})$    & $\pm \left(n-\frac{3}{2}\right)\alpha$\\
    \hline
    $N\leq k$   & $\triv_K $            & $\triv_M$                 & $i\lambda_k$\\
    \hline
\end{tabular}\captionof{table}{Langlands parameters of $\E_k$ when $G = \SO(2n,1)$}\label{Table Ktypes of E_k - Rcase}
\end{center}
The entries of Table \ref{Table Ktypes of E_k - Rcase} are computed as follows. Let $\mu_l$ be the highest weight of $\Hi^l(\R^{2n})$. A minimal $K$-type minimizes the Vogan norm of the highest weight in $\E_k$:
\begin{equation*}
    \rVert\mu_l\lVert_V = \<\mu_l+2\rho_K,\mu_l+2\rho_K\>
\end{equation*}
where $\rho_K$ is the half sum of positive roots in $S^{\Ak_\C}_+$ (see Appendix \ref{Section Plancherel density Rcase} for the definition).
One can check that $\rVert\mu_l\lVert_V$ is minimal when $l$ is minimal. This yields the first column of the table. 
\newline
Finally, to find $\delta$, one can use \cite[Theorem 3.4]{Bald1}. The $M$-types in $\Hi^l(\R^{2n})$ have highest weight
\begin{equation*}
    \mu_\delta(a) = a \epsilon_1
\end{equation*}
where $a \in [\![0,k+1 ]\!]$. The only one which has the same minimal $K$-type is $\left\{\begin{array}{cc}
    \mu_\delta(0)  & \text{, for } \circled{1} \\
    \mu_\delta(k+1)  & \text{, for } \circled{2}
\end{array}\right.$.

To find $\nu$ one has to compare the infinitesimal character of of $\Ind_{MAN}^G(\triv\otimes e^{i\lambda_k\alpha}\otimes \triv)$ and $\Ind_{MAN}^G(\delta \otimes e^{i\nu}\otimes \triv)$. They have to agree up to the action of the Weyl group of $(\g_\C, \Ak_\C)$. Theorem \ref{Thmintro2}, 1.,  follows from these computations.
\end{proof}

\subsection{Case of  $\SU(n,1)$} \label{section E_k study - Ccase}

Let now $G = \SU(n,1)$. The structure of the spherical principal series is described for $\U(n,1)$  in \cite{HoweTan} but Molchanov find the same result for $\SU(n,1)$ in \cite{Molchanov}.  Choose a basis $\{z_1, \ldots, z_n, z_{n+1}\}$ of $\C^{n+1}$ as complex vector space and $\{z_1, \ldots, z_n, z_{n+1},\overline{z_1}, \ldots, \overline{z_n},\overline{z_{n+1}}\}$ the corresponding basis of $\R^{2n+2}$. The $K$-types of the spherical principal series representations are the spaces \begin{equation*}
    \Hi^{m_1,m_2}(\C^{n})\otimes\Hi^l(\C)
\end{equation*} where $\Hi^{m_1,m_2}(\C^{n})$ are the spherical harmonics on $\C^{2n}$ of homogeneous degree $m_1$ and $m_2$ in the variables $z_1, \ldots, z_{n+1}$ and $\overline{z_1}, \ldots, \overline{z_n},\overline{z_{n+1}}$ respectively. 
Moreover, the space $\Hi^l(\C)$ is defined for all integer $l$ by:
\begin{equation*}
    \Hi^l(\C)= \left\{\begin{array}{cc}
    \Hi^{l,0}(\C) & \text{ if } l\geq0 \\
    \Hi^{0,-l}(\C) & \text{ if } l\leq0
\end{array}\right.
\end{equation*}
Their highest weights are of the form $m_1 \epsilon_1 - m_2\epsilon_n - l \epsilon_{n+1}$ with respect to the fundamental weights described in section \ref{Section Plancherel density Ccase}. 
From now on, let\begin{equation}
    \Hi_\tau \simeq \Hi^{M_1,M_2}(\C^{n}) \otimes\Hi^L(\C)~.
\end{equation} so the highest weight of $\tau$ is
\begin{equation}
    \mu_\tau = M_1\epsilon_1 - M_2\epsilon_n - L \epsilon_{n+1}
\end{equation}
for fixed nonnegative integers $M_1, M_2$, a fixed integer $L$. From now on we call this representation $\tau_{M_1,M_2,L}$. 
As $\sigma$ is trivial, the poles of the Plancherel density are the $\lambda_k = \pm i(\frac{n}{2}+k)$ where $k$ is a non-negative integer. Hence
\begin{equation}\label{eq Ktypes H_lambda_k}
    \Hi_{\lambda_k\alpha}\simeq \sum_{\substack{l\in \Z\\m_1, m_2 \geq 0\\m_1 - m_2 +l = 0}} P^*\big(\Hi^{m_1,m_2}(\C^{n})\otimes\Hi^l(\C)\big)
\end{equation}
where $P := P_{m_1,m_2,l}$ is the projection of $\Hi_{\lambda_k \alpha}$ onto the $K$-isotypic component isomorphic to $\Hi^{m_1,m_2}(\C^{n})\otimes\Hi^l(\C)$. One can see that the conditions in the sum in \eqref{eq Ktypes H_lambda_k} imply that the pair $(m_1+m_2,l)$ determines the triple $(m_1,m_2,l)$ and since $m_1$ and $m_2$ are nonnegative, $-l\leq m_1+m_2 \leq l$ for all $l \in \Z$.

\tikzstyle{circle}=[shape=circle,draw,inner sep=2pt]
\begin{tikzpicture}[scale=0.32]
\draw (-10,-3) node{\textbullet} node{$\bigcirc$};
\draw (-10,-3.25) node[below] {$\tau_{M_1,M_2,L}$};
\fill [pattern=north east lines,pattern color=light-gray] (-11,-8) rectangle (-9,-9);
\draw (-11,-8) rectangle (-9,-9);
\draw (-10,-11) node[above] {Kernel of the};
\draw (-10,-11.25) node {Poisson transform};
\draw (-11,-15) rectangle (-9,-16);
\draw [very thick, light-gray, -] (-11,-15.5) to (-9,-15.5);
\draw (-10,-18) node[above] {Image of $T$};

\draw [->] (0,0) to (11,0);
\draw [->] (0,-8) to (0,8);
\node[circle] (1) at (4,10){1};
\draw (11,0) node[below right] {\tiny $m$};
\draw (0,8) node[above left] {\tiny $l$};
\draw [-] (0,0) to (8,8);\draw [-,dashed] (10.5,10.5) to (8,8); 
\draw [-] (0,0) to (8,-8);\draw [-, dashed] (10.5,-10.5) to (8,-8); 
\foreach \Point in {(0,0), (1,1), (1,-1), (2,-2),(2,2), (2,0), (3,3), (3,-3), (3,1), (3,-1), (3,-3), (4,4), (4,2), (4,0), (4,-2), (4,-4),(5,3), (5,-3), (5,1), (5,-1), (5,-3), (5,-5),(5,5),(6,4), (6,2), (6,0), (6,-2), (6,-4),(6,-6),(6,6),(7,3),(7,-3), (7,1), (7,-1), (7,-3), (7,-5),(7,5),(7,-7),(7,7),(8,4), (8,2), (8,0), (8,-2), (8,-4),(8,-6),(8,6),(8,-8),(8,8),(9,3),(9,-3), (9,1), (9,-1), (9,-3), (9,-5),(9,5),(9,-7),(9,7),(9,-9),(9,9),(10,4), (10,2), (10,0), (10,-2), (10,-4),(10,-6),(10,6),(10,-8),(10,8), (10,-10),(10,10) }{
    \node at \Point {$\bullet$};}
\draw (9,3) node{\textbullet} node{$\bigcirc$};
\draw [thick,-] (0,-6) to (11,5);\draw [thick,->] (4.5,-1.5) to++ (0.75,-0.75);\draw [thick,->] (8.5,2.5) to++ (0.75,-0.75);\draw (0,-6) node[left] {\textbf{\tiny $-2k-2$}};
\draw [thick,-] (0,6) to (11,-5);\draw [thick,->] (4.5,1.5) to++ (0.75,0.75);\draw [thick,->] (8.5,-2.5) to++ (0.75,0.75);\draw (0,6) node[left] {\textbf{\tiny $2k+2$}};
\draw [very thick, light-gray] (11,5.5) to (5.5,0) to (11,-5.5);

\begin{scope}[xshift=20 cm] 
\draw [->] (0,0) to (11,0);
\draw [->] (0,-8) to (0,8);
\node[circle] (2) at (4,10){2};
\draw (11,0) node[below right] {\tiny $m$};
\draw (0,8) node[above left] {\tiny $l$};
\draw [-] (0,0) to (8,8);\draw [-,dashed] (10.5,10.5) to (8,8); 
\draw [-] (0,0) to (8,-8);\draw [-, dashed] (10.5,-10.5) to (8,-8); 
\foreach \Point in {(0,0), (1,1), (1,-1), (2,-2),(2,2), (2,0), (3,3), (3,-3), (3,1), (3,-1), (3,-3), (4,4), (4,2), (4,0), (4,-2), (4,-4),(5,3), (5,-3), (5,1), (5,-1), (5,-3), (5,-5),(5,5),(6,4), (6,2), (6,0), (6,-2), (6,-4),(6,-6),(6,6),(7,3),(7,-3), (7,1), (7,-1), (7,-3), (7,-5),(7,5),(7,-7),(7,7),(8,4), (8,2), (8,0), (8,-2), (8,-4),(8,-6),(8,6),(8,-8),(8,8),(9,3),(9,-3), (9,1), (9,-1), (9,-3), (9,-5),(9,5),(9,-7),(9,7),(9,-9),(9,9),(10,4), (10,2), (10,0), (10,-2), (10,-4),(10,-6),(10,6),(10,-8),(10,8), (10,-10),(10,10) }{\node at \Point {$\bullet$};}
\draw [thick,-] (0,-6) to (11,5);\draw [thick,->] (4.5,-1.5) to++ (0.75,-0.75);\draw [thick,->] (8.5,2.5) to++ (0.75,-0.75);\draw (0,-6) node[left] {\textbf{\tiny $-2k-2$}};
\draw [thick,-] (0,6) to (11,-5);\draw [thick,->] (4.5,1.5) to++ (0.75,0.75);\draw [thick,->] (8.5,-2.5) to++ (0.75,0.75);\draw (0,6) node[left] {\textbf{\tiny $2k+2$}};

\draw (7,5) node{\textbullet}node{$\bigcirc$};
\draw [very thick,light-gray] (11,11.75) to++ (-8.75,-8.75) to (11,-5.75);
\draw [-] (11,5.5) to (5.5,0) to (11,-5.5);
\fill [pattern=north east lines,pattern color=light-gray] (11,5.5) to (5.5,0) to (11,-5.5);
\end{scope}

\begin{scope}[yshift=-25 cm] 
\draw [->] (0,0) to (11,0);
\draw [->] (0,-8) to (0,8);
\node[circle] (3) at (4,10){3};
\draw (11,0) node[below right] {\tiny $m$};
\draw (0,8) node[above left] {\tiny $l$};
\draw [-] (0,0) to (8,8);\draw [-,dashed] (10.5,10.5) to (8,8); 
\draw [-] (0,0) to (8,-8);\draw [-, dashed] (10.5,-10.5) to (8,-8); 
\foreach \Point in {(0,0), (1,1), (1,-1), (2,-2),(2,2), (2,0), (3,3), (3,-3), (3,1), (3,-1), (3,-3), (4,4), (4,2), (4,0), (4,-2), (4,-4),(5,3), (5,-3), (5,1), (5,-1), (5,-3), (5,-5),(5,5),(6,4), (6,2), (6,0), (6,-2), (6,-4),(6,-6),(6,6),(7,3),(7,-3), (7,1), (7,-1), (7,-3), (7,-5),(7,5),(7,-7),(7,7),(8,4), (8,2), (8,0), (8,-2), (8,-4),(8,-6),(8,6),(8,-8),(8,8),(9,3),(9,-3), (9,1), (9,-1), (9,-3), (9,-5),(9,5),(9,-7),(9,7),(9,-9),(9,9),(10,4), (10,2), (10,0), (10,-2), (10,-4),(10,-6),(10,6),(10,-8),(10,8), (10,-10),(10,10) }{
    \node at \Point {$\bullet$};}
\draw [thick,-] (0,-6) to (11,5);\draw [thick,->] (4.5,-1.5) to++ (0.75,-0.75);\draw [thick,->] (8.5,2.5) to++ (0.75,-0.75);\draw (0,-6) node[left] {\textbf{\tiny $-2k-2$}};
\draw [thick,-] (0,6) to (11,-5);\draw [thick,->] (4.5,1.5) to++ (0.75,0.75);\draw [thick,->] (8.5,-2.5) to++ (0.75,0.75);\draw (0,6) node[left] {\textbf{\tiny $2k+2$}};

\draw (8,-6) node{\textbullet}node{$\bigcirc$};
\draw [very thick,light-gray] (11,-11.75) to++ (-8.75,8.75) to (11,5.75);
\draw [-] (11,5.5) to (5.5,0) to (11,-5.5);
\fill [pattern=north east lines,pattern color=light-gray] (11,5.5) to (5.5,0) to (11,-5.5);
\end{scope}

\begin{scope}[xshift=20 cm,yshift=-25 cm] 
\draw [->] (0,0) to (11,0);
\draw [->] (0,-8) to (0,8);
\node[circle] (4) at (4,10){4};
\draw (11,0) node[below right] {\tiny $m$};
\draw (0,8) node[above left] {\tiny $l$};
\draw [-] (0,0) to (8,8);\draw [-,dashed] (10.5,10.5) to (8,8); 
\draw [-] (0,0) to (8,-8);\draw [-, dashed] (10.5,-10.5) to (8,-8); 
\foreach \Point in {(0,0), (1,1), (1,-1), (2,-2),(2,2), (2,0), (3,3), (3,-3), (3,1), (3,-1), (3,-3), (4,4), (4,2), (4,0), (4,-2), (4,-4),(5,3), (5,-3), (5,1), (5,-1), (5,-3), (5,-5),(5,5),(6,4), (6,2), (6,0), (6,-2), (6,-4),(6,-6),(6,6),(7,3),(7,-3), (7,1), (7,-1), (7,-3), (7,-5),(7,5),(7,-7),(7,7),(8,4), (8,2), (8,0), (8,-2), (8,-4),(8,-6),(8,6),(8,-8),(8,8),(9,3),(9,-3), (9,1), (9,-1), (9,-3), (9,-5),(9,5),(9,-7),(9,7),(9,-9),(9,9),(10,4), (10,2), (10,0), (10,-2), (10,-4),(10,-6),(10,6),(10,-8),(10,8), (10,-10),(10,10) }{
    \node at \Point {$\bullet$};}
\draw [thick,-] (0,-6) to (11,5);\draw [thick,->] (4.5,-1.5) to++ (0.75,-0.75);\draw [thick,->] (8.5,2.5)to++ (0.75,-0.75);\draw (0,-6) node[left] {\textbf{\tiny $-2k-2$}};
\draw [thick,-] (0,6) to (11,-5);\draw [thick,->] (4.5,1.5) to++ (0.75,0.75);\draw [thick,->] (8.5,-2.5) to++ (0.75,0.75);\draw (0,6) node[left] {\textbf{\tiny $2k+2$}};

\draw (1,1) node{\textbullet}node{$\bigcirc$};
\draw [very thick,light-gray] (11,-11.5) to (8,-8.5) to (-0.5,-0) to (3,3.5) to (9,9.5) to (11,11.5);
\draw [-] (11, -11.35) to++ (-8.35,8.35) to ++(3, 3) to++(-3,3) to (11, 11.35);
\fill [pattern=north east lines,pattern color=light-gray] (11, -11.35) to++ (-8.35,8.35) to ++(3, 3) to++(-3,3) to (11, 11.35);
\end{scope}
\end{tikzpicture}

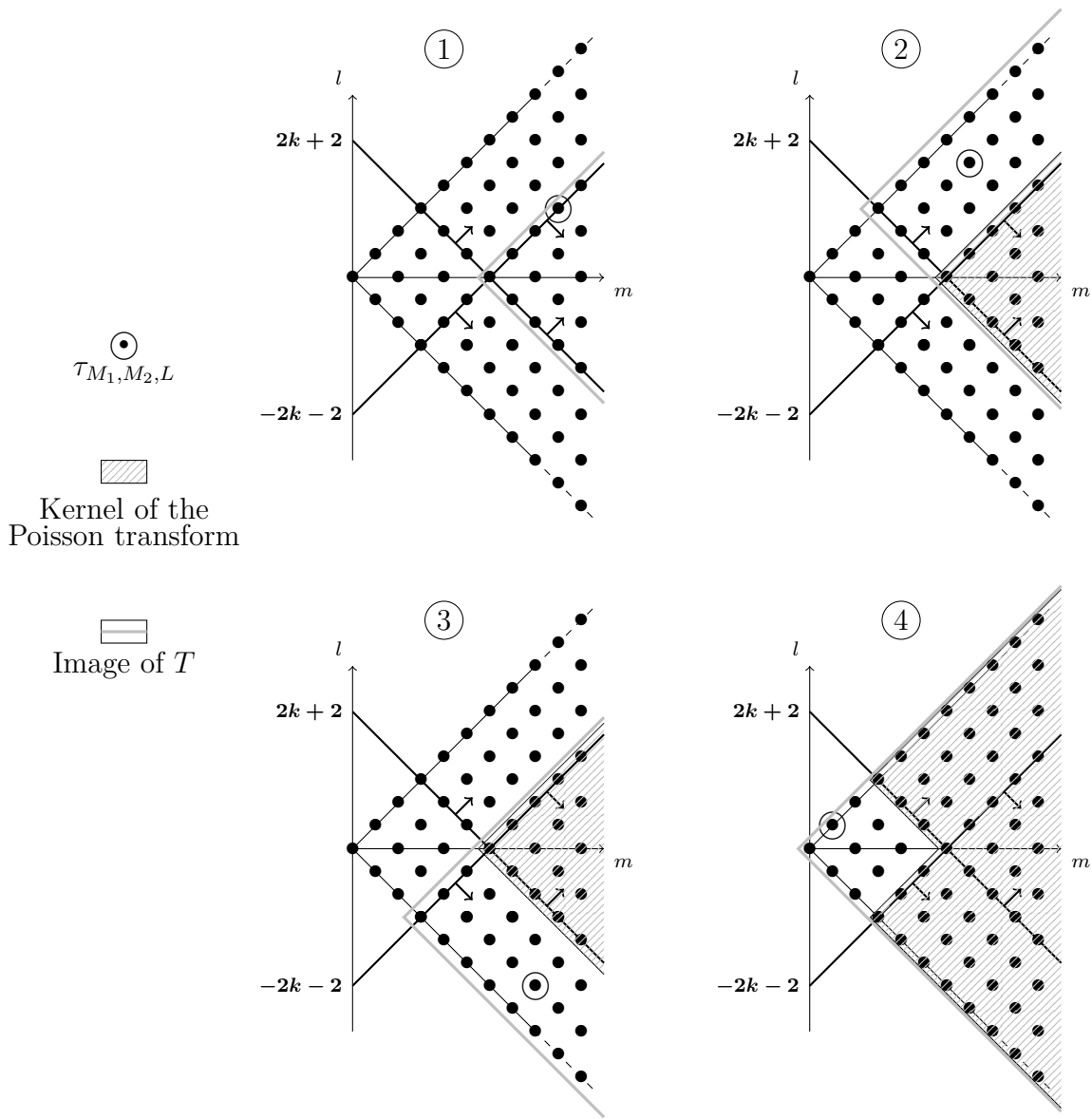
\captionof{figure}{$K$-types in $\Hi_{\lambda_k\alpha}$ and decomposition series for $k=1$}\label{complexKtypes}

\begin{proof}[Langlands parameters of $\E_k$.]
In Figure \ref{complexKtypes}, each bullet corresponds to one $K$-type of $\Hi_{\lambda_k\alpha}$, the coordinates of the point being the pair $(m_1+m_2,l)$ for the $K$-type $\Hi^{m_1,m_2}(\C^{n})\otimes\Hi^l(\C)$. 
\cite[Lemma 4.4]{HoweTan} ensures us that the action of $G$ cannot send a $K$-type to another constituent if it doesn't follow the sense of the arrows.
Thus the space is separated in  four constituents. We denote these constituents North - East - South - West depending on their position. The figure describe the four cases when the $K$-type $\tau_{M_1,M_2,L}$ is in each constituent and is represented by the bullet $(M_1+M_2,L)$. 
The reader should recall Lemma \ref{imageofTl} and \eqref{eq decompo of Rk triv case}.
\begin{enumerate}
    \item[] \textbf{Case $\circled{1}$} : $M_1+M_2 \geq 2k+2$ and $|L| \leq -2k-2+M_1+M_2$  \newline Since the action of $G$ on the $K$-types cannot cross the barriers from the side out, the image of $T$ is the entire East-component \begin{equation*}
    \sum_{\substack{m_1, m_2 \geq 0, ~ m_1+m_2\geq 2k+2\\l\in \Z,~ |l|\leq -2k-2+m_1+m_2\\m_1 - m_2 +l = 0\\}} P^*\big(\Hi^{m_1,m_2}(\C^{n})\otimes\Hi^l(\C)\big)~.
\end{equation*} The Poisson transform is an intertwining operator. Hence its kernel, contained in an irreducible representation, is either $0$ or the entire representation. Choosing a nonzero $h \in\Hi^N(\R^{2n})$, the function $P_{\tau_{M_1,M_2,L}} ~\pi_{\lambda_k\alpha}(~\cdot~^{-1})P^*(h)$ has value $h$ at $e_G$. 
Thus $\E_k $ is equivalent as a representation to the image $T$. \vspace{2mm}
\item[] \textbf{Case $\circled{2}$} : $L\geq |-2k-1+M_1+M_2| +1$ 
\newline From the North-constituent, the action of $G$ can cross the barrier $l=-2k-2+m$ but not the barrier $l=2k+2-m$. Thus the image of $T$ is composed of two North and East-components: 
\begin{equation*}
    \sum_{\substack{m_1, m_2 \geq 0\\l\in \Z,~ l> 2k+2-m_1-m_2\\m_1 - m_2 +l = 0\\}} P^*\big(\Hi^{m_1,m_2}(\C^{n})\otimes\Hi^l(\C)\big)~.
\end{equation*} 
Following the proof in the previous case, one can conclude that the image of the Poisson transform is nonzero and that the North-constituent, where $\tau_{M_1,M_2,L}$ is, is not in the kernel of this map. Because of the barrier $l=-2k-2+m$ the action of $G$ cannot bring a $K$-type from the East-constituent into the North-constituent. Thus the East-constituent is the kernel of the Poisson transform and $\E_k$ is equivalent to the subquotient of the North-East constituents modulo the East-constituent. 
\item[] \textbf{Case $\circled{3}$} : $L\leq |-2k-1+M_1+M_2|-1$ 
\newline This case is completely symmetric to the previous one. The figure explains the result.  \vspace{2mm}
\item[] \textbf{Case $\circled{4}$} : $M_1+M_2 \in [0,2k+2[$ and $L< |2k+2-M_1+M_2 |$
\newline  $\tau_{M_1,M_2,L}$ is in the finite-dimensional West-constituent. From there, the action of $G$ can send a $K$-type onto any other $K$-type in $\Hi_{\lambda_k\alpha}$. Thus the image of $T$ is the entire spherical principal series representation.
As in the previous cases, one can prove that the image of the Poisson transform is nonzero and that the West-constituent, where $\tau_{M_1,M_2,L}$ is, is not in the kernel of this map. Because of the two barriers, the action of $G$ cannot bring a $K$-type from the others constituents into the West-constituent. This means that the kernel of the Poisson transform is the quotient
\begin{equation*}
    \fracobl{\Hi_{\lambda_k\alpha}}{\displaystyle\sum_{\substack{m_1, m_2 \geq 0\\l\in \Z,~ |l|\geq +2k+2-m_1-m_2\\m_1 - m_2 +l = 0\\}} j_{-k-n}\big(\Hi^{m_1,m_2}(\C^{n})\otimes\Hi^l(\C)\big)}
\end{equation*}
\end{enumerate}

The four representations $\E_k$ we have found above are irreducible subquotients of $\Hi_{\lambda_k\alpha}$. Their unitarity is given in \cite[Diagram 3.15]{HoweTan}. We list their Langlands parameter in Table \ref{Table Ktypes of E_k - Ccase}. 

\bigskip
{\hspace{-0.5cm}\tiny\begin{tabular}{|c|c|c|c|c|}
\hline
    Case      & Minimal $K$-type              & $\delta$                    & values of $\nu$\\
    \hline
    $\circled{1}$     & $(k+1,0, \ldots, 0,-(k+1),0)$ & $\begin{array}{cc}
        \big((k+1), 0,\ldots,0,-(k+1),0\big) & \text{if } n>2 \\
        \emptyset & \text{if } n=2
    \end{array}$    & $\pm \left(\frac{n}{2}-1\right)\alpha$ if $n>2$\\
    \hline 
    $\circled{2}$     & $\begin{array}{cc}
        (0, \ldots, 0,-(k+1),k+1) & \text{if } k+1 \leq n-1 \\
        (\left\lfloor \frac{k+2-n}{2}\right\rfloor,0, \ldots, 0,-(k+1),\left\lceil \frac{k+n}{2}\right\rceil) & \text{if } k+1 > n-1
    \end{array}$ &    $\big(0, 0,\ldots,0,-(k+1),(k+1)/2\big)$ & $\pm \left(\frac{k}{2}+\frac{n}{2}-\frac{1}{2}\right)\alpha$\\
    \hline
    $\circled{3}$     & $\begin{array}{cc}
        ((k+1),0, \ldots, 0,k+1) & \text{if } k+1 \leq n-1 \\
        ((k+1),0, \ldots, 0,-\left\lfloor \frac{k+2-n}{2}\right\rfloor,\left\lceil \frac{k+n}{2}\right\rceil) & \text{if } k+1 > n-1
    \end{array}$           & $\big((k+1), 0,\ldots,0,(k+1)/2\big)$ & $\pm \left(\frac{k}{2}+\frac{n}{2}-\frac{1}{2}\right)\alpha$\\
    \hline
    $\circled{4}$     & $\triv_K $            & $\triv_M$                 & $i\lambda_k$\\
    \hline
\end{tabular}
\captionof{table}{Langlands parameters of $\E_k$ when $G = \SU(n,1)$}\label{Table Ktypes of E_k - Ccase}}
To compute the entries of Table \ref{Table Ktypes of E_k - Ccase}, let $\mu_{m_1,m_2,l}$ be the highest weight of $\Hi^{m_1,m_2}(\C^{n})\otimes\Hi^l(\C)$. The minimal $K$-type is obtained after minimising the Vogan norm of their highest weight in $\E_k$:
\begin{equation*}
    \rVert\mu_{m_1,m_2,l}\lVert_V = \<\mu_{m_1,m_2,l}+2\rho_K,\mu_{m_1,m_2,l}+2\rho_K\>
\end{equation*}
where $\rho_\Ak$ is the half sum of positive roots in $S_{\Ak_\C}^+$ (see Appendix \ref{Section Plancherel density Ccase} for the definition). 
Computations show that $\rVert\mu_{m_1,m_2,l}\lVert_V$ is minimal when $(m_1 + n-1)^2 + (m_2 + n-1)^2 + l^2$ is minimal.
\newline
To find $\delta$, one can use \cite[Theorem 4.4]{Bald1}. We show the reasoning in case $\circled{1}$. The minimal $K$-type $\tau_{\min{}}$ has highest weight \begin{equation*}
    (k+1) \epsilon_1 - (k+1) \epsilon_n
\end{equation*} 

\begin{itemize}
    \item[]Suppose $n>2$: The branching rules imply that \begin{equation*}
    \delta \in \hat{M}(\tau_{\min{}}) \Leftrightarrow \mu_\delta = a \epsilon_2 + b \epsilon_n - \frac{a+b}{2} (\epsilon_{n+1} + \epsilon_1) ~,
\end{equation*}
where $0\leq a \leq k+1$ and $0\leq -b \leq k+1$. But the minimal $K$-type of $\Ind_{M}^K(\delta)$ is $\tau_{\min{}}$ only when $a = k+1$ and $b=-k-1$. So\begin{equation*}
    \mu_\delta = (k+1) \epsilon_2 + (k+1) \epsilon_n
\end{equation*}
    \item[]Suppose $n=2$: The difference with the case $n>2$ is that there is no integer between 1 and $n$. Here 
    \begin{equation*}
    \delta \in \hat{M}(\tau_{\min{}}) \Leftrightarrow \mu_\delta = a \epsilon_2 + \frac{-a}{2}  (\epsilon_{1} + \epsilon_3) ~,
\end{equation*}
where $-(k+1)\leq a \leq k+1$. One can prove that $\mu' = a\epsilon_1 -a\epsilon_3$ is always the highest weight of a $K$-type in $\Ind_{M}^K(\delta)$ with a smaller Vogan norm than $\tau_{\min{}}$. So we get then discrete series representation. The Blattner parameter of the discrete series is the highest weight $\mu_{\tau_{\min{}}}$ of its minimal $K$-type $\tau_{\min{}}$. Its Harish-Chandra parameter is \begin{equation*}
    \Lambda_k = \mu_{\tau_{\min{}}} +2\rho_\Ak -\rho_\g = (k+1) \epsilon_1 - k\epsilon_2 -\epsilon_3 ~. 
\end{equation*}
\end{itemize}
To find $\nu$, one has to compare the infinitesimal characters of $\Ind_{MAN}^G(\triv\otimes e^{i\lambda_k\alpha}\otimes \triv)$ and of $\Ind_{MAN}^G(\delta \otimes e^{i\nu}\otimes \triv)$. They have to coincide up to the action of the Weyl group of $(\g_\C, \Ak_\C)$. For Case $\circled{1}$, one gets respectively the infinitesimal characters
\begin{equation*}
    \left(\frac{n}{2} +k\right) (\epsilon_1 - \epsilon_{n+1}) + \frac{1}{2} \sum_{j=2}^n (n-2i +2) \epsilon_i 
\end{equation*}
and 
\begin{equation*}
    \nu_\alpha (\epsilon_1 - \epsilon_{n+1}) + \left(\frac{n}{2} +k \right) \epsilon_2 + \frac{1}{2} \sum_{j=3}^{n-1} (n-2i +2) \epsilon_i +  \left(-\frac{n}{2} -k \right)\epsilon_n
\end{equation*}
The (complex) Weyl group permutes the $\pm\epsilon_i$'s, so $\nu = \pm\left(\frac{n}{2} - 1\right)\alpha$.
The theorem follows from similar computations in other cases.
\end{proof}

Since the multiplicity of $\tau$ in $\Hi^\delta_\nu$ is 1 (see for instance \cite{KoornwMultfree}) for these two first cases, the residue representation can be identified as the unique irreducible subquotient of $\Hi^\delta_\nu$ containing $\tau$.

\subsection{Case of $\Sp(n,1)$} \label{section E_k study - Qcase}

Let $G = \Sp(n,1)$. We recall that $K = \Sp(1)\times \Sp(n)$. Up to equivalence there a unique representation of $\Sp(1) = \SU(2)$ of dimension $j+1$ for all nonnegative integer $j$. Denote this representation by $\theta_j$, acting on the space $V_1^j$. Denote by $V^{m_1,m_2}_p$ the irreducible representation of $\Sp(p)$ with highest weight $(m_1, m_2, 0, \ldots,0)$ where $m_1\geq m_2 \geq 0$ (see \cite[Theorem 6 page 327]{Zhelobenko}). \newline As $\sigma$ is trivial, the poles are the points $\lambda_k = \pm i(n + \frac{1}{2}+k)$ with $k\in \N$.
The spherical principal series representation decomposes over $K$ as follows: 
\begin{equation}
    \Hi \simeq \sum P^*\big(V^{m_1,m_2}_p\otimes V^l_1\big)
\end{equation}
where $P := P_{m_1,m_2,l}$ is the projector in $\Hi_{\lambda_k\alpha}$ on the $K$-isotopic component isomorphic to $V^{m_1,m_2}_p\otimes V^l_1$ and the sum is over the following set : \begin{equation}\label{set of conditions for K types in Sp}
    \begin{array}{cc}
        m_1\geq m_2\geq0, & l\geq0, \\
        m_1 + m_2 \geq l, & m_1 + m_2 +l\in 2\Z~.
    \end{array}
\end{equation}
These conditions imply that a fixed pair $(m,l) = (m_1+m_2,l)$ represents a fiber of $K$-types $V^{m_1,m_2}_p\otimes V^l_1$, one for each value of $m_1-m_2$.
The highest weight of $V^{m_1,m_2}_p\otimes V^l_1$ can be computed as  $m_1\epsilon_1 + m_2\epsilon_2 + |l|\epsilon_{n+1}$ with respect to the sets of roots in \eqref{Section Plancherel density Qcase}. Let $\tau$ be $V^{a_1,a_2}_p\otimes V^{a_n+1}_1$. From now on we call this representation $\tau_{a_1,a_2,a_{n+1}}$. 

\begin{proof}[Langlands parameters of $\E_k$]

To understand the composition series of the representation $\E_k$, we have now to know how $\p$ acts on its $K$-types. This action is given in \cite[Lemma 5.3]{HoweTan}. The diagram 5.18 in that paper give us the different cases.

One can prove that the action of $\g$ can bring one $K$-type in each fiber to every other in the same fiber. Then we can follow the method used for $G = \SO(2n,1)$ or $\SU(n,1)$ putting the $K$-types of $\Hi_{\lambda_k\alpha}$ in the same two-dimensional space corresponding to the points of coordinates $(m_1 + m_2,l)$. Notice that here each point in Figure \ref{quaternKtypes} corresponds then to a fiber of $K$-types. We refer to \cite[Lemma 5.4]{HoweTan} for more details.
Figure \ref{quaternKtypes} illustrates the computation of the residue representation $\E_k$. Two barriers cross the set of $K$-types. 
These separate the space $\Hi_{\lambda_k\alpha}$ in  three constituent. The figure shows the three cases where the $K$-type $\tau_{a_1,a_2,a_{n+1}}$ is in each constituent. The arguments are the same as in the complex case. Naming the constituent North, West ant East according to their position, one gets the following equivalences: 
{\begin{equation}
    \E_k \simeq\begin{cases}
           \text{East-constituent}  &\text{if } a_{n+1}\leq -2k-4+a_1+a_2\\
            \fracobl{\text{North-East-constituents}}{\text{East-constituent}}  &\text{if } a_{n+1}\geq |-2k-2+a_1+a_2|\\
        \fracobl{\Hi_{\lambda_k\alpha}}{\text{North-East-constituents}}     &\text{if } a_{n+1}\geq 2k+2-a_1-a_2
    \end{cases}
\end{equation}}

\begin{center}
\begin{tikzpicture}[scale=0.35]
\begin{scope}[xshift=5cm] 
\draw (5,-5) node{\textbullet} node{$\bigcirc$};
\draw (5,-5.25) node[below] {$\tau_{a_1,a_2,a_{n+1}}$};
\fill [pattern=north east lines,pattern color=light-gray] (27,-4.5) rectangle (29,-5.5);
\draw (27,-4.5) rectangle (29,-5.5);
\draw (28,-8) node[above] {Kernel of the};
\draw (28,-8.25) node {Poisson transform};
\draw (15,-4.5) rectangle (17,-5.5);
\draw [very thick, light-gray, -] (15,-5) to (17,-5);
\draw (16,-8) node[above] {Image of $T$};
\end{scope}
\draw [->] (0,0) to (11,0);
\draw [->] (0,0) to (0,8);
\draw (11,0) node[below right] {\tiny $m$};
\draw (0,8) node[above left] {\tiny $l$};
\draw [-] (0,0) to (8,8);\draw [-,dashed] (10.5,10.5) to (8,8); 
\foreach \Point in {(0,0), (1,1), (2,2), (2,0), (3,3),  (3,1),  (4,4), (4,2), (4,0), (5,3), (5,1), (5,5),(6,4), (6,2), (6,0),(6,6),(7,3), (7,1), (7,5),(7,7),(8,4), (8,2), (8,0), (8,6),(8,8),(9,3), (9,1), (9,5),(9,7),(9,9),(10,4), (10,2), (10,0), (10,6),(10,8), (10,10) }{
    \node at \Point {$\bullet$};}
\draw (8,2) node{\textbullet} node{$\bigcirc$};
\draw [thick,-] (11,5) to (6,0);\draw [thick,->] (8.5,2.5) to++ (0.75,-0.75);\draw (6,-0.25) node[below] {\textbf{\tiny $2k+4$}};
\draw [thick,-] (0,4) to (4,0);\draw [thick,->] (2.5,1.5) to++ (0.75,0.75);\draw (0,4) node[left] {\textbf{\tiny $2k+2$}};
\draw [very thick, light-gray] (11,5.5) to (5, -0.5) to (11,-0.5) ;

\begin{scope}[xshift=17 cm] 
\draw [->] (0,0) to (11,0);
\draw [->] (0,0) to (0,8);
\draw (11,0) node[below right] {\tiny $m$};
\draw (0,8) node[above left] {\tiny $l$};
\draw [-] (0,0) to (8,8);\draw [-,dashed] (10.5,10.5) to (8,8); 
\foreach \Point in {(0,0), (1,1), (2,2), (2,0), (3,3),  (3,1),  (4,4), (4,2), (4,0), (5,3), (5,1), (5,5),(6,4), (6,2), (6,0),(6,6),(7,3), (7,1), (7,5),(7,7),(8,4), (8,2), (8,0), (8,6),(8,8),(9,3), (9,1), (9,5),(9,7),(9,9),(10,4), (10,2), (10,0), (10,6),(10,8), (10,10) }{
    \node at \Point {$\bullet$};}
\draw [thick,-] (11,5) to (6,0);\draw [thick,->] (8.5,2.5) to++ (0.75,-0.75);\draw (6,-0.25) node[below] {\textbf{\tiny $2k+4$}};
\draw [thick,-] (0,4) to (4,0);\draw [thick,->] (2.5,1.5) to++ (0.75,0.75);\draw (0,4) node[left] {\textbf{\tiny $2k+2$}};
\draw (7,5) node{\textbullet} node{$\bigcirc$};
\draw [very thick, light-gray] (11,-0.5) to (4,-0.5) to++ (-2.5,2.5) to++ (9,9);
\draw[-] (11,5.7) to (5, -0.3) to (11,-0.3);
\fill [pattern=north east lines,pattern color=light-gray] (11,5.7) to (5, -0.3) to (11,-0.3);
\end{scope}

\begin{scope}[xshift=34 cm] 
\draw [->] (0,0) to (11,0);
\draw [->] (0,0) to (0,8);
\draw (11,0) node[below right] {\tiny $m$};
\draw (0,8) node[above left] {\tiny $l$};
\draw [-] (0,0) to (8,8);\draw [-,dashed] (10.5,10.5) to (8,8); 
\foreach \Point in {(0,0), (1,1), (2,2), (2,0), (3,3),  (3,1),  (4,4), (4,2), (4,0), (5,3), (5,1), (5,5),(6,4), (6,2), (6,0),(6,6),(7,3), (7,1), (7,5),(7,7),(8,4), (8,2), (8,0), (8,6),(8,8),(9,3), (9,1), (9,5),(9,7),(9,9),(10,4), (10,2), (10,0), (10,6),(10,8), (10,10) }{
    \node at \Point {$\bullet$};}
\draw [thick,-] (11,5) to (6,0);\draw [thick,->] (8.5,2.5) to++ (0.75,-0.75);\draw (6,-0.25) node[below] {\textbf{\tiny $2k+4$}};
\draw [thick,-] (0,4) to (4,0);\draw [thick,->] (2.5,1.5) to++ (0.75,0.75);\draw (0,4) node[left] {\textbf{\tiny $2k+2$}};
\draw (1,1) node{\textbullet} node{$\bigcirc$};
\draw [very thick, light-gray] (11,-0.5) to (-1,-0.5) to++ (11.5,11.5);
\draw [-] (11,-0.3) to (4,-0.3) to++ (-2.3,2.3) to++ (8.7,8.7);
\fill [pattern=north east lines,pattern color=light-gray] (11,-0.3) to (4,-0.3) to++ (-2.3,2.3) to++ (8.7,8.7);
\end{scope}
\end{tikzpicture}
\captionof{figure}{$K$-types in $\Hi_{\lambda_k\alpha}$ and decomposition series for $k=1$}\label{quaternKtypes}
\end{center}

The three representations $\E_k$ we have found above are irreducible subquotients of $\Hi_{\lambda_k\alpha}$. Their unitarity is given in \cite[Diagram 3.15]{HoweTan}. We find the following results using Proposition 3.1 in \cite{Bald2}:

\medskip
{\hspace{-0.7cm}\begin{tabular}{|c|c|c|c|c|}
\hline
    Case      & Minimal $K$-type              & $\delta$                    & values of $\nu$\\
    \hline
    $\circled{1}$     & $(k+2)\epsilon_1+(k+2)\epsilon_2$ & $
        (k+2)e_3+ (k+2)e_4$& $\pm \left(n-\frac{3}{2}\right)\alpha$\\
    \hline 
    $\circled{2}$     & 
        $(k+1)\epsilon_1+(k+1)\epsilon_{n+1}$ &    $\begin{array}{cc}
            (k+1)e_3+\frac{k+1}{2}(e_1-e_2)& \text{ if } k \leq 2n-4 \\
            \emptyset & \text{ if } k > 2n-4
        \end{array}$ & $\pm \left(\frac{k}{2}+n\right)\alpha$~ if  $k \leq 2n-4$\\
    \hline
    $\circled{3}$     & $\triv_K $            & $\triv_M$                 & $i\lambda_k$\\
    \hline
\end{tabular}
\captionof{table}{Langlands parameters of $\E_k$ when $G = \Sp(n,1)$}\label{Table Ktypes of E_k - Qcase}}

We indicate how to compute the entries of this table. In case $\circled{2}$, for $k > 2n-4$, one cannot find a representation $\delta$ following the conditions of the Langlands parameters. We conclude that $\E_k$ is in the discrete series in these cases. Their Blattner parameter is the highest weight of the minimal $K$-type, which is $(k+1)\epsilon_1+(k+1)\epsilon_{n+1}$. The Harish-Chandra parameter of the discrete series is \begin{equation*}
    (k+n)\epsilon_1 + (k+3-n)\epsilon_{n+1} + \sum_{j=2}^n 2(n-i+1)\epsilon_i~. 
\end{equation*}
\end{proof}

\begin{Rem}
One can verify that Theorem 1 in \cite{Parthasarathy} gives the same conclusion as ours: there are discrete series representations in case $\circled{2}$ for $k > 2n-4$. 
\end{Rem}

\subsection{Case of $F_4$} \label{section E_k study - Ocase}
The paper of Johnson \cite{JohnF4} gives the results we need in this case. 
Let $G$ be the exceptional Lie group $F_4$. We recall that here $K = \Spin(9)$ and $M = \Spin(7)$. As $\sigma$ is trivial, the poles of the meromorphically extended resolvent are the points $\lambda_k = \pm i(\frac{11}{2}+k)$ with $k\in \N$.

The $K$-types of $\Hi_{\lambda_k\alpha}$ are the $V^{p,q}$ with $p\geq q\geq0$ and $p + q\in 2\Z$ (see \cite[Theorem 3.1]{JohnF4}), with highest weight\begin{equation*}
    \mu_{p,q} =  \frac{p}{2}\epsilon_1 + \frac{q}{2}\epsilon_2 + \frac{q}{2}\epsilon_3 + \frac{q}{2}\epsilon_4
\end{equation*}
with respect to the sets of roots in Appendix \eqref{Section Plancherel density Qcase}. Let $\tau$ be $V^{a,b}$, for $a\geq b\geq0$ and $a + b \in 2\Z$. In the following, we call this representation $\tau_{a,b}$. 

We can follow the method of the cases $G = \SO(2n,1)$ or $\SU(n,1)$, putting the $K$-types of $\Hi_{\lambda_k\alpha}$ in the same two-dimensional space corresponding to the points of coordinates $(p,q)$.
Figure \ref{OctoKtypes} illustrates the computations of $\E_k$. There are again two barriers. 
They separate the space $\Hi_{\lambda_k\alpha}$ in  three constituents. The figure describes the three cases where the $K$-type $\tau_{p,q}$ is in different constituents. The arguments are the same as in the complex case.

Naming the constituent North, West ant East according to their position, one gets the following equivalences: 

{\begin{equation}
    \E_k \simeq\begin{cases}
           \text{East-constituent}  &      \text{if } a_{n+1}\leq -2k-4+a_1+a_2\\
            \fracobl{\text{North-East-constituents}}{\text{East-constituent}}  &      \text{if } a_{n+1}\geq |-2k-2+a_1+a_2|\\
        \fracobl{\Hi_{\lambda_k\alpha}}{\text{North-East-constituents}}     &      \text{if } a_{n+1}\geq 2k+2-a_1-a_2
    \end{cases}
\end{equation}}

\begin{center}
\begin{tikzpicture}[scale=0.35]
\begin{scope}[xshift=5cm] 
\draw (5,-5) node{\textbullet} node{$\bigcirc$};
\draw (5,-5.25) node[below] {$\tau_{a,b}$};
\fill [pattern=north east lines,pattern color=light-gray] (27,-4.5) rectangle (29,-5.5);
\draw (27,-4.5) rectangle (29,-5.5);
\draw (28,-8) node[above] {Kernel of the};
\draw (28,-8.25) node {Poisson transform};
\draw (15,-4.5) rectangle (17,-5.5);
\draw [very thick, light-gray, -] (15,-5) to (17,-5);
\draw (16,-8) node[above] {Image of $T$};
\end{scope}
\draw [->] (0,0) to (11,0);
\draw [->] (0,0) to (0,8);
\draw (11,0) node[below right] {\tiny $p$};
\draw (0,8) node[above left] {\tiny $q$};
\draw [-] (0,0) to (8,8);\draw [-,dashed] (10.5,10.5) to (8,8); 
\foreach \Point in {(0,0), (1,1), (2,2), (2,0), (3,3),  (3,1),  (4,4), (4,2), (4,0), (5,3), (5,1), (5,5),(6,4), (6,2), (6,0),(6,6),(7,3), (7,1), (7,5),(7,7),(8,4), (8,2), (8,0), (8,6),(8,8),(9,3), (9,1), (9,5),(9,7),(9,9),(10,4), (10,2), (10,0), (10,6),(10,8), (10,10) }{
    \node at \Point {$\bullet$};}
\draw (10,0) node{\textbullet}node{$\bigcirc$};
\draw [thick,-] (11,3) to (8,0);\draw [thick,->] (9.5,1.5) to++ (0.75,-0.75);\draw (8,-0.25) node[below] {\textbf{\tiny $2k+8$}};
\draw [thick,-] (0,2) to (2,0);\draw [thick,->] (1.5,0.5) to++ (0.75,0.75);\draw (0,2) node[left] {\textbf{\tiny $2k+2$}};
\draw [very thick, light-gray] (11,3.5) to (7, -0.5) to (11,-0.5);

\begin{scope}[xshift=17 cm] 
\draw [->] (0,0) to (11,0);
\draw [->] (0,0) to (0,8);
\draw (11,0) node[below right] {\tiny $p$};
\draw (0,8) node[above left] {\tiny $q$};
\draw [-] (0,0) to (8,8);\draw [-,dashed] (10.5,10.5) to (8,8); 
\foreach \Point in {(0,0), (1,1), (2,2), (2,0), (3,3),  (3,1),  (4,4), (4,2), (4,0), (5,3), (5,1), (5,5),(6,4), (6,2), (6,0),(6,6),(7,3), (7,1), (7,5),(7,7),(8,4), (8,2), (8,0), (8,6),(8,8),(9,3), (9,1), (9,5),(9,7),(9,9),(10,4), (10,2), (10,0), (10,6),(10,8), (10,10) }{
    \node at \Point {$\bullet$};}
\draw [thick,-] (11,3) to (8,0);\draw [thick,->] (9.5,1.5) to++ (0.75,-0.75);\draw (8,-0.25) node[below] {\textbf{\tiny $2k+8$}};
\draw [thick,-] (0,2) to (2,0);\draw [thick,->] (1.5,0.5) to++ (0.75,0.75);\draw (0,2) node[left] {\textbf{\tiny $2k+2$}};
\draw (7,5) node{\textbullet}node{$\bigcirc$};
\draw [very thick, light-gray] (11,-0.5) to (2,-0.5) to++ (-1.5,1.5) to++ (10,10);
\draw [-] (11,3.7) to (7, -0.3) to (11,-0.3);
\fill [pattern=north east lines,pattern color=light-gray] (11,3.7) to (7, -0.3) to (11,-0.3);
\end{scope}

\begin{scope}[xshift=34 cm] 
\draw [->] (0,0) to (11,0);
\draw [->] (0,0) to (0,8);
\draw (11,0) node[below right] {\tiny $p$};
\draw (0,8) node[above left] {\tiny $q$};
\draw [-] (0,0) to (8,8);\draw [-,dashed] (10.5,10.5) to (8,8); 
\foreach \Point in {(0,0), (1,1), (2,2), (2,0), (3,3),  (3,1),  (4,4), (4,2), (4,0), (5,3), (5,1), (5,5),(6,4), (6,2), (6,0),(6,6),(7,3), (7,1), (7,5),(7,7),(8,4), (8,2), (8,0), (8,6),(8,8),(9,3), (9,1), (9,5),(9,7),(9,9),(10,4), (10,2), (10,0), (10,6),(10,8), (10,10) }{
    \node at \Point {$\bullet$};}
\draw [thick,-] (11,3) to (8,0);\draw [thick,->] (9.5,1.5) to++ (0.75,-0.75);\draw (8,-0.25) node[below] {\textbf{\tiny $2k+8$}};
\draw [thick,-] (0,2) to (2,0);\draw [thick,->] (1.5,0.5) to++ (0.75,0.75);\draw (0,2) node[left] {\textbf{\tiny $2k+2$}};
\draw (0,0) node{\textbullet}node{$\bigcirc$};
\draw [very thick, light-gray] (11,-0.5) to (-1,-0.5) to++ (11.5,11.5);
\draw [-] (11,-0.3) to (2,-0.3) to++ (-1.3,1.3) to++ (9.7,9.7);
\fill [pattern=north east lines,pattern color=light-gray] (11,-0.3) to (2,-0.3) to++ (-1.3,1.3) to++ (9.7,9.7);
\end{scope}

\end{tikzpicture}
\captionof{figure}{$K$-types in $\Hi_{\lambda_k\alpha}$ and decomposition series for $k=0$}\label{OctoKtypes}
\end{center}

The three representations $\E_k$ we have found above are irreducible subquotients of $\Hi_{\lambda_k\alpha}$. We find the following results using \cite[Theorem 3.4]{Bald1}:

{\begin{center}\begin{tabular}{|c|c|c|c|c|}
\hline
    Case      & Minimal $K$-type              & $\delta$                    & values of $\nu$\\
    \hline
    $\circled{1}$     & $(k+4)\epsilon_1$ & $\frac{(k+4)}{4}(3 \epsilon_1+ \epsilon_{2}+\epsilon_{3}+\epsilon_{4})$ & $\pm \frac{1}{2}\left(k+1\right)\alpha$\\
    \hline 
    $\circled{2}$     & 
        $\frac{\lfloor \frac{3k+5}{2}\rfloor}{2}\epsilon_1+\frac{\lceil \frac{k-1}{2}\rceil}{2}(\epsilon_{2}+\epsilon_{3}+\epsilon_{4})$ &    $\frac{(k+4)}{4}(3 \epsilon_1+ \epsilon_{2}+\epsilon_{3}+\epsilon_{4})$ & $\frac{1}{2} \left(k+10\right)\alpha$\\
    \hline
    $\circled{3}$     & $\triv_K $            & $\triv_M$                 & $i\lambda_k$\\
    \hline
\end{tabular}
\captionof{table}{Langlands parameters of $\E_k$ when $G = \F_4$}\label{Table Ktypes of E_k - Ocase}\end{center}}
In the table $\lfloor \cdot\rfloor$ denotes the integer part and $\lceil \cdot\rceil$ is the upper integer part. 
Here the method is exactly the same as when $G = \Spin(2n,1)$. One has just to take care of the embedding of $M$ in $K$ which is not standard (see \cite[section 6]{Bald1}). We used \cite{WeylgroupF4} to know how the (complex) Weyl group acts on the infinitesimal characters of the principal series.

\section{Wave front set of the residue representations}

In this section, we continue to assume that $(\tau,\Hi_\tau)$ contains the trivial representation of $M$. The residue representations $\E_k$ are those listed in Theorem \ref{Thmintro2} and defined in section \ref{section residue repr}. The purpose of this part is to compute the wave front set $\WF(\E_k)$ of these representations. For this, we need some results on the Gelfand-Kirillov dimension \cite{VoganGelfKirill} and on the nilpotent orbits in semisimple Lie algebras \cite{CollinMcGov}. 
Proposition 2.4 in \cite{HoweWFS} tells us that the wave front set of $\E_k$ is equal to a closed union of nilpotent orbits of $\g$. 
To know which occurs, we use the Gelfand-Kirillov dimension of $\E_k$, computed by \cite[Theorem 1.2]{VoganGelfKirill} in terms of $K$-types. 
This dimension is the half of the dimension of the wave front set see as a nilpotent orbit (see \cite{VoganGelfKirill, Rossman, BarbaschVogan}). Combining this information with the list of nilpotent orbits in semisimple Lie algebra and their dimensions in  \cite{CollinMcGov}, we get the results.

\subsection{Generalities}

The definition of the wave front set of a representation can be found in \cite[page 118]{HoweWFS}. The wave front sets are closed conical sets in $T^*G$. For Lie groups, they can be seen as invariant sets of $\g^*$ under the coadjoint action $\Ad^*$ of $G$.
For semisimple Lie groups, $\g$ and $\g^*$ can be identified by the Killing form. Then $\Ad^*$ invariant subsets of $\g^\ast$ can be seen as $\Ad$ invariant subsets of $\g$. More precisely, the wave front set of a representation can be identified with the closure of a union of nilpotent orbits under the adjoint action in $\g$ (see \cite[Proposition 2.4]{HoweWFS}). It will be denoted by $\WF(\pi)$.

\subsection{Case of $\SO(2n,1)$}

In this section, we prove the results about wave front set stated in Theorem \ref{Thmintro2} for $G = \SO(2n,1)$. First we compute the Gelfand-Kirillov dimension of the Harish-Chandra module of $\E_k$ defined in section \ref{section residue repr}.

\begin{Lemma}\label{Gelf.-Kir.dim/real}
The Gelfand-Kirillov dimension of the residue representation $\E_k$ is $\begin{cases}{cl}
    0 & \text{ if }N\leq k \\
    2n-1 & \text{ if }N > k
\end{cases}$
\end{Lemma}

\begin{proof}
First of all, if $N\leq k$, the residue representation is finite-dimensional, so the Gelfand-Kirillov dimension is 0. 
For the case $N > k$, the residue representation is infinite-dimensional. One can compute the eigenvalue of each $K$-type $\tau_m$ in $\Hi_{\lambda_k\alpha}$ for the Casimir operator $\Omega_\Ak$ of $\Ak$ (using for example \cite[Proposition 10.6]{Hall}): 
\begin{equation}
    \tau_m(\Omega_\Ak) = \Big((m+n-1)^2 - (n-1)^2\Big) \Id
\end{equation}
Knowing that the dimension of the space of harmonic polynomials of homogeneous degree $m$ is $d_m = \binom{m+2n-1}{2n-1} - \binom{m+2n-3}{2n-1}$ (see \cite{Zhelobenko} for example), we compute the sum $N_{\E_k}(t)$ of the dimensions $d_m$ of $\tau_m$ until its eigenvalue exceed a fixed real number $t^2$. This sum depends on $t$ as follows: 
\begin{equation}
    N_{\E_k}(t) = \sum_{m=k+1}^{N_t} d_m = \binom{N_t+2n-1}{2n-1} + \binom{N_t+2n-2}{2n-1} - \binom{k+2n-2}{2n-1} - \binom{k+2n-1}{2n-1}
\end{equation}
where $N_t +n-1= \lfloor t\rfloor$ and if $t$ goes to infinity, we have : 
\begin{equation}
    N_{\E_k}(t) \overset{\infty}{\sim} \frac{(N_t+2n-1)!}{N_t!} \overset{\infty}{\sim} t^{2n-1}
\end{equation}
Hence the Gelfand-Kirillov dimension is equal to $2n-1$ (see \cite[Theorem 1.2]{VoganGelfKirill}).
\end{proof}

\begin{Lemma}\label{Lemma number of orbits in Spin}
 Let \begin{equation}
    \g = \m \oplus \Aa \oplus \g_\alpha \oplus \g_{-\alpha}
\end{equation}be the restricted root space decomposition of $\g$. Then there are 2 nilpotent orbits in $\g$, namely the zero orbit and the orbit generated by any non-zero element of $\g_\alpha$. 
\end{Lemma}

\begin{proof}
The non-zero elements of $\g_\alpha$ are conjugated by the elements of $MA$. Moreover, $\g_{-\alpha} = \theta(\g_\alpha) = \Ad(k_\theta)(\g_\alpha)$ for a suitable element $k_\theta \in K$. Then all nilpotent elements in the restricted root spaces above are conjugate, except 0 which forms an orbit alone. 

This can also be found using Theorem 9.3.4 in \cite{CollinMcGov}. The two only possible Young diagrams for $\soe(2n,1)$ are : 

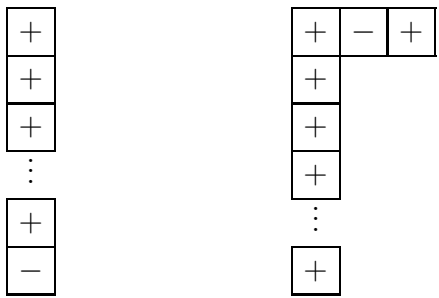
\ytableausetup{centertableaux}
\begin{center}\begin{ytableau}
   + \\
   + \\
   + \\
   \none [\vdots] \\
   + \\
   -
\end{ytableau}\hspace{30mm}
    \begin{ytableau}
      +  & -& +\\
   + & \none  & \none \\
   +& \none  & \none \\
   + & \none  & \none \\
   \none [\vdots] & \none  & \none \\
   + & \none  & \none 
\end{ytableau}
\end{center}
with $2n$ '+', which respectively correspond respectively to the 0 orbit and the orbit generated by $\g_\alpha$. 
\end{proof}

We are now able to prove the results about the wave front set of $\E_k$ in Theorem \ref{Thmintro2} for $G = \SO(2n,1)$. 


\begin{proof}[Wave front set of $\E_k$.]
We have just two cases. When $N\leq k$, $\E_k$ is finite-dimensional so the wave front set is the zero orbit. If $N>k$, as $\E_k$ is infinite-dimensional, only the nilpotent orbit generated by $\g_\alpha$ can correspond. 

This can be checked using the formula for the dimension of nilpotent orbits in $\g_\C$, the complexification of $\g$ (\cite[Corollary 6.1.4]{CollinMcGov}). In fact if $N>k$, the Gelfand-Kirillov dimension is $2n-1$ because of Lemma \ref{Gelf.-Kir.dim/real}. The dimension of the wave front set is then $4n-2$. Because of the corollary cited above, we have:
\begin{equation}
    4n-2 = \dim \WF(\E_k) = (2n+1)^2 - \frac{1}{2}\sum s_i^2 -\frac{1}{2} \sum_{\text{odd}} r_i 
\end{equation}
where $s_i = \left|\left\{j~|~d_j\geq i \right\}\right|$ and $r_i = \left|\left\{j~|~d_j = i \right\}\right|$ in a partition $[d_1, \ldots, d_k]$ of $2n+1$.\newline
Thus, only one complex nilpotent orbit has this dimension: the one matching to the partition $[3,1,1,\ldots,1]$ of $2n+1$. This corresponds indeed to the Young diagram of the (real) nilpotent orbit generated by $\g_\alpha$. 
\end{proof}

\begin{minipage}{0.6\textwidth}
\subsection{Case of $\SU(n,1)$}In this section, we prove the results about wave front set in Theorem \ref{Thmintro2} for $G = \SU(n,1)$. First of all, we compute the Gelfand-Kirillov dimension of the Harish-Chandra module of $\E_k$ defined in section \ref{section residue repr}. Recall that there are 4 possibilities for the residue representation (see \ref{section E_k study - Ccase}). We label each case as in Figure \ref{cases for E_k - Complex case}. 
\end{minipage}
\begin{minipage}{0.05\textwidth}
\end{minipage}
\begin{minipage}{0.35\textwidth}
\begin{center}
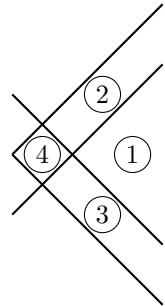

\begin{tikzpicture}[scale=0.4]
\draw [-,thick] (0,0) to (5,5);
\draw [-,thick] (0,0) to (5,-5);
\draw [-,thick] (0,-2) to (5,3);
\draw [-,thick] (0,2) to (5,-3);
\node[draw,circle] (D)at(4,0) {\footnotesize 1};
\node[draw,circle] (D)at(3,2) {\footnotesize 2};
\node[draw,circle] (D)at(3,-2) {\footnotesize 3};
\node[draw,circle] (D)at(1,0) {\footnotesize 4};
\end{tikzpicture}
\captionof{figure}{Cases for $\E_k$ - $\SU(n,1)$}\label{cases for E_k - Complex case}
\end{center}
\end{minipage}

\begin{Lemma}\label{Gelf.-Kir.dim/complex case}
The Gelfand-Kirillov dimension of the residue representation $\E_k$ is $\left\{\begin{array}{cl}
    2n-1 & \text{in case }\emph{\circled{1}}\\
    n & \text{in case }\emph{\circled{2}}\\
    n & \text{in case }\emph{\circled{3}}\\
    0 & \text{in case }\emph{\circled{4}}
\end{array}\right.$.
\end{Lemma}
\begin{proof}
In case $\circled{4}$, the representation $\E_k$ is finite-dimensional. The Gelfand-Kirillov dimension is then 0. For the three other cases, we need some computations.

First of all, one can compute the eigenvalue of each $K$-type $\tau_{m,l}$ in $\Hi_{\lambda_k\alpha}$ for the Casimir operator $\Omega_\Ak$ of $\Ak$ (using for example \cite[Proposition 10.6]{Hall}). Recall from \ref{section E_k study - Ccase} that \begin{equation}
    \tau_{m,l} \simeq \Hi^{\frac{m-l}{2},\frac{m+l}{2}}(\C^n)\times \Hi^l(\C)
\end{equation}
with highest weight $\mu_{m,l} = \frac{m-l}{2} \epsilon_1 - \frac{m+l}{2} \epsilon_n + l \epsilon_{n+1}$. Then one finds:
\begin{equation}\label{eigenvalue for Casimir - Complex case}
    \tau_{m,l}(\Omega_\Ak) = \frac{1}{2}\Big((m+n-1)^2 - (n-1)^2 +3l^2\Big) \Id
\end{equation}
The dimension of the space of $\tau_{m,l}$ follows for example from the Weyl dimension formula \cite[Theorem 5.84]{Knapp-LieGroups}: \begin{equation*}
    d_{m,l} = \binom{\frac{m-l}{2}+n-1}{n-1} \binom{\frac{m+l}{2}+n-1}{n-1} - \binom{\frac{m-l}{2}+n-2}{n-1} \binom{\frac{m+l}{2}+n-2}{n-1} 
\end{equation*}
Let $\delta({m,l,N}) := \binom{\frac{m-l}{2}+N-1}{n-1}\binom{\frac{m+l}{2}+N-1}{n-1}$. So \begin{equation}
    d_{m,l} = \delta({m,l,n}) - \delta({m,l,n-1})~.
\end{equation}
We compute the sum $N_{\E_k}(t)$ of the dimensions $d_m$ of $\tau_m$ until the eigenvalue \eqref{eigenvalue for Casimir - Complex case} exceeds a fixed real number $t^2$. This sum depends on $t$ and we want to know what is the equivalent power of $t$ to use \cite[Theorem 1.2]{VoganGelfKirill}. 
First one can see, that for a fixed $l$, the sum between two values $m_{\min{}}$ and $m_{\max{}}$ of $m$ is telescopic:
\begin{equation}\label{sum of dimensions - complex case}
    \sum_{m=m_{\min{}}}^{m_{\max{}}} d_{m,l} = \delta({m_{\max{}},l,n}) - \delta({m_{\min{}},l,n-1})
\end{equation}
The inequality given by the eigenvalue of the Casimir operator \eqref{eigenvalue for Casimir - Complex case} limits the area of points of $K$-types by an ellipsis. The goal is to know the behaviour in $t$ of the sum over th $K$-types inside the ellipsis when this $t$-ellipsis is going to infinity. 

\begin{minipage}{0.49\textwidth}
\underline{Case $\circled{1}$}: The sum \eqref{sum of dimensions - complex case} is telescopic on every (gray) line for a fixed $l$, so it is just the difference $\delta({m_{\max{}},l,n}) - \delta({m_{\min{}},l,n-1})$ where $m_{\min{}}$ and $m_{\max{}}$ are respectively the abscissa of the first (circled) and the last (squared) point on the line.
Let $l_{\max{}}$ be the maximal $l$-coordinate of the $K$-types that are indexed by the points in this area.
This number is roughly equal to the coordinate of the intersection between the ellipsis and the line "limiting" the representation.
Computing the intersection, we have $l_{\max{}}(t)$ is asymptotic at infinity to $\frac{t}{2}$. As $m_{\max{}}(l,t)$ is roughly on the ellipsis, we also have that $m_{\max{}}(l,t)$ is asymptotic at infinity to $t^2-3l^2$ on the line with ordinate $l$. We can forget the term $\delta({m_{\min{}},l,n-1})$ because it is a constant in $t$ so it will bring strictly lower powers of $t$ than $\delta({m_{\max{}},l,n})$ in the sum.
We get then (the symbol $\overset{\infty}{\sim}$ means that both sides have the same highest power of $t$) 
\end{minipage}
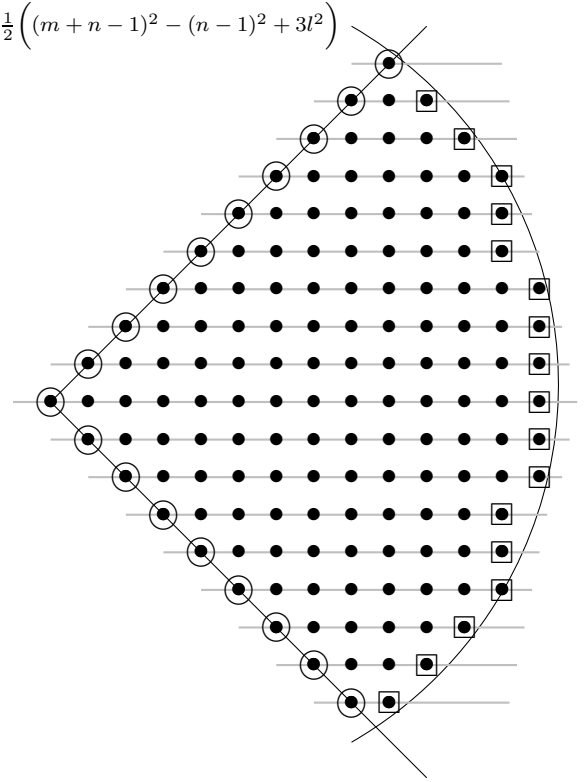
\begin{minipage}{0.5\textwidth}
\begin{tikzpicture}{scale = 0.5}
\draw (4,5) arc (60:-60:5.5);
\draw [-] (5,5) to (0,0) to (5,-5);
\node[left] at (4,5) {\tiny $t^2 = \frac{1}{2}\Big((m+n-1)^2 - (n-1)^2 +3l^2\Big)$};
\draw [-,thick, light-gray] (-0.5,0) to (7,0);
\draw [-,thick, light-gray] (0,0.5) to (6.9,0.5);
\draw [-,thick, light-gray] (0,-0.5) to (6.9,-0.5);
\draw [-,thick, light-gray] (0.5,1) to (6.8,1);
\draw [-,thick, light-gray] (0.5,-1) to (6.8,-1);
\draw [-,thick, light-gray] (1,1.5) to (6.6,1.5);
\draw [-,thick, light-gray] (1,-1.5) to (6.6,-1.5);
\draw [-,thick, light-gray] (1.5,2) to (6.5,2);
\draw [-,thick, light-gray] (1.5,-2) to (6.5,-2);
\draw [-,thick, light-gray] (2,2.5) to (6.4,2.5);
\draw [-,thick, light-gray] (2,-2.5) to (6.4,-2.5);
\draw [-,thick, light-gray] (2.5,3) to (6.3,3);
\draw [-,thick, light-gray] (2.5,-3) to (6.3,-3);
\draw [-,thick, light-gray] (3,3.5) to (6.2,3.5);
\draw [-,thick, light-gray] (3,-3.5) to (6.2,-3.5);
\draw [-,thick, light-gray] (3.5,4) to (6.1,4);
\draw [-,thick, light-gray] (3.5,-4) to (6.1,-4);
\draw [-,thick, light-gray] (4,4.5) to (6,4.5);
\foreach \ordo in {0,0.5,1,1.5,2,2.5,3} {\foreach \absci in {6,5.5,...,\ordo} {\node at (\absci,\ordo) {$\bullet$};}}
\foreach \ordo in {-2.5,-2,-1.5,-1,-0.5} {\foreach \absci in {6,5.5,...,-\ordo} {\node at (\absci,\ordo) {$\bullet$};}}
\foreach \absci in {5.5,5,...,3.5} {\node at (\absci,3.5) {$\bullet$};}
\foreach \absci in {5,4.5,4} {\node at (\absci,4) {$\bullet$};}
\foreach \absci in {5.5,5,...,3} {\node at (\absci,-3) {$\bullet$};}
\foreach \absci in {5,4.5,...,3.5} {\node at (\absci,-3.5) {$\bullet$};}
\foreach \absci in {4.5,4} {\node at (\absci,-4) {$\bullet$};}
\foreach \ordo in {-1,-0.5,0,0.5,1,1.5} {\node at (6.5,\ordo) {$\bullet$};\node at (6.5,\ordo) {$\square$};}
\foreach \ordo in {-2.5,-1.5,-2,2,2.5,3} {\node at (6,\ordo) {$\bullet$};\node at (6,\ordo) {$\square$};}
\foreach \Point in {(0,0), (1,1), (2,2),(3,3),(4,4), (1,-1), (2,-2),(3,-3),(4,-4),(0.5,0.5),(0.5,-0.5), (1.5,1.5), (2.5,2.5),(3.5,3.5),(4.5,4.5), (1.5,-1.5), (2.5,-2.5),(3.5,-3.5)}{
    \node at \Point {$\bullet$};\node at \Point {$\bigcirc$};}
\foreach \Point in {(5.5,3.5), (5,4), (5.5,-3), (5,-3.5), (4.5,-4)}{\node at \Point {$\bullet$};\node at \Point {$\square$};}
\end{tikzpicture}
\end{minipage}
\begin{equation*}
    N_{\E_k}(t) \overset{\infty}{\sim} \displaystyle\sum_{l=-l_{\max{}}(t)}^{l_{\max{}}(t)}  \delta({m_{\max{}}(l,t),l,n}) 
                \overset{\infty}{\sim} \sum_{|l|\leq t/2}   \frac{\prod_{\tiny j=1}^{\tiny n-1} \left(\frac{(m_{\max{}}(l,t)+n-1-j)^2}{4}-\frac{l^2}{4}\right)}{(n-1)!}
\end{equation*}
And replacing $m_{\max{}}$:
\begin{equation*}
    N_{\E_k}(t) \overset{\infty}{\sim} \sum_{0\leq l\leq t/2}   (t^2 -4l^2)^{n-1}  \overset{\infty}{\sim} t^{2n-1}
\end{equation*}
The Gelfand-Kirillov dimension follows from Theorem 1.2 in \cite{VoganGelfKirill}.

For the cases $\circled{2}$ and $\circled{3}$ (which are symmetric), the reasoning is exactly the same and give \begin{equation*}
    N_{\E_k} \overset{\infty}{\sim} t^n ~.
\end{equation*}
\end{proof}

The following Lemma gives us the different nilpotent orbits in the nilradical of $\g$.

\begin{Lemma}\label{definition of n1 and n2}
There are four nilpotent orbits in $\SU(n,1)$:
\begin{enumerate}
    \item the trivial orbit,
    \item the one generated by $\g_{\alpha/2}$, of dimension $4n-2$,
    \item the one generated by $n_1 = i\begin{pmatrix} 1 &0&-1\\0&$\textbf{$0$}$ &0\\1&0&-1\end{pmatrix}$, of dimension $2n$,
    \item the one generated by $n_2 = i\begin{pmatrix} 1 &0&-1\\0&$\textbf{$0$}$ &0\\1&0&-1\end{pmatrix}$, of dimension $2n$.
\end{enumerate}
Here the \textbf{$0$} in the center of the matrix is the $(n-1)\times(n-1)$ zero matrix. 
\end{Lemma}

\begin{proof}
Using Theorem 9.3.3 in \cite{CollinMcGov}, one finds that there are four possible Young diagrams corresponding to the nilpotent orbit for $\su(n,1)$:
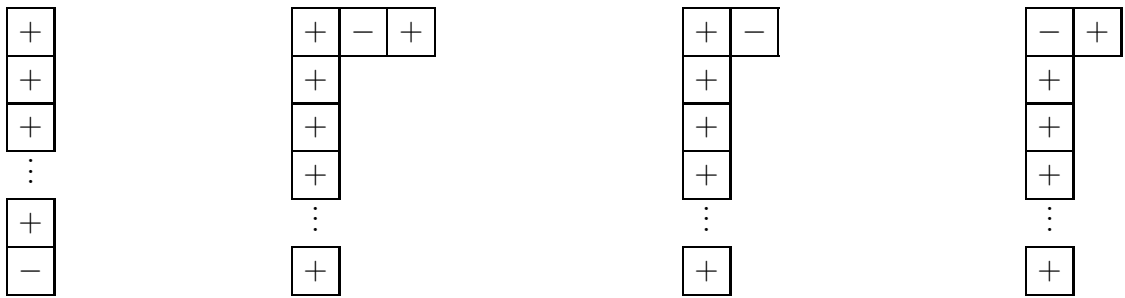
\begin{center}\begin{ytableau}
   + \\
   + \\
   + \\
   \none [\vdots] \\
   + \\
   -
\end{ytableau}\hspace{30mm}
\begin{ytableau}
      +  & -& +\\
   + & \none  & \none \\
   +& \none  & \none \\
   + & \none  & \none \\
   \none [\vdots] & \none  & \none \\
   + & \none  & \none 
\end{ytableau}
\hspace{30mm}
\begin{ytableau}
      +  & -\\
   + & \none   \\
   +& \none   \\
   + & \none   \\
   \none [\vdots] & \none   \\
   + & \none 
\end{ytableau}
\hspace{30mm}
\begin{ytableau}
      -  & +\\
   + & \none   \\
   +& \none   \\
   + & \none   \\
   \none [\vdots] & \none   \\
   + & \none 
\end{ytableau}
\end{center}
The first one corresponds to the zero orbit. The second one is the only one which is nilpotent of degree 2. It corresponds to the orbit generated by any element of $\g_{\alpha/2}$. In fact, the proof of the fact that $\g_{\alpha/2}$ and $\g_{-\alpha/2}$ are in the same orbit is exactly the same as the one in Lemma \ref{Lemma number of orbits in Spin}. \newline
The two last ones are nilpotent of degree 1. Computations show that $n_1$ and $n_2$ (defined in the Lemma) cannot be conjugated by an element of $K$. In particular, $K$ cannot change the last number on the last line and the last column. The conjugation by $A$ is a multiplication by a positive factor and the conjugation by $N$ doesn't affect $n_1$ or $n_2$. So $n_1$ and $n_2$ are each representative of one nilpotent orbit of degree 1. 

The dimensions are given by Corollary 6.1.4 in \cite{CollinMcGov}. Recall that the partition of the complex nilpotent orbit coming from a real nilpotent orbit is given by the boxes of the corresponding Young diagram. 
\end{proof}

We are now able to prove the results about the wave front set of $\E_k$ in Theorem \ref{Thmintro2} for $\SU(n,1)$. 

\begin{proof}[Wave front set of $\E_k$]
The two lemmas above conclude the cases $\circled{1}$ and $\circled{4}$. Now we have to know which the cases $\circled{2}$ and $\circled{3}$ correspond to the nilpotent orbit generated by $n_1$ or the one generated by $n_2$. They have the same dimension, so we need to have more information about the wave front set. We will use the projection over $K$ to figure it out.
In fact, \cite[Proposition 2.3]{HoweWFS} ensures us that \begin{equation}
    \WF(\E_k| _K) = \Ad^*(K) \big(-\mathop{AC}(\supp \E_k)\big)
\end{equation}
where $\mathop{AC}$ the asymptotic cone and $\supp \rho$ is the set of highest weights of $\E_k$. Proposition 2.5 gives us:
\begin{equation}
    \WF(\E_k| _K) = q(\WF \E_k)
\end{equation}
where $q$ is the projection from $\g^*$ onto $\Ak^*$.
We write $\E_k^2$ or $\E_k^3$ if $\E_k$ is respectively in case $\circled{2}$ or in case $\circled{3}$ (see figure \ref{cases for E_k - Complex case}).
We have:
\begin{equation*}
    -\mathop{AC}(\supp \E_k^2) \ni -\mu_{1,1} = \epsilon_n - \epsilon_{n+1} \text{ and } -\mathop{AC}(\supp \E_k^3) \ni \mu_{1,-1} = -\epsilon_1 + \epsilon_{n+1}
\end{equation*}
Moreover,
\begin{align*}
    q\circ B :~~ & n_1 \mapsto (2n+2) (-\epsilon_1 +\epsilon_{n+1})\\
               & n_2 \mapsto (2n+2) (\epsilon_1 -\epsilon_{n+1})
\end{align*}
This gives the result, as $K$ sends by the coadjoint action $\epsilon_n - \epsilon_{n+1}$ to $\epsilon_1 - \epsilon_{n+1}$.
\end{proof}

\begin{minipage}{0.5\textwidth}
\subsection{Case of $\Sp(n,1)$}
In this section, we prove the results about wave front set in Theorem \ref{Thmintro2} for $G = \Sp(n,1)$. First of all, we compute the Gelfand-Kirillov dimension of the Harish-Chandra module of $\E_k$ defined in section \ref{section residue repr}. Recall that there are 3 possibilities of residue representation (see \ref{section E_k study - Qcase}). We label each case as in Figure \ref{cases for E_k - Quater case}. The proofs are very similar to the complex case. So we will be less specific. 
\end{minipage}
\begin{minipage}{0.5\textwidth}
\begin{center}
\begin{tikzpicture}[scale=0.8]
\draw [-,thick] (0,0) to (5,5);
\draw [-,thick] (0,0) to (5,0);
\draw [-,thick] (3,0) to (5,2);
\draw [-,thick] (0,2.5) to (2.5,0);
\node[draw,circle] (D)at(5,0.75) {\footnotesize 1};
\node[draw,circle] (D)at(3.5,2) {\footnotesize 2};
\node[draw,circle] (D)at(1.25,0.5) {\footnotesize 3};
\end{tikzpicture}
\captionof{figure}{Cases for $\E_k$ - $\Sp(n,1)$}\label{cases for E_k - Quater case}
\end{center}
\end{minipage}

\begin{Lemma}\label{Gelf.-Kir.dim/quater case}
The Gelfand-Kirillov dimension of the residue representation $\E_k$ is $\left\{\begin{array}{cl}
    4n-1 & \text{in case 1}\\
    2n+1 & \text{in case 2}\\
    0 & \text{in case 3}\\
\end{array}\right.$
\end{Lemma}

\begin{proof}
First of all, one can compute the eigenvalue on each $K$-type $\tau_{m,l}$ in $\Hi_{\lambda_k\alpha}$ for the Casimir operator $\Omega_\Ak$ of $\Ak$ (using for example \cite[Proposition 10.6]{Hall}). Recall from \ref{section E_k study - Qcase} that $\tau_{m,l}$ has highest weight $\mu_{m,l} = \frac{m+l}{2} \epsilon_1 - \frac{m-l}{2} \epsilon_n + l \epsilon_{n+1}$. Then one finds:
\begin{equation}\label{eigenvalue for Casimir - Quater case}
    \tau_{m,l}(\Omega_\Ak) = \frac{1}{2}\Big((m+2n-1)^2 - (2n-1)^2 +3(l+1)^2 -1\Big) \Id
\end{equation}
The dimension of the space of $\tau_{m,l}$ follows for example from the Weyl dimension formula \cite[Theorem 5.84]{Knapp-LieGroups}: \begin{equation*}
    d_{m,l} = \frac{(l+1)^2}{2n-1}\left(\binom{\frac{m-l}{2}+2n-2}{2n-2} \binom{\frac{m+l}{2}+2n-1}{2n-2} - \binom{\frac{m-l}{2}+2n-3}{2n-2} \binom{\frac{m+l}{2}+2n-2}{2n-2} \right)
\end{equation*}
Setting $\delta({m,l,N}) := \binom{\frac{m-l}{2}+N}{2n-2}\binom{\frac{m+l}{2}+N+1}{2n-2}$ we have \begin{equation}
    d_{m,l} = \frac{(l+1)^2}{2n-1}\big(\delta({m,l,2n-2}) - \delta({m,l,2n-1})\big)
\end{equation}

As before, for a fixed $l$, the sum of $d_{m,l}$ between the values $m_{\min{}}$ and $m_{\max{}}$ of $m$ is telescopic:
\begin{equation}
    \sum_{m=m_{\min{}}}^{m_{\max{}}} d_{m,l} = \frac{(l+1)^2}{2n-1}\big(\delta({m_{\max{}},l,n}) - \delta({m_{\min{}},l,n-1})\big)
\end{equation}
The inequality given by the eigenvalue of the Casimir operator \ref{eigenvalue for Casimir - Complex case} limits the area of points of $K$-types by an ellipsis. The goal is to know the behaviour in $t$ of the sum over the $K$-types delimited by this ellipsis when this ellipsis is going to infinity. The method to compute the $N_{\E_k}(t)$ out of \cite[Theorem 1.2]{VoganGelfKirill} is exactly the same as the $\SU(n,1)$ case and left to the reader. 
\end{proof}

The following lemma gives us the different nilpotent orbits in the nilradical of $\g$.

\begin{Lemma}
There are three nilpotent orbits in $\Asp(n,1)$:
\begin{enumerate}
    \item the trivial orbit,
    \item the orbit generated by $\g_{\alpha/2}$, of dimension $8n-2$,
    \item the orbit generated by $\g_{\alpha}$, of dimension $4n+2$.
\end{enumerate}
\end{Lemma}

\begin{proof}
Using Theorem 9.3.5 in \cite{CollinMcGov}, one finds that there are three possible Young diagrams corresponding to the nilpotent orbits for $\Asp(n,1)$:
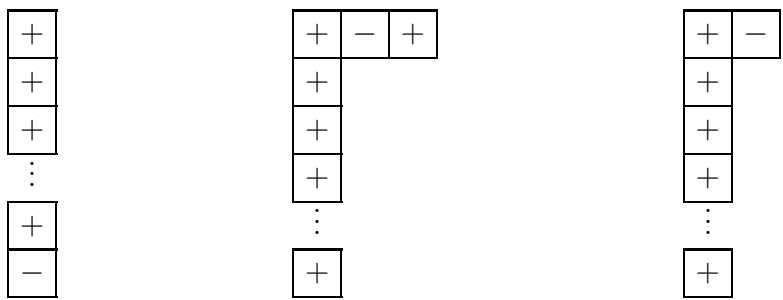
\begin{center}\begin{ytableau}
   + \\
   + \\
   + \\
   \none [\vdots] \\
   + \\
   -
\end{ytableau}\hspace{30mm}
\begin{ytableau}
      +  & -& +\\
   + & \none  & \none \\
   +& \none  & \none \\
   + & \none  & \none \\
   \none [\vdots] & \none  & \none \\
   + & \none  & \none 
\end{ytableau}
\hspace{30mm}
\begin{ytableau}
      +  & -\\
   + & \none   \\
   +& \none   \\
   + & \none   \\
   \none [\vdots] & \none   \\
   + & \none 
\end{ytableau}
\end{center}
The first one corresponds to the zero orbit. The second is the only one which is nilpotent of degree 2. It corresponds to the orbit generated by any element of $\g_{\alpha/2}$. 
The last one is the only one nilpotent of degree 1. So it is generated by any element of $\g_{\alpha}$. 

The dimensions are given by Corollary 6.1.4 in \cite{CollinMcGov}. Also here we have to recall that the partition of the complex nilpotent orbit coming from a real nilpotent orbit is given by the boxes of the corresponding Young diagram. 
\end{proof}

The combination of the two lemmas is sufficient to conclude the results about wave front set of $\E_k$ in Theorem \ref{Thmintro2} for $\Sp(n,1)$.

\begin{minipage}{0.5\textwidth}
\subsection{Case of $\F_4$}
In this section, we prove the results about wave front set in Theorem \ref{Thmintro2} for $G = F_4$. We first compute the Gelfand-Kirillov dimension of the Harish-Chandra module of $\E_k$ defined in section \ref{section residue repr}. Recall that there are three possible residue representations (see \ref{section E_k study - Ocase}) which we label as in Figure \ref{cases for E_k - Octo case}. The proofs are very similar to the complex case. 
\end{minipage}
\begin{minipage}{0.5\textwidth}
\begin{center}
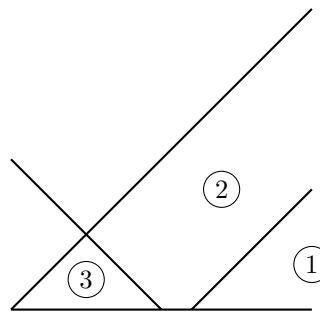

\begin{tikzpicture}[scale=0.8]
\draw [-,thick] (0,0) to (5,5);
\draw [-,thick] (0,0) to (5,0);
\draw [-,thick] (3,0) to (5,2);
\draw [-,thick] (0,2.5) to (2.5,0);
\node[draw,circle] (D)at(5,0.75) {\footnotesize 1};
\node[draw,circle] (D)at(3.5,2) {\footnotesize 2};
\node[draw,circle] (D)at(1.25,0.5) {\footnotesize 3};
\end{tikzpicture}
\captionof{figure}{Cases for $\E_k$ - $\F_4$}\label{cases for E_k - Octo case}
\end{center}
\end{minipage}

\begin{Lemma}\label{Gelf.-Kir.dim/Octo case}
The Gelfand-Kirillov dimension of the residue representation $\E_k$ is $\left\{\begin{array}{cl}
    15 & \text{in case 1}\\
    11 & \text{in case 2}\\
    0 & \text{in case 3}\\
\end{array}\right.$ .
\end{Lemma}

\begin{proof}
As in the complex case, one can compute the eigenvalue on each $K$-type $\tau_{p,q}$ in $\Hi_{\lambda_k\alpha}$ for the Casimir operator $\Omega_\Ak$ of $\Ak$ (using for example \cite[Proposition 10.6]{Hall}). Recall from \ref{section E_k study - Ocase} that $\tau_{p,q}$ has highest weight $\mu_{p,q} = \frac{p}{2} \epsilon_1 + \frac{q}{2} (\epsilon_2 + \epsilon_3 +\epsilon_4)$. Then one finds:
\begin{equation}\label{eigenvalue for Casimir - Octo case}
    \tau_{p,q}(\Omega_\Ak) = \left(\left(\frac{p+7}{2}\right)^2+ \frac{3(q+3)^2}{4} -19 \right) \Id
\end{equation}
The dimension of the space of $\tau_{p,q}$ follows from the Weyl dimension formula \cite[Theorem 5.84]{Knapp-LieGroups}: \begin{equation*}
    d_{p,q} = C(q) (\delta(p+2,q) - \delta(p,q))
\end{equation*}
where $\delta(p,q) := \frac{\left(\frac{p+q}{2}+6\right)!}{\left(\frac{p+q}{2}+2\right)!}\frac{\left(\frac{p-q}{2}+3\right)!}{\left(\frac{p-q}{2})-1\right)!}$ and $C(q) := \frac{(q+1)(q+3)(q+5)(2q+8)(2q+6)(2q+4)}{2^{12} ~ 3^{4} ~ 5^2 ~ 7}$. 
For a fixed $q$, the sum of $d_{p,q}$ between 2 values $p_{\min{}}$ and $p_{\max{}}$ of $p$ is telescopic:
\begin{equation}
    \sum_{p=p_{\min{}}}^{p_{\max{}}} d_{p,q} = C(q)\big(\delta({p_{\max{}}+2,q}) - \delta({p_{\min{}},q})\big)
\end{equation}
The inequality given by the eigenvalue of the Casimir operator \eqref{eigenvalue for Casimir - Octo case} bounds a region of $K$-types by an ellipsis. As in the previous case, we need to know the behaviour in $t$ of the sum limited by this ellipsis when the ellipsis goes to infinity. The method to compute the $N_{\E_k}(t)$ of the \cite[Theorem 1.2]{VoganGelfKirill} is exactly the same as the $\SU(n,1)$ case and is left to the reader.
\end{proof}

The following lemma gives us the different nilpotent orbits in the nilradical of $\g$.

\begin{Lemma}
There are three nilpotent orbits in $\g$:
\begin{enumerate}
    \item the trivial orbit,
    \item the orbit generated by $\g_{\alpha/2}$, of dimension $30$,
    \item the orbit generated by $\g_{\alpha}$, of dimension $22$.
\end{enumerate}
\end{Lemma}

\begin{proof}
The table on page 151 in \cite{CollinMcGov} ensures us that there are three possible Dynkin diagrams corresponding to the nilpotent orbits. Relating these diagrams to the table on page 128 allows us to find the dimensions associated to these (real) nilpotent orbit. 
Now we want to relate the orbits generated respectively by $\g_\alpha$ and $\g_{\alpha/2}$ and the two nilpotent orbits of $\g$. First choose an element $X_{\alpha/2}$ in $\g_{\alpha/2}$ such that $B\big(X_{\alpha/2},\theta(X_{\alpha/2})\big) \ne 0$. Lemma 3.2 in \cite[Chapter IX]{HelgDGLGSS} proves that for all $X_\alpha\in \g_\alpha$, \begin{equation*}
    \big[X_{\alpha/2},[\theta(X_{\alpha/2}),X_\alpha]\big] = 2\<\alpha, \alpha\>~ B\big(X_{\alpha/2},\theta(X_{\alpha/2})\big) ~X_\alpha~.
\end{equation*}
This proves that \begin{equation*}
    \g_\alpha = [\g_{\alpha/2},\g_{\alpha/2}]~.
\end{equation*}
The Jacobi identity allows us to conclude that $Z(\g_{\alpha/2}) \subset Z(\g_\alpha)$. This proves that the (nilpotent) orbit generated by $\g_\alpha$ has lower dimension than that of $\g_{\alpha/2}$. 
\end{proof}

The combination of the two lemmas above proves the results for $F_4$ about wave front set of $\E_k$ in Theorem \ref{Thmintro2}. 

\appendix

\section{The Plancherel densities}
\label{Plancherel densities section}

Recall that to determine the resolvent of the Laplacian, we use the inversion formula \eqref{plancherelformula} for vector-valued Helgason-Fourier transform. The resonances arise from the singularities of the Plancherel measure. To find them, we need an explicit formula for the Plancherel density $p_\sigma$. 
Even if such an explicit formula is not known for an arbitrary group $G$, in the rank-one case, several authors have computed it (\cite{Oka}, \cite{Mia}; see also \cite[Epilogue, pp. 414 ff.]{Warner2}). 
Our reference in the following is Miatello's article \cite{Mia}. 
The formula depends on the group $G$ and on the highest weight of $\sigma$. 
According to this formula, $p_\sigma$ is the product of two factors. 
The first one is a polynomial function in $\lambda$, denoted $q_\sigma$. 
The second factor is either a hyperbolic tangent or a hyperbolic cotangent (or 1 if $G = \Spin(2n+1,1)$). 
We denote it by $\phi_\sigma$. The goal of this section is to explain how one can compute $p_\sigma$ with $\sigma\in \hat{M}(\tau)$, where $\tau\in \hat{K}$ is arbitrarily fixed. 
For this, we are going to use the branching rules given by Baldoni Silva in \cite{Bald1}. 
Remark 1.3 in \cite{Mia} ensures us that the Plancherel formula of one of the four groups listed in the table in the introduction gives us the Plancherel formula for every group of real rank one. As we do not care about the constants in our computations in this paper, $p_\sigma$ is given up to a constant.

\subsection{Case of $\Spin(n,1)$}
\label{Section Plancherel density Rcase}
Here $G/K$ is the real hyperbolic space. Recall that $K$ is $\Spin(n)$ and $M$ is $\Spin(n-1)$. In this case, the parity of $n$ plays a role in the Plancherel measure. In fact, when $n$ is odd, $\phi_\sigma(\lambda) = 1$. So it is non-singular and there are no resonances.
In the following we therefore disregard the case of 
$\Spin(2n+1,1)$ and suppose that $G = \Spin(2n,1)$. 

Let us recall some Lie algebraic structure. Our reference is \cite[Section 3]{Bald1}.
The Lie algebra of $G$ is $\g = \soe(2n,1)$. Its complexification is $\g_\C = \soe(2n+1,\C)$. We have also $Lie(K)_\C = \Ak_\C = \soe(2n,\C)$ and $Lie(M)_\C = \m_\C = \soe(2n-1,\C)$.

We choose the following Cartan subalgebra of $\g_\C$ is:
\begin{equation*}
    \h_\C = \left\{H \in \soe(2n+1, \C) ~|~ H = \mathsf{diag}\left[\begin{pmatrix}
0 & ih_1 \\
-ih_1 & 0 \\
\end{pmatrix},\begin{pmatrix}
0 & ih_2 \\
-ih_2 & 0 \\
\end{pmatrix},\ldots,\begin{pmatrix}
0 & ih_n \\
-ih_n & 0 \\
\end{pmatrix},0\right]
    \right\}
\end{equation*}

Let $\{\epsilon_j\}_{j=1,\ldots,n}$ be the elementary weights defined by $\epsilon_j(H) = h_j$. We denote $S$, $S^+$ and $S^0$ respectively the set of roots, positive roots and simple roots of $(\g_\C,\h_\C)$. For the set of roots of $\Ak_\C$ or $\m_\C$ we let the Lie algebra in index. 
\begin{equation*}
    S = \big\{\pm \epsilon_i \pm \epsilon_j ~|~ i\not = j \big\} \cup \big\{\pm \epsilon_k ~|~ k \in \{1,\ldots,n\} \big\} ~;~ S^0 = \big\{ \alpha_i = \epsilon_i - \epsilon_{i+1} ~|~ i\in \{1,\ldots,n-1\} \big\} \cup \big\{\alpha_n = \epsilon_n \big\}
\end{equation*}
\begin{align*}\label{realroots}
    S_{\Ak_\C} = \big\{\pm \epsilon_i \pm \epsilon_j ~|~ i\not = j \big\}~~;~~ & S_{\m_\C} = \big\{\pm \epsilon_i \pm \epsilon_j ~|~ i\not = j \big\} \cup \big\{\pm \epsilon_k ~|~ k \in \{1,\ldots,n-1\} \big\}\\
    S^0_{\Ak_\C} = \big\{ \epsilon_i - \epsilon_{i+1} ~|~ i\in \{1,\ldots,n-1\} \big\} ~~;~~ & S^0_{\m_\C} = \big\{ \epsilon_i - \epsilon_{i+1} ~|~ i\in \{1,\ldots,n-2\} \big\} \cup \big\{\epsilon_{n-1} \big\}
\end{align*}
Notice that, unlike \cite[Theorem 3.4]{Bald1}, we keep the same notation for the $\epsilon_j$'s and their projections on $\m_\C$.
As in \cite[Lemma 3.2]{Bald1}, the fixed $(\tau, \Hi_\tau) \in \hat{K}$ has highest weight of the form:
\begin{equation*}
    \mu_\tau = \sum_{j=1}^{n} a_j \epsilon_j
\end{equation*}
where $a_1\geq\ldots \geq a_{n-1}\geq ~|~a_n~|~\geq0$, $a_i - a_j \in \Z$ and $2 a_j \in \Z$ for all $i,j = 1,\ldots, n$.

Let $\sigma \in \hat{M}(\tau)$. According to \cite[Theorem 3.4]{Bald1}, the highest weight $\mu_\sigma$ of a representation $\sigma \in \hat{M}(\tau)$ has the form: 
\begin{equation*}
    \mu_\sigma = \sum_{i=1}^{n-1} b_j \epsilon_j
\end{equation*}
where for all $i, j=1,\ldots,n-1$ we have  $a_j - b_i \in \Z$ and $a_1\geq b_{1}\geq a_{2}\geq\ldots \geq a_{n-1}\geq b_{n-1}\geq ~|~a_n~|~\geq0$. 
Furthermore, $m(\sigma,\tau|_M)=1$ for every $\sigma \in \hat{M}(\tau)$.

For $\sigma \in \hat{M}(\tau)$, the Plancherel density is given in \cite[pp. 256-257]{Mia}: 
\begin{equation}\label{eq Plancherel density Rcase}
     p_\sigma(\lambda) = \left\{\begin{array}{c}
         \tanh(\pi\lambda_\alpha) \text{ , if } b_j \in \Z\\
          \coth(\pi\lambda_\alpha)\text{ , if } b_j \in \frac{1}{2} +\Z
     \end{array}\right\} \lambda_\alpha \prod_{j=1}^{n-1} \left(\lambda_\alpha^2 +(\rho +b_j -j)^2 \right)
\end{equation}

We recall that $\lambda_\alpha$ is the complex number associated to $\lambda\in \Aa_\C^\ast$ by \eqref{lambda}.

\begin{Rem}
The above formula agrees with the Plancherel density for the real hyperbolic space for $\tau = \triv_K$ (and the, $\sigma = \triv_M$). See e.g. \cite{HilgPasq}. It also gives the different irreducible continuous constituents of the Plancherel density for the $p$-forms on the real hyperbolic space, as computed in \cite{PedTh}.
\end{Rem}

\subsection{Case of $\SU(n,1)$}

\label{Section Plancherel density Ccase}
Here $G/K$ is the complex hyperbolic space. Recall that $K$  is $\Super(\U(n)\times\U(1))$ and $M$ is $\Super(\U(1)\times\U(n-1)\times\U(1))$. Our notations follows \cite[Section 4]{Bald1}. 
The Lie algebra of $G$ is $\g = \su(n,1)$. Its complexification is $\g_\C = \sL(n+1,\C)$. We have also $Lie(K)_\C = \Ak_\C = \sL(n,\C)$ and $Lie(M)_\C = \m_\C = \sL(n-1,\C)$. 

The elliptic Cartan subalgebra $\h_\C$ of $\g_\C$ consists of the diagonal matrices in $\sL(n+1,\C)$. 

Let $\{\epsilon_j\}_{j=1,\ldots,n}$ be the elementary weights defined by $\epsilon_j(H) = h_j$ where $H = \mathsf{diag}(h_1,h_2, \ldots,h_{n+1})$ described the elements of $\h_\C$.
We denote $S$, $S^+$ and $S^0$ respectively the set of roots, positive roots and simple roots of $(\g_\C,\h_\C)$. For the set of roots of $\Ak_\C$ or $\m_\C$ we let the Lie algebra in index.
\begin{equation*}
    S = \big\{\pm (\epsilon_i - \epsilon_j) ~|~ 1\leq i < j \leq n + 1\big\} ~;~
    S^0 = \big\{\alpha _i = \epsilon_i - \epsilon_{i+1} ~|~ 1\leq i \leq n\big\} 
\end{equation*}
\begin{align*}
    S_{\Ak_\C} = \big\{\pm (\epsilon_i - \epsilon_j) ~|~ 1\leq i < j \leq n \big\}  &~~~~ S_{\m_\C} = \big\{\pm (\epsilon_i - \epsilon_j) ~|~ 2\leq i < j \leq n\big\} \\
    S^0_{\Ak_\C} = \big\{\alpha _i = \epsilon_i - \epsilon_{i+1} ~|~ 1\leq i \leq n-1\big\}  &~~~~ S^0_{\m_\C} = \big\{\alpha _i = \epsilon_i - \epsilon_{i+1} ~|~ 2\leq i \leq n-1\big\}
\end{align*}
    
As in \cite[Lemma 4.2]{Bald1}, the fixed $(\tau, \Hi_\tau) \in \hat{K}$ has highest weight of the form
\begin{equation*}
    \mu_\tau = \sum_{j=1}^{n+1} a_j \epsilon_j
\end{equation*}
where $a_1\geq\ldots \geq a_{n-1}\geq a_n$ and  $a_i\in \Z$ for all $i = 1,\ldots, n+1$.

Let $\sigma \in \hat{M}(\tau)$. By \cite[Theorem 4.4]{Bald1}. The highest weight $\mu_\sigma$ of $\sigma \in \hat{M}(\tau)$
has the form
\begin{equation*}
    \mu_\sigma = b_0 (\epsilon_1 + \epsilon_{n+1}) + \sum_{j=2}^{n} b_j \epsilon_j
\end{equation*}
where for all $j$, we have $b_j \in \Z$, $a_1\geq b_2 \geq a_2 \geq \ldots \geq a_{n-1}\geq b_n \geq a_n$ and $b_0 = \frac{\sum_{j=1}^{n+1} a_j - \sum_{j=2}^n b_j}{2}$. 


The fact that we are working with zero trace matrices implies that $\sum_{j=1}^{n+1} \epsilon_j = 0$. Thus we can relate to Miatello's form for the weights. 
We get the same formula as in \cite[p. 258]{Mia}: 
\begin{equation*}
    \mu_\sigma = \sum_{j=2}^{n-1} (b_j - b_n) \epsilon_j + (b_n - b_0) \sum_{j=2}^{n} \epsilon _j
\end{equation*}

Miatello gives this following Plancherel density: 
\begin{equation}\label{eq Plancherel density Ccase}
     p_\sigma(\lambda) = \left\{\begin{array}{c}
         \tanh(\pi\lambda_\alpha) \text{ , if } 2b_0 + n \text{ is odd}\\
          \coth(\pi\lambda_\alpha)\text{ , if } 2b_0 +n \text{ is even}
     \end{array}\right\} \lambda_\alpha \prod_{j=1}^{n-1} \left(\lambda_\alpha^2 +(b_{j+1} - b_0 + \frac{n}{2} -j)^2 \right)
\end{equation}

We recall that $\lambda_\alpha$ is the complex number associated to $\lambda\in \Aa_\C^\ast$ by (\ref{lambda}).

\subsection{Case of $\Sp(n,1)$}

\label{Section Plancherel density Qcase}

Here $G/K$ is the quaternionic hyperbolic space. In this case, $K$ is $\Sp(n)\times \Sp(1)$ and $M$ is $\Sp(1)\times\Sp(n-1)\times \Sp(1)$.

Let us recall some Lie algebraic structure. 
The Lie algebra of $G$ is $\g = \Asp(n,1)$. Its complexification is $\g_\C = \Asp(n+1,\C)$. We have also $Lie(K)_\C = \Ak_\C = \Asp(n,\C)$ and $Lie(M)_\C = \m_\C = \Asp(n-1,\C)$. 

Let $\Algt$ denote the set of diagonal matrices in $\g$. So $\Algt_\C$ is the elliptic Cartan subalgebra of $\g_\C$. Let $\h^-$ be a maximal abelian algebra of $\m$. Then $\h:= \h^- + \Aa$ is a Cartan subalgebra of $\g$, and $\h_\C$ is a Cartan subalgebra of $\g_\C$.

Let $\{\epsilon_j\}_{j=1,\ldots,n}$ be the elementary weights defined by $\epsilon_j(H) = h_j$, where $H = \mathsf{diag}(h_1, h_2, \ldots, h_{n+1}) \in \Algt_\C$. We denote $S$, $S^+$ and $S^0$ respectively the set of roots, positive roots and simple roots of $(\g_\C,\h_\C)$. For the set of roots of $\Ak_\C$ or $\m_\C$ we let the Lie algebra in index.
 $$\begin{array}{c}
     S = \big\{\pm \epsilon_i \pm \epsilon_j ~|~ 1\leq i < j \leq n + 1\big\} \cup \big\{\pm 2\epsilon_i ~|~ 1\leq i \leq n + 1\big\} ~;~\\ S^0 = \big\{ \epsilon_i - \epsilon_{i+1} ~|~ 1\leq i \leq n\big\} \cup \big\{2\epsilon_{n+1}\big\} ~;~\\
     S_{\Ak_\C} = \big\{\pm \epsilon_i \pm \epsilon_j ~|~ 1\leq i < j \leq n \big\} \cup \big\{\pm 2\epsilon_i ~|~ 1\leq i \leq n\big\} ~;~\\ S^0_{\Ak_\C} = \big\{ \epsilon_i - \epsilon_{i+1} ~|~ 1\leq i \leq n-1 \big\} \cup \big\{2\epsilon_{n}\big\} ~;~\\S_{\m_\C} = \big\{\pm (\epsilon_i - \epsilon_j) ~|~ 1\leq i < j \leq n-1\big\}\cup \big\{\pm 2\epsilon_i ~|~ 1\leq i \leq n-1\big\} ~;~\\ S^0_{\m_\C} = \big\{ \epsilon_i - \epsilon_{i+1} ~|~ 1\leq i \leq n-2 \big\} \cup \big\{2\epsilon_{n-1}\big\}
 \end{array}$$
The fixed $(\tau, \Hi_\tau)\in \hat{K}$ has highest weight of the form:
\begin{equation*}
    \mu_\tau = \sum_{j=1}^{n+1} a_j \epsilon_j
\end{equation*}
where $a_1\geq\ldots \geq a_{n-1}\geq a_n\geq0$, $a_{n+1}\geq 0$ and  $a_i\in \Z$ for all $i = 1,\ldots, n+1$.

See \cite[Lemma 5.2]{Bald1}. Theorem 5.5 in \cite{Bald1} gives us the form of the highest weight $\mu_\sigma$ of $\sigma \in \hat{M}(\tau)$ as follows: 
\begin{equation}
    \mu_\sigma = b_0 (\epsilon_1 + \epsilon_{n+1}) + \sum_{j=2}^{n} b_j \epsilon_j
\end{equation}
where we have:
\begin{enumerate}
    \item $a_j\geq b_{j+1}$, for all $j =1, \ldots, n-1$
    \item $b_j\geq a_{j+1}$, for all $j =2, \ldots, n-1$
    \item $b_0 = \frac{a_{n+1}+b_1-2j}{2}$, for some $j =0, \ldots, \min(a_{n+1},b_1)$
\end{enumerate}  
and $b_1$ satisfies $b_1\in \Z_+$ and $\sum_{j=1}^{n} (a_j -  b_j) \in 2\Z$.


We want to use the Plancherel formula given by Miatello \cite{Mia}. But the Cartan subalgebra $\h_\C$ used in this paper is not the same as the Cartan algebra $\At_\C$ used by Baldoni-Silva for the branching rules. So we have to use a Cayley transform to find the highest weight of $\sigma$ relative to $\h_\C$. This computation was already done in \cite{BaldKral}. Namely (see \cite[p155-156]{HelgDGLGSS}) one know that for each $\beta \in S$ we can select a root vector $X_\beta$ such that $B(X_a, X_{-a}) = \frac{2}{\<\beta, \beta\>}$ and $\theta \overline{X_\beta} = -X_{\beta}$. Choose $\beta = \epsilon_1 - \epsilon_{n+1}$. Then if we denote by $e_1, \ldots, e_{n+1}$ the fundamental weights of the root system $\Delta(\g_\C, \h_\C)$, we have : 
$$e_1 = \epsilon_1 \circ \Ad u_\beta^{-1}~~ , ~~ e_2 = - \epsilon_{n+1} \circ \Ad u_\beta^{-1}~~ , ~~ e_i = - \epsilon_{i-1} \circ \Ad u_\beta^{-1} \text{ , for all } i = 3, \ldots n+1$$
where $u_\beta = \exp(\pi/4)(X_\beta - X_{-\beta})$.

So $\mu_\sigma$ is written in the root system $\Delta(\g_\C, \h_\C)$ as: 

\begin{equation*}
    \mu_\sigma = b_0 (e_1 - e_2) + \sum_{j=3}^{n+1} b_{j-1} e_j
\end{equation*}

and \cite{Mia} gives the following Plancherel density:

\begin{multline}\label{eq Plancherel density Qcase}
    p_\sigma(\lambda) = \left\{\begin{array}{c}
         \tanh(\pi\lambda_\alpha) \text{ , if } b_0 \in \Z\\
          \coth(\pi\lambda_\alpha)\text{ , if } b_0 \in \Z + \frac{1}{2}
     \end{array}\right\} \lambda_\alpha \left(\lambda_\alpha^2 +(b_0 + \frac{1}{2})^2\right) \\
     \prod_{j=3}^{n+1} \left(\lambda_\alpha^2 +(b_{j-1} - b_0 + n - j +\frac{3}{2})^2\right) \left(\lambda_\alpha^2 +(b_{j-1} - b_0 + n - j +\frac{5}{2})^2 \right)
\end{multline}

We recall that $\lambda_\alpha$ is the complex number associated to $\lambda\in \Aa_\C^\ast$ by (\ref{lambda}).

\subsection{Case of $F_4$} 
\label{Section Plancherel density Ocase}

Here, $G = F_4$. We recall $\g = f_4^{-20}$,  $K$ is $Spin(9)$ and its Lie algebra is $\Ak = \soe(9)$. 
Let $\At \subset \Ak$ be the compact Cartan subalgebra for both $\g$ and $\Ak$. The branching rules for the fixed $\tau$ are determined in \cite[Paragraph 6]{Bald1}. For a maximal abelian subspace $\Aa$ of $\p$ its centralizer in $\Ak$ is $\m = \soe(7)$. The problem is that this Lie algebra $\m$ is not contained in the standard way in $\Ak$. So, we cannot use twice the branching rules for $\SO(n)$ directly. Let $K_1 = \Spin(8)$ a subgroup of $K$  contained in the standard way. We denote by $\Ak_1$ its Lie algebra. 

Let $\{\epsilon_j\}_{j=1,\ldots,n}$ be the elementary weights in $\Delta (\g_\C, \At_\C)$. We denote $S$, $S^+$ and $S^0$ respectively the set of roots, positive roots and simple roots of $(\g_\C,\h_\C)$. For the set of roots of $\Ak_\C$ or $\m_\C$ we let the Lie algebra in index. \begin{center}
    $S = \big\{\pm \epsilon_i \pm \epsilon_j ~|~ 1\leq i < j \leq 4\big\} \cup \big\{\pm \epsilon_i ~|~ 1\leq i \leq 4 \big\} \cup \big\{\frac{1}{2}(\pm \epsilon_1 \pm \epsilon_2\pm \epsilon_3\pm \epsilon_4) \big\}$ \newline $S^0 = \big\{ \alpha_1 = \epsilon_2 - \epsilon_3 ~,~\alpha_2 = \epsilon_3 - \epsilon_4 ~,~\alpha_3 = \epsilon_4  ~,~\alpha_4 = \frac{1}{2}( \epsilon_1 - \epsilon_2- \epsilon_3- \epsilon_4) \big\}$
    \newline $S_{\Ak_\C} = \big\{\pm \epsilon_i \pm \epsilon_j ~|~ 1\leq i < j \leq 4\big\} \cup \big\{\pm \epsilon_i ~|~ 1\leq i \leq4\big\}  ~~~~ S_{(\Ak_1)_\C} = \big\{\pm (\epsilon_i \pm \epsilon_j) ~|~ 1\leq i < j \leq 4\big\}$
    \newline $S^0_{\Ak_\C} = \big\{ \alpha_1  ~,~\alpha_2 ~,~\alpha_3 ~,~\alpha_2 + 2 \alpha_3 +2 \alpha_4 = \epsilon_1 - \epsilon_2 \big\}  ~~~~ S^0_{(\Ak_1)_\C} = \big\{ \epsilon_1-\epsilon_2~,~\epsilon_2-\epsilon_3, \epsilon_3-\epsilon_4, \epsilon_3+\epsilon_4 \big\}$
\end{center}

Let $H_{\alpha_4}$ be the unique element of $\Aa$ such that $\alpha_4 = B(H_{\alpha_4}, ~\cdot~)$. Choose the root vectors $X_{\alpha_4}$ and $X_{-\alpha_4}$ such that $[X_{\alpha_4}, X_{-\alpha_4}] = H_{\alpha_4}$ and $X_{\alpha_4}+ X_{-\alpha_4} \in \p$. Define $\Aa$ to be the one-dimensional space spanned by $X_{\alpha_4}+ X_{-\alpha_4}$. Let $\h = \h^-\oplus \Aa$ where $\h^-$ is a Cartan subalgebra of $\m$. Then the Cayley transform $\Ad\left(\exp\frac{\pi}{4}(X_{\alpha_4}- X_{-\alpha_4}\right))$ maps $\At_\C$ onto $\h_\C$. The set of roots $\Delta(\g_\C, \h_\C)$ is the following root system on $\m$
\begin{center}
  $\Phi_\m = \big\{\pm (\epsilon_i - \epsilon_j) ~|~ 1\leq i < j \leq n-1\big\}\cup \big\{\pm 2\epsilon_i ~|~ 1\leq i \leq n-1\big\}$
\newline $\Phi_\m^0 = \big\{ \alpha_1  ~,~\alpha_2 ~,~\alpha_2 + 2 \alpha_3 + \alpha_4 = \frac{1}{2}( \epsilon_1 - \epsilon_2+ \epsilon_3+ \epsilon_4) \big\}$
\end{center}

As said before, $\m$ is not contained in the standard way in $\Ak_1$. Let $\phi$ be the automorphism of $\Ak_1$ which keeps the roots and such that $\phi(M)$ is contained in the standard way in $\phi(K_1)$. This automorphism is given by 
\begin{equation*}
    \begin{array}{cc}
     \phi(\epsilon_1 - \epsilon_2) = \epsilon_3 - \epsilon_4 ~ , ~ & \phi(\epsilon_3 - \epsilon_4) = \epsilon_1 - \epsilon_2 \\
     \phi(\epsilon_2 - \epsilon_3) = \epsilon_2 - \epsilon_3 ~ , ~ & \phi(\epsilon_3 + \epsilon_4) = \epsilon_3 + \epsilon_4
\end{array}
\end{equation*}

The fixed $K$-type $\tau$ has highest weight of the form: \begin{equation*}
    \mu_\tau = a_1 \epsilon_1 + a_2 \epsilon_2 + a_3 \epsilon_3 + a_4 \epsilon_4
\end{equation*}
where $a_1\geq\ldots \geq a_4\geq0$, $2 a_i\in \Z$ and $a_i-a_j \in \Z$ for all $i,j = 1,\ldots, 4$.

Because of the branching rules of $\Spin(n)$, the representations of $K_1$ which are contained in the restriction of $\tau$ to $K_1$ have the following highest weights:
\begin{equation*}
    c_1 \epsilon_1 + c_2 \epsilon_2 + c_3 \epsilon_3 + c_4 \epsilon_4
\end{equation*}
where $a_1\geq c_1 \ldots \geq a_4\geq~|~c_4~|~$, $2 c_i\in \Z$ and $a_i-c_j \in \Z$ for all $i,j = 1,\ldots, 4$.\newline
We have now to apply $\phi$ to this highest weight to use the branching rules of $\soe(n)$. We get a highest weight of the form
\begin{multline*}
    d_1 \epsilon_1 + d_2 \epsilon_2 + d_3 \epsilon_3 + d_4 \epsilon_4 := \frac{1}{2} (c_1+c_2+c_3-c4) \epsilon_1 + \frac{1}{2} (c_1+c_2-c_3+c4) \epsilon_2 \\+\frac{1}{2} (c_1-c_2+c_3+c4) \epsilon_3 +\frac{1}{2} (-c_1+c_2+c_3+c4) \epsilon_4
\end{multline*}
A representation of $\widehat{\phi(M)}$ is of the form $\sigma \circ \phi$ where $\sigma \in \hat{M}$. The $\phi(M)$-types which appear in the restriction of $K_1$-type $(d_1,d_2,d_3,d_4)$ to $\phi(M)$ have highest weight of the form:
\begin{equation*}
    \mu_{\sigma \circ \phi} = b_1 \epsilon_1 + b_2 \epsilon_2 + b_3 \epsilon_3
\end{equation*}
where $d_1\geq b_1 \ldots \geq b_3\geq |d_4|$, $2 c_i\in \Z$ and $a_i-c_j \in \Z$ for all $i,j = 1,\ldots, 4$.
Applying $\phi$, one gets \begin{equation*}
    \mu_\sigma = b_1 \alpha_2 + (b_1+b_2) \alpha_1 + \frac{1}{2}(b_1+b_2+b_3) (\alpha_2 + 2\alpha_3 +\alpha_4)
\end{equation*}

The Cartan subalgebra used in \cite{Mia} is neither $\At$ nor $\h$ because the maximal abelian subalgebra of $\p$ is not the same $\Aa$ here. Define the map
\begin{equation*}
    \begin{array}{cccc}
    \Psi : & \alpha_1 &\mapsto & \alpha_2 \\
           & \alpha_2 &\mapsto & \alpha_1 \\
           & \alpha_2 + 2\alpha_3 +\alpha_4 &\mapsto & \alpha_3 \\
           & \alpha_3 &\mapsto & -(\alpha_2 + 2\alpha_3 +\alpha_4) \\
    \end{array}
\end{equation*} This map $\Psi$ sends $\alpha_4$ on $-\epsilon_1$, so $\Aa$ onto the maximal abelian subspace used in \cite{Mia}. The Cartan subalgebra $\Psi(\h)$ is that used in that paper. Applying $\Psi$ we get
\begin{equation*}
    \mu_{\sigma\circ\Psi} =  b_1 \epsilon_1 + b_2 \epsilon_2 + b_3 \epsilon_3
\end{equation*} and $\epsilon_1$ becomes the real root, i.e. $\epsilon_1 | _{\Psi(\Aa)} = -\alpha\circ\Psi$ where $\alpha$ is the longest restricted root.
Now we can apply the Plancherel formula from \cite{Mia} and get

\begin{multline}\label{eq Plancherel density Ocase}
     p_\sigma(\lambda) = \left\{\begin{array}{c}
         \tanh(\pi\lambda_\alpha) \text{ , if } b_i \in \Z \\
          \coth(\pi\lambda_\alpha)\text{ , if } b_i \in \Z + \frac{1}{2}
     \end{array}\right\} \lambda_\alpha \left(\lambda_\alpha^2 +( \frac{b_3+1}{2})^2\right) ~\left(\lambda_\alpha^2 +( \frac{b_2+3}{2})^2\right) \\~\left(\lambda_\alpha^2 +( \frac{b_1+5}{2})^2\right) ~\left(\lambda_\alpha^2 +( b_1-b_2-b_3+\frac{1}{2})^2\right) \\ ~\left(\lambda_\alpha^2 +( b_1-b_2+b_3+\frac{3}{2})^2\right) ~\left(\lambda_\alpha^2 +( b_1+b_2-b_3+\frac{7}{2})^2\right) ~\left(\lambda_\alpha^2 +( b_1+b_2+b_3+\frac{9}{2})^2\right) 
\end{multline}

We recall that $\lambda_\alpha$ is the complex number associated to $\lambda\in \Aa_\C^\ast$ by (\ref{lambda}).

\bibliographystyle{alpha}
\bibliography{Biblio}

\end{document}